\documentclass[11pt,reqno]{amsart}
\usepackage[utf8]{inputenc}
\usepackage[margin=1.1in]{geometry}
\usepackage{amsmath,amssymb,amsthm,mathtools}
\usepackage{enumitem}
\usepackage{microtype}
\usepackage{hyperref}
\usepackage{tikz}
\usetikzlibrary{arrows.meta,calc,positioning,decorations.pathmorphing}

\hypersetup{colorlinks=true,linkcolor=blue,citecolor=blue,urlcolor=blue}

\numberwithin{equation}{section}

\newtheorem{theorem}{Theorem}[section]
\newtheorem{proposition}[theorem]{Proposition}
\newtheorem{corollary}[theorem]{Corollary}
\newtheorem{lemma}[theorem]{Lemma}
\newtheorem{definition}[theorem]{Definition}
\newtheorem{remark}[theorem]{Remark}
\newtheorem{convention}[theorem]{Convention}
\newtheorem{question}[theorem]{Question}

\newtheorem*{maintheorem}{Main Theorem}

\newcommand{\Z}{\mathbb Z}
\newcommand{\Ftwo}{\mathbb F_2}
\newcommand{\R}{\mathbb R}
\newcommand{\Out}{\operatorname{Out}}
\newcommand{\Diff}{\operatorname{Diff}}
\newcommand{\Mod}{\operatorname{Mod}}

\title[Image Nonconcordance of \(\pi_1\)-Injective Surfaces]
{Image nonconcordance of positive-genus \(\pi_1\)-injective surfaces}
\author{Weizhe Niu}
\address{Yau Mathematical Sciences Center, Tsinghua University}
\email{weizheniu@mail.tsinghua.edu.cn}


\begin{document}
\begin{abstract}
We construct, for every \(g\ge2\), infinite families of homotopic smooth
embeddings of a closed genus-\(g\) surface whose images are pairwise not
smoothly image-concordant, while each surface is \(\pi_1\)-injective.  The
main closed examples lie in one-fold stabilizations of closed aspherical
mapping tori with torsion-free fundamental group: after stabilization by
\(S^2\times S^2\), the surfaces have a common framed dual sphere and the
inclusion of each complement induces a \(\pi_1\)-isomorphism.  The
image-nonconcordance already occurs before stabilization, in the underlying
closed aspherical mapping torus, and persists after every finite number of
\(S^2\times S^2\)-stabilizations. The obstruction is a computable mod-two coordinate of Freedman--Quinn/Dax-type self-intersection data for
concordance tracks, indexed by self-dual double-cosets of a possibly
non-normal surface subgroup \(H\leq\pi_1X\).  The geometric source of the
relevant labels is a M\"obius-band square-root relation: elements
\(t\notin H\) with \(t^2\in H\) produce self-dual labels in torsion-free
ambient groups.  These square roots are realized naturally in
Klein-bottle \(I\)-bundle pieces and persist in closed graph-manifold
mapping-torus examples.
\end{abstract}
\maketitle


\section{Introduction}\label{sec:introduction}

Dual spheres have long made embedded surfaces in smooth four-manifolds
amenable to Whitney-move and light-bulb arguments.  In the sphere case, Freedman--Quinn developed the relevant mod-two
secondary self-intersection obstruction after the primary intersections
have been paired, and Stong corrected the associated uniqueness statement
\cite{FreedmanQuinn1990,Stong1993}; Dax and Hatcher--Quinn developed closely
related self-intersection and embedding-space invariants
\cite{Dax1972,HatcherQuinn1974}.  Gabai's four-dimensional light bulb theorem
shows that homotopic embedded spheres with a common transverse sphere are
isotopic when the ambient fundamental group has no elements of order two
\cite{Gabai2020}, while Schneiderman--Teichner identify the corresponding
Freedman--Quinn obstruction framework for spheres with duals
\cite{SchneidermanTeichner2022}.

For positive-genus surfaces, Gabai also proved a light-bulb theorem under a
\(G\)-inessentiality hypothesis \cite{Gabai2020}.  Maggie Miller established
a concordance analogue for positive-genus surfaces under corresponding
fundamental-group triviality hypotheses \cite{Miller2021}, and
Klug--Miller classified concordance for homotopic orientable
\(\pi_1\)-negligible surfaces with a framed immersed dual in terms of the
Freedman--Quinn obstruction \cite{KlugMiller2021}.  Sunukjian's theorem gives
broad concordance flexibility for same-genus homologous surfaces in simply
connected four-manifolds \cite{Sunukjian2015}.  These results place positive-genus concordance in the
\(\pi_1\)-negligible setting within the Freedman--Quinn framework.
Closed positive-genus \(\pi_1\)-injective surfaces have also been studied
using Dax methods \cite{LinWuXieZhang}.
The present work treats concordance in this \(\pi_1\)-injective regime.
In the principal closed examples, the homomorphisms induced by the embeddings have a common image subgroup
\[
H\leq \pi_1X,
\]
which is proper and nonnormal.  The natural labels therefore lie in
\(H\backslash\pi_1X/H\), and the obstruction rules out concordance of
embedded images.  After one stabilization, the inclusion of the
complement of each surface induces an isomorphism on fundamental groups.

The equivalence relation is image concordance.  Thus two embeddings
\(F_0,F_1:\Sigma_g\hookrightarrow X\) are smoothly image-concordant if
\(F_0\) is smoothly parametrized concordant to \(F_1\circ\phi\) for some
diffeomorphism \(\phi:\Sigma_g\to\Sigma_g\).  The endpoint
reparametrization \(\phi\) is arbitrary: image nonconcordance means that no
such concordance exists for any \(\phi\).  This forgets the chosen
parametrization of the image at the top end, and is therefore coarser than
parametrized concordance.  The main question is the following.

\begin{question}\label{q:intro-main-question}
Can homotopic positive-genus \(\pi_1\)-injective surfaces with a common
framed dual sphere and complement \(\pi_1\)-isomorphisms fail to be smoothly
image-concordant, even when the ambient fundamental group is torsion-free?
\end{question}

We answer this question first in closed aspherical mapping tori.  For every
\(g\geq2\), we construct a closed orientable aspherical graph manifold
\(N_g^K\), a Dehn twist \(\tau:N_g^K\to N_g^K\) along a nonseparating JSJ
torus, and the closed aspherical mapping torus
\[
   Y_\tau=N_g^K\rtimes_\tau S^1 .
\]
The group \(\pi_1Y_\tau\) is torsion-free, and \(Y_\tau\) contains infinitely
many homotopic \(\pi_1\)-injective genus-\(g\) surfaces whose common image
subgroup is proper and nonnormal and whose embedded images are pairwise not
smoothly image-concordant.  One \(S^2\times S^2\)-stabilization
supplies a common framed dual sphere and complement \(\pi_1\)-isomorphisms,
without removing the image nonconcordance; the common-dual, complement, and
image-nonconcordance conclusions persist under every finite further stabilization.

For the notation below, fix basepoints \(z_0\in\Sigma_g\) and
\(x_0\in X\), write \(\pi=\pi_1(X,x_0)\), and let \(H\leq\pi\).
A \(\rho\)-marked embedding \(F:\Sigma_g\hookrightarrow X\) is equipped
with an isomorphism
\[
   \rho:\pi_1(\Sigma_g,z_0)\xrightarrow{\cong}H
\]
and a basing path from \(x_0\) to \(F(z_0)\), chosen so that the induced
based homomorphism to \(\pi\) is the inclusion of \(H\) composed with
\(\rho\).  Thus \(H\) is fixed as an actual subgroup of \(\pi\), not
merely up to conjugacy.  The formal definitions of marked surfaces and
marked tracks are given in Definitions~\ref{def:marked-surface}
and~\ref{def:tracks-concordance}.

The obstruction groups \(Q^D_\rho(X,F_0)\) and classes
\(FQ^D_\rho(F_0,F_1)\) are defined in
Section~\ref{sec:absolute}. 
\begin{maintheorem}
\label{thm:intro-closed-graph-manifold}
For every \(g\geq2\), there exist a closed orientable aspherical graph
manifold \(N_g^K\), a Dehn twist \(\tau:N_g^K\to N_g^K\) along a
nonseparating JSJ torus, and a closed aspherical mapping torus
\[
   Y_\tau=N_g^K\rtimes_\tau S^1
\]
with torsion-free fundamental group containing infinitely many embedded
surfaces
\[
   F_n^Y:\Sigma_g\hookrightarrow Y_\tau,
   \qquad n=0,1,2,\ldots,
\]
that are all homotopic to \(F_0^Y\), are \(\pi_1\)-injective with common
image subgroup \(H\leq\pi_1Y_\tau\), and whose embedded images are
pairwise not smoothly image-concordant in \(Y_\tau\).  The subgroup \(H\)
is proper and nonnormal.

After one stabilization
\[
   X_\tau=Y_\tau\#(S^2\times S^2),
\]
the corresponding embedded surfaces
\[
   F_n:\Sigma_g\hookrightarrow X_\tau,
   \qquad n=0,1,2,\ldots,
\]
have the following properties:
\begin{enumerate}[label=\textup{(\roman*)}]
\item the surfaces \(F_n\) are all homotopic to \(F_0\);
\item each \(F_n\) is \(\pi_1\)-injective, with common image subgroup
\(H\leq\pi_1X_\tau\);
\item the surfaces have a common framed embedded dual sphere;
\item the complement maps
\[
   \pi_1(X_\tau\setminus\nu F_n)\xrightarrow{\cong}\pi_1X_\tau
\]
are isomorphisms for every \(n\geq0\);
\item the embedded images
\[
   F_0(\Sigma_g),F_1(\Sigma_g),F_2(\Sigma_g),\ldots
\]
are pairwise not smoothly image-concordant in \(X_\tau\).
\end{enumerate}
Moreover, for every \(k\geq0\), after the stabilization
\[
   X_\tau^{(k)}=X_\tau\# k(S^2\times S^2),
\]
with all connected sums taken away from the surfaces, construction tracks,
marking data, and common dual sphere, the common framed dual sphere and the
complement \(\pi_1\)-isomorphisms persist, and the embedded images remain
pairwise not smoothly image-concordant.

Finally, the stabilized nonconcordance is detected by marked self-dual
double-coset coordinates: there are pairwise distinct nontrivial self-dual
double-cosets
\[
   \mathcal D_n=Ht_nH\in H\backslash\pi_1X_\tau/H,
   \qquad n\geq1,
\]
with \(t_n^2\in H\), such that, for every \(i\geq1\),
\[
   Q^{\mathcal D_i}_\rho(X_\tau,F_0)
   \cong
   \mathbb F_2\langle u_{\mathcal D_i}\rangle,
\]
and, for all \(i,n\geq1\),
\[
   FQ^{\mathcal D_i}_\rho(F_0,F_n)
   =
   \begin{cases}
   u_{\mathcal D_i},& i=n,\\
   0,& i\neq n.
   \end{cases}
\]
\end{maintheorem}

The same examples are not locally flat topologically image-concordant.  By
Daher--Powell, a locally flat topological concordance between smoothly
embedded surfaces in a smooth four-manifold can be homotoped rel boundary to
a smooth concordance \cite{DaherPowell2026}.  Thus a locally flat
topological image concordance between two of the surfaces in the theorem would imply a smooth image concordance,
contradicting the theorem. 

In a recent preprint, Lin--Wu--Xie--Zhang develop a closed-surface Dax
framework and construct homotopic closed positive-genus
\(\pi_1\)-injective surfaces in product settings such as
\(\Sigma\times S^2\), with a common geometric dual, whose embedded images
remain nonisotopic even after reparametrization
\cite{LinWuXieZhang}. The present paper
instead obstructs concordance of embedded images.  Its principal closed examples use a proper
nonnormal surface subgroup and self-dual double-coset labels, and its main
examples lie in closed aspherical mapping tori; after one stabilization they have a common framed dual sphere and complement
\(\pi_1\)-isomorphisms, and their image nonconcordance persists through
every finite further stabilization.

In a separate recent preprint, Lin--Wu construct homotopic embedded tori in
\(T^4\#(S^2\times S^2)\) with a common geometric dual that remain smoothly
nonisotopic after any finite number of stabilizations
\cite{LinWu2026}.  The present result differs in that it proves persistence of image nonconcordance, using a proper nonnormal surface subgroup and self-dual double-coset labels.

Earlier, Auckly--Kim--Melvin--Ruberman--Schwartz proved one-stabilization isotopy
results for homologous same-genus surfaces with simply connected complements
\cite{AucklyKimMelvinRubermanSchwartz2019}; their theorem applies under different hypotheses and provides useful context for persistence under every finite stabilization.

We next describe the obstruction.  For a \(\rho\)-marked surface, let
\(X_H\to X\) be the cover associated to the actual subgroup
\(H\leq\pi\) fixed by the marking. If two lifted points
\(\widehat p,\widehat q\in X_H\) project to the same point of \(X\), their
sheet difference is naturally a double-coset in \(H\backslash\pi/H\).  On
exchanging the two ordered points, the label is inverted.  Hence a single
mod-two coordinate is naturally attached to a nontrivial self-dual
double-coset
\[
   D=HgH=Hg^{-1}H\neq H .
\]
When \(H\) is nonnormal, these labels cannot in general be reduced to labels
in a quotient group, so the double-coset formulation is intrinsic.

For a marked homotopy track
\[
   A:\Sigma_g\times I\to X\times I
\]
from \(F_0\) to \(F_1\), the raw coordinate counts, modulo two, the connected
components of the off-diagonal self-intersection locus labelled by \(D\) in
the \(H\)-cover.  This raw count is taken modulo the standard
non-simply-connected indeterminacies: local \(D\)-labelled
\(\pi_3X\)-insertions and loops in the marked mapping-space component.  The
result is the class
\[
   FQ^D_\rho(F_0,F_1)\in Q^D_\rho(X,F_0),
\]
which vanishes when \(F_0\) and \(F_1\) are joined by a concordance
respecting the chosen marking.  The construction
uses standard Freedman--Quinn/Dax/Hatcher--Quinn-type self-intersection
data. From these data we extract a computable marked mod-two coordinate for
closed \(\pi_1\)-injective surfaces, allowing \(H\) to be nonnormal, and
combine it with square-root realizations and endpoint rigidity to obstruct
image concordance. In the main examples, the ambient group is torsion-free and
\(H^1(\Sigma_g;(\pi_2X)_\rho)=0\), so the relevant target is the
one-dimensional vector space \(\mathbb F_2\langle u_D\rangle\). Related Dax methods for
disks and embedding spaces in four-manifolds were developed by
Kosanovi\'c--Teichner \cite{KosanovicTeichner2024}.

The geometric realization of the self-dual labels is elementary.  Suppose
\(\Sigma_g\subset M\) is a \(\pi_1\)-injective surface in a compact
orientable three-manifold and put \(H=\pi_1\Sigma_g\leq\pi_1M\).  If
\(M\setminus\operatorname{Int}\nu\Sigma_g\) contains an embedded M\"obius
band \(B\) whose boundary is a normal push-off \(c^+\) of an essential curve
\(c\subset\Sigma_g\), then the core of \(B\) represents an element
\(t\in\pi_1M\) with
\[
   t^2=c\in H .
\]
Thus \(t^{-1}=tc^{-1}\), and consequently
\[
   HtH=Ht^{-1}H .
\]
The M\"obius band therefore supplies a self-dual double-coset label even
when the ambient group is torsion-free.  In irreducible surface complements with incompressible boundary, the
Jaco--Shalen--Johannson characteristic-pair theory localizes essential
M\"obius bands in the characteristic submanifold
\cite{JacoShalen1979,Johannson1979,NeumannSwarup1997}.

The infinite supply of such roots comes from the orientable twisted
\(I\)-bundle over the Klein bottle.  In the notation used below, its
fundamental group is
\[
   \langle a,t\mid tat^{-1}=a^{-1}\rangle .
\]
It contains properly embedded M\"obius bands with core elements \(a^nt\),
\(n\in\mathbb Z\), all having the same boundary square:
\[
   (a^nt)^2=t^2 .
\]
Thus one obtains infinitely many square-root elements with common square.
The closed graph-manifold construction closes this local model while
retaining distinct double-coset labels.  A Dehn twist along a nonseparating
JSJ torus is then chosen so that, in the mapping-torus group, the normalizer
of the surface subgroup is exactly \(H\).  This endpoint-normalizer rigidity
turns the marked coordinate into an obstruction to image concordance under
arbitrary endpoint reparametrization.  Stabilizing by
one \(S^2\times S^2\)-summand adds the common framed dual sphere and makes
the complement maps \(\pi_1\)-isomorphisms, without changing the local
double-coset calculation.

There is also a normal index-two specialization.  If
\(p:\Sigma_g\to N_{g+1}\) is the orientation double cover of a closed
nonorientable surface, then
\[
   H=p_*\pi_1\Sigma_g
\]
is normal of index two in \(\pi_1N_{g+1}\).  The unique nontrivial
double-coset label is represented by the deck transformation of the
\(H\)-cover.  Applying the same marked coordinate in this regular-cover
setting gives a pair of homotopic \(\pi_1\)-injective genus-\(g\) surfaces in
\[
   D(L\oplus\varepsilon^1)\#(S^2\times S^2),
\]
with common framed dual sphere and complement \(\pi_1\)-isomorphisms, whose
embedded images are not smoothly image-concordant.  This comparison
illustrates the regular-cover case; the main closed graph-manifold examples
are non-normal.

\medskip
\noindent\textbf{Organization.}
Section~\ref{sec:labels} fixes markings, preferred lifts, and non-normal
double-coset labels.  Sections~\ref{sec:raw-count}--\ref{sec:cellular}
construct the marked double-coset coordinate, including the local
\(\pi_3X\)-insertion and mapping-space-loop indeterminacies, and prove the
survival criterion.  Section~\ref{sec:D-crossed-blocks} gives the local
geometric input: \(D\)-crossed tracks, the annular movie, standard
square-root neighbourhoods, and the M\"obius-band square-root lemma.
Section~\ref{sec:image} proves the endpoint-rigid promotion from marked
nonconcordance to image nonconcordance.  Section~\ref{sec:applications}
applies the obstruction to M\"obius-band square-root families, the compact
Klein-bottle \(I\)-bundle model, and the closed graph-manifold mapping-torus
family, first unstabilized in \(Y_\tau\) and then stabilized in \(X_\tau\).
Section~\ref{sec:orientation-family} gives the orientation-double-cover
comparison, and Section~\ref{sec:questions} records further questions.

\medskip
\noindent\textbf{Acknowledgments.}
The author used Prism, Overleaf AI and Grammarly to assist with improving the LaTeX formatting and English grammar of this paper.

\section{The marked obstruction}\label{sec:marked-obstruction}
\subsection{Marked surfaces, \texorpdfstring{\(H\)}{H}-covers, and double-coset labels}\label{sec:labels}

In this section, we fix basepoints, markings, and preferred lifts.  We then define the sheet-difference label for two lifted points over the same point of \(X\).  The essential facts are that the label is independent of path choices and that exchanging the ordered pair inverts the label.  The same definition is used for tracks in \(X\times I\) and loop-tracks in \(X\times S^1\) by ignoring the parameter coordinate after equality in the target has forced the two parameter values to coincide.

\begin{convention}\label{conv:basepoints-markings}
Throughout the paper, \(X\) is a connected oriented smooth four-manifold with basepoint \(x_0\), and \(\pi=\pi_1(X,x_0).\) The surface \(\Sigma_g\) has basepoint \(z_0.\)  A based embedding is an embedding
\(
f:\Sigma_g\hookrightarrow \operatorname{int}X
\)
together with a path \(\lambda_f\) from \(x_0\) to \(f(z_0)\).  This path determines
\[
f_*^{\lambda_f}:\pi_1(\Sigma_g,z_0)\to \pi,
\qquad
f_*^{\lambda_f}([\alpha])=[\lambda_f\cdot f(\alpha)\cdot \overline{\lambda_f}].
\]
Thus the surface subgroup is fixed as an actual subgroup of \(\pi\), not only up to conjugacy.
\end{convention}

\begin{definition}\label{def:marked-surface}
Let \(H\leq \pi\) be a fixed subgroup.  A \(\rho\)-marked genus-\(g\) surface in \(X\) is an embedding
\(
f:\Sigma_g\hookrightarrow \operatorname{int}X
\)
together with a base path \(\lambda_f\) and an isomorphism
\(
\rho:\pi_1(\Sigma_g,z_0)\xrightarrow{\cong}H
\)
such that
\(
f_*^{\lambda_f}=i_H\circ\rho,
\)
where \(i_H:H\hookrightarrow\pi\) is inclusion.  In particular, \(f\) is \(\pi_1\)-injective.
\end{definition}

Let
\(
q:X_H\to X
\)
be the covering corresponding to \(H\leq\pi\), and fix a lift \(\widehat x_0\in q^{-1}(x_0)\).  A \(\rho\)-marked surface \((f,\lambda_f,\rho)\) has a preferred lift \(\widehat f:\Sigma_g\to X_H\): lift \(\lambda_f\) starting at \(\widehat x_0\), and require \(\widehat f(z_0)\) to be the endpoint of this lifted path.

\begin{definition}\label{def:tracks-concordance}
A track is a proper map
\(
A:\Sigma_g\times I\to X\times I
\)
which is product near \(\Sigma_g\times\{0,1\}\) and satisfies
\(
A^{-1}(X\times\{0,1\})=\Sigma_g\times\{0,1\}.
\)

Let \(F_0,F_1\) be \(\rho\)-marked surfaces with preferred lifts \(\widehat F_0,\widehat F_1\).  A \(\rho\)-marked track from \(F_0\) to \(F_1\) is a track \(A\) with endpoint restrictions \(F_0\times\{0\}\) and \(F_1\times\{1\}\), admitting a lift
\(
\widehat A:\Sigma_g\times I\to X_H\times I
\)
whose endpoint restrictions are \(\widehat F_0\times\{0\}\) and \(\widehat F_1\times\{1\}\).  A \(\rho\)-marked concordance is a \(\rho\)-marked track which is an embedding.

A parametrized concordance is an embedded track preserving the endpoint parametrizations but not necessarily the markings.  The embedded images of \(F_0\) and \(F_1\) are image-concordant if \(F_0\) is parametrized concordant to \(F_1\circ\phi\) for some diffeomorphism \(\phi:\Sigma_g\to\Sigma_g\).  If \(\phi\) is required to preserve orientation, this is oriented image concordance; otherwise it is unoriented image concordance.
\end{definition}

\begin{definition}\label{def:double-coset-label}
Let \(\widehat p,\widehat q\in X_H\) satisfy \(q(\widehat p)=q(\widehat q)\).  Choose paths
\(
\widehat\alpha:[0,1]\to X_H\) and \(\widehat\beta:[0,1]\to X_H
\) from \(\widehat x_0\) to \(\widehat p\) and \(\widehat q\), respectively.  Let \(\alpha=q\widehat\alpha\) and \(\beta=q\widehat\beta\).  Define
\[
\ell(\widehat p,\widehat q)=H[\alpha\overline\beta]H\in H\backslash\pi/H.
\]
If \(\widehat p=(p_H,t)\) and \(\widehat q=(q_H,t)\) are points in \(X_H\times I\) or \(X_H\times S^1\) with the same parameter coordinate, we set
\[
\ell(\widehat p,\widehat q):=\ell(p_H,q_H).
\]
\end{definition}
We call an element of \(H\backslash\pi/H\) arising in this way a
double-coset label, or simply a label.

\begin{lemma}\label{lem:label-well-defined}
The double coset \(\ell(\widehat p,\widehat q)\) is independent of the choices of \(\widehat\alpha\) and \(\widehat\beta\).  Moreover
\(
\ell(\widehat q,\widehat p)=\ell(\widehat p,\widehat q)^{-1}.
\)
\end{lemma}

\begin{proof}
Let \(\widehat\alpha'\) be another path from \(\widehat x_0\) to \(\widehat p\), and put \(\alpha'=q\widehat\alpha'\).  The loop \(\alpha'\overline\alpha\) lifts to a loop in \(X_H\) based at \(\widehat x_0\), hence represents an element \(h_L\in H\).  Therefore
\[
[\alpha'\overline\beta]=[\alpha'\overline\alpha]\,[\alpha\overline\beta]=h_L[\alpha\overline\beta],
\]
so changing the first lifted path multiplies the representative on the left by an element of \(H\).

Similarly, if \(\widehat\beta'\) is another path from \(\widehat x_0\) to \(\widehat q\), then \(\beta'\overline\beta\) represents an element \(h_R\in H\), and
\[
[\alpha\overline{\beta'}]=[\alpha\overline\beta]\,[\beta\overline{\beta'}]
=[\alpha\overline\beta]h_R^{-1}.
\]
Thus changing the second lifted path multiplies the representative on the right by an element of \(H\).  The double coset is therefore independent of all path choices.

Exchanging the two lifted points replaces \([\alpha\overline\beta]\) by \([\beta\overline\alpha]=[\alpha\overline\beta]^{-1}\), so the double coset is inverted.
\end{proof}

\begin{lemma}
\label{lem:marking-label-compatibility}
Let \(A:\Sigma_g\times I\to X\times I\) be a \(\rho\)-marked track with
preferred lift \(\widehat A\). Let \(x,y\in\Sigma_g\times I\) satisfy \(A(x)=A(y)\).  Then 
\(
\ell(\widehat A(x),\widehat A(y))\in H\backslash\pi/H
\)
may be computed from any two source paths from \((z_0,0)\) to \(x\) and
\(y\), after prefixing by the fixed bottom base path \(\lambda_0\).  Changing
either source path changes the resulting representative in \(\pi\) only by
left or right multiplication by an element of \(H\).  Hence the resulting
double coset agrees with Definition~\ref{def:double-coset-label}.
\end{lemma}

\begin{proof}
Write
\[
A_X=\operatorname{pr}_X\circ A:\Sigma_g\times I\to X
\]
for the \(X\)-component of the track. Let \(m_0=(z_0,0)\), and let \(\lambda_0:x_0\to F_0(z_0)\) be the fixed bottom base path used in the marking of \(F_0\).  If \(\omega\) is a loop in \(M=\Sigma_g\times I\) based at \(m_0\), then
\[
[\lambda_0\cdot A_X(\omega)\cdot\overline{\lambda_0}]\in H.
\]
Indeed, the preferred lift of \(A\) starts at the preferred lift of \(F_0\), and the lift of \(\lambda_0\) from \(\widehat x_0\) ends at \(\widehat F_0(z_0)\).  Thus the based loop
\(
\lambda_0\cdot A_X(\omega)\cdot\overline{\lambda_0}
\)
lifts to a loop based at \(\widehat x_0\), hence represents an element of the subgroup \(H\leq\pi_1(X,x_0)\).

Now compute the label of \(\widehat A(x)\) and \(\widehat A(y)\) using source paths \(\alpha:m_0\to x\) and \(\beta:m_0\to y\).  The based representative is
\[
[\lambda_0\cdot A_X(\alpha)\cdot \overline{A_X(\beta)}\cdot\overline{\lambda_0}]\in\pi.
\]
Changing \(\alpha\) by preconcatenating a source loop \(\omega\) changes this representative by left multiplication by
\(
[\lambda_0\cdot A_X(\omega)\cdot\overline{\lambda_0}].
\)
Changing \(\beta\) by preconcatenating a source loop changes the representative by right multiplication by an element of \(H\), with inverse depending on the convention for the changed path.  These are exactly the left and right ambiguities of the double coset.  Therefore the source-path definition agrees with Definition~\ref{def:double-coset-label}.
\end{proof}

\subsection{The raw \texorpdfstring{\(D\)}{D}-labelled cover count}\label{sec:raw-count}

In this section, we define the raw count for a generic marked track.  In the general setting, it is formulated using the off-diagonal double locus downstairs, restricted to ordered pairs whose label is a chosen self-dual double coset \(D\).  We prove compactness, boundarylessness, relative invariance, and level straightening.

\begin{convention}\label{conv:interior-target}
All tracks, homotopies of tracks, loop-tracks, and local insertions are taken to satisfy the following interior-target condition.  After prescribed endpoint collars, the image lies in
\(
\operatorname{int}X\times(0,1).
\) For loop-tracks \(\Sigma_g\times S^1\to X\times S^1\), the image lies in \(\operatorname{int}X\times S^1\).  All perturbations are taken inside this relative class, rel endpoint
collars, and preserve the data which make the track \(\rho\)-marked: the
source basepoint, the endpoint base paths, and the prescribed bottom and
top preferred lifts.
\end{convention}

Fix a nontrivial self-dual double coset
\(
D=HgH\neq H\), \(D=D^{-1}.\) Let \(\Ftwo\langle u_D\rangle\) denote the one-dimensional \(\Ftwo\)-vector space generated by a formal symbol \(u_D\). Let \(M=\Sigma_g\times I\), and let
\(
A:M\to X\times I
\)
be a \(\rho\)-marked track with preferred lift \(\widehat A:M\to X_H\times I\).

\begin{definition}\label{def:PD}
Define the off-diagonal double locus
\[
P^\circ(A)=\{(x,y)\in M\times M\setminus\Delta_M:A(x)=A(y)\}.
\]
For \((x,y)\in P^\circ(A)\), define
\[
\ell_A(x,y)=\ell(\widehat A(x),\widehat A(y))\in H\backslash\pi/H.
\]
Then set
\[
P_D(A)=\{(x,y)\in P^\circ(A):\ell_A(x,y)=D\}.
\]
For a generic track, define the raw count
\[
FQ^{0,D}_\rho(A)=
\bigl(\#\pi_0(P_D(A))\bmod2\bigr)u_D
\in\Ftwo\langle u_D\rangle.
\]
In particular, each connected component of \(P_D(A)\) contributes one
copy of \(u_D\).
The same definition applies to a level-preserving loop-track \(\Sigma_g\times S^1\to X\times S^1\), with \(M=\Sigma_g\times S^1\).
\end{definition}

\begin{lemma}[Compactness and boundarylessness]\label{lem:PD-compact-boundaryless}
For a generic \(\rho\)-marked track \(A\) satisfying Convention~\ref{conv:interior-target}, the space \(P_D(A)\) is a compact smooth \(1\)-manifold without boundary.
\end{lemma}

\begin{proof}
Write \(A_X=\operatorname{pr}_X\circ A\).  Consider the map
\[
\Phi:M\times M\setminus\Delta_M\to (X\times I)\times(X\times I),
\qquad
\Phi(x,y)=(A(x),A(y)).
\]
After a perturbation rel endpoint collars and marked data, \(\Phi\) is transverse to the diagonal.  Since \(\dim(M\times M)=6\) and the diagonal in \((X\times I)^2\) has codimension \(5\), the off-diagonal double locus
\(
P^\circ(A)=\Phi^{-1}(\Delta)
\)
is a smooth one-manifold.

The label \(\ell_A\) is locally constant on \(P^\circ(A)\).  Indeed, along a small path in \(P^\circ(A)\), the two lifted points move continuously over the same downstairs point.  Over a small evenly covered neighbourhood in \(X\times I\), the sheet set is discrete, so the double-coset sheet difference cannot jump.  Therefore \(P_D(A)\) is a union of connected components of \(P^\circ(A)\).

It remains to prove that no \(D\)-labelled component can accumulate on the deleted source diagonal.  Suppose, for contradiction, that a sequence of off-diagonal double points
approaches the source diagonal.  Thus there are pairs
\(
(x_n,y_n)\in P^\circ(A)
\)
such that
\(
x_n\to x\) and  \(y_n\to x\) for some \(x\in M\). Choose a small coordinate ball \(U\subset M\) around \(x\) such that \(A(U)\) is contained in a contractible coordinate ball \(W\subset X\times I\).  For all sufficiently large \(n\), \(x_n,y_n\in U\).

Let \(m_0=(z_0,0)\).  Choose a path \(\gamma\) in \(M\) from \(m_0\) to \(x\).  Choose paths \(u_n\) from \(x\) to \(x_n\), \(v_n\) from \(x\) to \(y_n\), and \(\delta_n\) from \(x_n\) to \(y_n\), all contained in \(U\).  Since \(A(x_n)=A(y_n)\), the path \(A(\delta_n)\) is a loop in the contractible ball \(W\), based at \(A(x_n)=A(y_n)\), and hence is nullhomotopic rel basepoint in \(X\times I\).

Let \(\lambda_0:x_0\to F_0(z_0)\) be the fixed bottom base path.  Compute the double-coset label using the lifted source paths \(\gamma*u_n\) and \(\gamma*v_n\).  The corresponding based representative is
\[
[\lambda_0\cdot A_X(\gamma*u_n)\cdot
\overline{A_X(\gamma*v_n)}\cdot\overline{\lambda_0}]
\in\pi.
\]
The path \(A(\delta_n)\) is a nullhomotopic loop in the contractible coordinate ball \(W\), based at \(A(x_n)=A(y_n)\).  Inserting this nullhomotopic loop after \(A_X(\gamma*u_n)\) does not change the based element.  Hence the representative equals the class of
\[
L=\lambda_0\cdot
A_X(\gamma*u_n*\delta_n*\overline{v_n}*\overline{\gamma})
\cdot\overline{\lambda_0}.
\]
The source path
\(
\gamma*u_n*\delta_n*\overline{v_n}*\overline{\gamma}
\)
is a loop in \(M\) based at \(m_0\).  Since \(A\) is a \(\rho\)-marked track with preferred lift, the based loop obtained by prefixing and suffixing with \(\lambda_0\), \(L\) lifts to a loop in \(X_H\) based at \(\widehat x_0\).  Therefore its class lies in \(H\).  Hence the double coset of \((x_n,y_n)\) is the trivial double coset \(H\) for all sufficiently large \(n\).  Since \(D\neq H\), no \(D\)-labelled component accumulates on \(\Delta_M\).

Now consider boundary faces of \(M\times M\).  Suppose a point of the closure of \(P_D(A)\) has \(x\in\Sigma_g\times\{0\}\).  Equality \(A(x)=A(y)\) forces the \(I\)-coordinate of \(A(y)\) to be \(0\).  Since the track is proper as a map of pairs, \(y\in\Sigma_g\times\{0\}\).  Endpoint embeddedness gives \(F_0(x)=F_0(y)\) only when \(x=y\).  Thus every bottom-level boundary limit is diagonal and has label \(H\), hence is excluded from \(P_D(A)\).  The top endpoint is identical with \(F_1\) in place of \(F_0\).  There are no target-boundary contributions by Convention~\ref{conv:interior-target}.

Thus \(P_D(A)\) is closed in a compact subset of \(M\times M\setminus\Delta_M\), and the transverse inverse-image structure has no source-boundary, diagonal, or target-boundary endpoints.  Hence \(P_D(A)\) is a compact smooth one-manifold without boundary.
\end{proof}

\begin{lemma}[Relative invariance]\label{lem:relative-invariance-D}
Let \(A^r\), \(r\in[0,1]\), be a generic one-parameter family of \(\rho\)-marked tracks satisfying Convention~\ref{conv:interior-target}, fixed near \(r=0,1\), fixed near endpoint collars, and fixed on marked data.  Then
\[
FQ^{0,D}_\rho(A^0)=FQ^{0,D}_\rho(A^1).
\]
\end{lemma}

\begin{proof}
The parametrized off-diagonal double locus is
\[
\mathcal P^\circ=
\{(x,y,r)\in M\times M\setminus\Delta_M\times I:A^r(x)=A^r(y)\}.
\]
For a generic one-parameter family this is a smooth surface, since the domain has dimension \(7\) and the target diagonal has codimension \(5\).  The \(D\)-labelled part
\[
\mathcal P_D=
\{(x,y,r)\in\mathcal P^\circ:\ell_{A^r}(x,y)=D\}
\]
is a union of components.  The diagonal and endpoint arguments from Lemma~\ref{lem:PD-compact-boundaryless}, applied uniformly in the compact parameter interval, show that \(\mathcal P_D\) is compact and has boundary exactly
\[
\partial\mathcal P_D=P_D(A^0)\sqcup P_D(A^1).
\]
There is no additional boundary from \(\partial M\), the source diagonal, or \(\partial X\times I\).

The surface \(\mathcal P_D\) is orientable.  The ambient manifold \(M\times M\times I\) is oriented, and the diagonal in \((X\times I)^2\) has oriented normal bundle canonically identified with \(T(X\times I)\); a transverse inverse image therefore inherits an orientation.

The exchange involution
\(
\iota(x,y,r)=(y,x,r)
\)
is free because \(x=y\) is excluded.  By Lemma~\ref{lem:label-well-defined}, exchange sends label \(D\) to label \(D^{-1}\).  Since \(D=D^{-1}\), it preserves \(\mathcal P_D\).  Thus \(\mathcal P_D\to\mathcal P_D/\iota\) is a twofold covering and
\[
\chi(\mathcal P_D)=2\chi(\mathcal P_D/\iota)\equiv0\pmod2.
\]
For a compact orientable surface \(S\), \(\chi(S)\equiv\#\pi_0(\partial S)\pmod2\).  Hence
\[
0\equiv\chi(\mathcal P_D)
\equiv
\#\pi_0(P_D(A^0))+\#\pi_0(P_D(A^1))
\pmod2,
\]
which proves the equality of raw counts.
\end{proof}

\begin{lemma}[Level straightening]\label{lem:level-straightening-D}
Every \(\rho\)-marked proper track
\(
A:\Sigma_g\times I\to X\times I
\)
which is product near the boundary and has prescribed endpoint lifts is rel-boundary homotopic through \(\rho\)-marked proper tracks satisfying Convention~\ref{conv:interior-target} to a level-preserving track
\[
A^{\operatorname{lev}}(z,s)=(a(z,s),s).
\]
\end{lemma}

\begin{proof}
Write \(A(z,s)=(a(z,s),t(z,s))\).  Properness as a map of pairs and the product condition near the boundary give \(t=s\) near \(\Sigma_g\times\{0,1\}\) and
\(
t^{-1}\{0,1\}=\Sigma_g\times\{0,1\}.
\)
Define
\[
A_u(z,s)=\bigl(a(z,s),(1-u)t(z,s)+us\bigr),\qquad u\in[0,1].
\]
If the second coordinate is \(0\), then \((1-u)t(z,s)+us=0\).  Both terms are nonnegative, so either \(s=0\) when \(u>0\), or \(t(z,s)=0\) when \(u=0\), and properness again gives \(s=0\).  The top endpoint is the same after replacing \(t\) and \(s\) by \(1-t\) and \(1-s\).  Thus the homotopy remains proper as a map of pairs.

The \(X\)-coordinate is unchanged throughout the homotopy.  Therefore the marking and the preferred lift are preserved.  If a generic level-preserving representative is needed, perturb only the \(X\)-coordinate, away from endpoint collars and marked data.  This preserves the second coordinate \(s\), the endpoint lifts, and the interior-target convention.
\end{proof}

\subsection{Local \texorpdfstring{\(D\)}{D}-labelled \texorpdfstring{\(\pi_3X\)}{pi3X}-insertions}\label{sec:local-insertions}

In this section, we define the local \(D\)-labelled \(\pi_3X\)-insertion
indeterminacy \(\mu_D\).  A local representative is inserted in a small
prepared ball while the original track is kept fixed outside the support.
The definition separates the local insertion contribution from the mixed
comparison terms needed to compare the inserted track with the original
track. We prove that the resulting value is independent of the preparation choices,
changes the raw count by exactly \(\mu_D\), and vanishes when \(D\) contains
no element of order two.

\begin{definition}[Admissible local \(D\)-insertions]\label{def:muD}
Let \(A:M=\Sigma_g\times I\to X\times I\) be a generic marked track with preferred lift \(\widehat A\).  Choose a point
\(
p\in\operatorname{int}M
\)
away from endpoint collars and marked data, and away from the two projections of \(P_D(A)\).  Choose a coordinate ball
\(
B\subset\operatorname{int}M
\)
centered at \(p\), small enough that
\[
B\cap\operatorname{pr}_1P_D(A)=\varnothing,
\qquad
B\cap\operatorname{pr}_2P_D(A)=\varnothing.
\]
Choose a smaller coordinate ball \(B'\Subset B\).

Choose a collapse map \(c:B\to B\) which is the identity near \(\partial B\), is constant equal to \(p\) on a neighbourhood of \(B'\), satisfies \(c(B)\subset B\), and is homotopic to the identity rel \(\partial B\).  Replacing \(A|_B\) by \(A\circ c\), and leaving \(A\) unchanged outside \(B\), gives the prepared track \(A_{\operatorname{prep}}\).  Its preferred lift is denoted \(\widehat A_{\operatorname{prep}}\).

Choose a based class
\(
\alpha_p\in \pi_3(X\times I,A(p))
\cong
\pi_3(X,\operatorname{pr}_X A(p)).
\)
Equivalently, one may obtain such a class by choosing a path from \(x_0\)
to \(\operatorname{pr}_X A(p)\) and transporting a class of
\(\pi_3(X,x_0)\) along that path.  The subgroup
\(\operatorname{im}\mu_D\) below is generated over all insertion points,
all such transport paths, and all based classes obtained in this way.

Choose a based representative
\[
b:(B',\partial B')\to
\bigl(\operatorname{int}X\times(t_p-\epsilon,t_p+\epsilon),A(p)\bigr),
\]
constant on a closed collar \(C\) of \(\partial B'\), where \(A(p)\in X\times\{t_p\}\) and \(0<t_p-\epsilon<t_p+\epsilon<1\).  Let
\[
B'_{\operatorname{nc}}=\overline{B'\setminus C},
\qquad
B'_{\operatorname{act}}=\operatorname{int}(B'_{\operatorname{nc}}),
\qquad
S=\partial B'_{\operatorname{nc}}.
\]
The set \(B'_{\operatorname{act}}\) is the active, possibly nonconstant, part of the insertion.  The collar \(C\) is the constant collar.  The designed interface is \(S\).  Since \(B'\) is simply connected, the preferred point \(\widehat A(p)\) determines a unique lift
\[
\widehat b:(B',\partial B')\to (X_H\times I,\widehat A(p)),
\]
constant on \(C\).

The associated inserted track \(A^{\operatorname{ins}}\) is obtained by using \(b\) on \(B'\) and \(A_{\operatorname{prep}}\) outside \(\operatorname{int}B'\).  This is well-defined because \(b\) and \(A_{\operatorname{prep}}\) are both constant equal to \(A(p)\) near \(\partial B'\).  For the purpose of counting, decompose the source as
\[
M=B'_{\operatorname{act}}\sqcup M_{\operatorname{prep}},
\qquad
M_{\operatorname{prep}}:=M\setminus B'_{\operatorname{act}}.
\]
On \(M_{\operatorname{prep}}\cap B'\), both \(A^{\operatorname{ins}}\) and \(A_{\operatorname{prep}}\) are the same constant map.

Let
\[
\mathcal L_D(\widehat A(p))=
\{\widehat z\in(q\times\operatorname{id})^{-1}(A(p)):
\ell(\widehat z,\widehat A(p))=D\}.
\]
Since \(D=D^{-1}\), this set also records the opposite ordered collar relation.

The representative is called \emph{admissible} if, after a perturbation supported in \(B'_{\operatorname{act}}\), the following conditions hold.
\begin{enumerate}[label=\textup{(L\arabic*)}]
\item The collar \(C\) maps constantly to \(\widehat A(p)\).
\item No \(D\)-labelled collar--active point occurs: there is no \(x\in B'_{\operatorname{act}}\) such that
\[
b(x)=A(p)
\quad\text{and}\quad
\widehat b(x)\in\mathcal L_D(\widehat A(p)).
\]
Moreover, the same exclusion holds on a neighbourhood of the interface \(S\) inside \(B'_{\operatorname{nc}}\).
\item No \(D\)-labelled insertion--insertion solution occurs on the designed interface:
\[
\begin{aligned}
&\{(x,y)\in S\times B'_{\operatorname{nc}}:
 b(x)=b(y),\ \ell(\widehat b(x),\widehat b(y))=D\}=\varnothing,\\
&\{(x,y)\in B'_{\operatorname{nc}}\times S:
 b(x)=b(y),\ \ell(\widehat b(x),\widehat b(y))=D\}=\varnothing.
\end{aligned}
\]
\item No mixed solution occurs on the designed interface:
\[
\begin{aligned}
&\{(x,y)\in S\times M_{\operatorname{prep}}:
 b(x)=A_{\operatorname{prep}}(y),
 \ell(\widehat b(x),\widehat A_{\operatorname{prep}}(y))=D\}=\varnothing,\\
&\{(x,y)\in M_{\operatorname{prep}}\times S:
 A_{\operatorname{prep}}(x)=b(y),
 \ell(\widehat A_{\operatorname{prep}}(x),\widehat b(y))=D\}=\varnothing.
\end{aligned}
\]
\item The following three ordered maps are transverse to the diagonal in \((X\times I)^2\):
\[
B'_{\operatorname{act}}\times B'_{\operatorname{act}}\setminus\Delta
\longrightarrow (X\times I)^2,
\qquad
(x,y)\longmapsto (b(x),b(y)),
\]
\[
B'_{\operatorname{act}}\times M_{\operatorname{prep}}
\longrightarrow (X\times I)^2,
\qquad
(x,y)\longmapsto (b(x),A_{\operatorname{prep}}(y)),
\]
and
\[
M_{\operatorname{prep}}\times B'_{\operatorname{act}}
\longrightarrow (X\times I)^2,
\qquad
(x,y)\longmapsto (A_{\operatorname{prep}}(x),b(y)).
\]
\end{enumerate}

With these conventions, define
\[
P_{II}(\widehat b)=
\{(x,y)\in B'_{\operatorname{act}}\times B'_{\operatorname{act}}\setminus\Delta:
 b(x)=b(y),\ \ell(\widehat b(x),\widehat b(y))=D\},
\]
\[
P_{IP}=
\{(x,y)\in B'_{\operatorname{act}}\times M_{\operatorname{prep}}:
 b(x)=A_{\operatorname{prep}}(y),
 \ell(\widehat b(x),\widehat A_{\operatorname{prep}}(y))=D\},
\]
\[
P_{PI}=
\{(x,y)\in M_{\operatorname{prep}}\times B'_{\operatorname{act}}:
 A_{\operatorname{prep}}(x)=b(y),
 \ell(\widehat A_{\operatorname{prep}}(x),\widehat b(y))=D\},
\]
and
\[
P_{PP}=
\{(x,y)\in M_{\operatorname{prep}}\times M_{\operatorname{prep}}\setminus\Delta:
 A_{\operatorname{prep}}(x)=A_{\operatorname{prep}}(y),
 \ell(\widehat A_{\operatorname{prep}}(x),\widehat A_{\operatorname{prep}}(y))=D\}.
\]
By Definition~\ref{def:muD}, conditions \textup{(L2)}--\textup{(L4)} exclude collar and interface boundary points, and \textup{(L5)} gives transversality.  Hence \(P_{II}\), \(P_{IP}\), and \(P_{PI}\) are compact smooth one-manifolds without boundary. 

Define the local value 
\[
\mu_D(\alpha_p)=
\bigl(\#\pi_0(P_{II}(\widehat b))\bmod2\bigr)u_D.
\]
Finally, \(\operatorname{im}\mu_D\subset\Ftwo\langle u_D\rangle\) is the subgroup generated by all such values.
\end{definition}
Admissible representatives for every transported based class are constructed in Lemma~\ref{lem:relative-extension-local-D}\textup{(ii)}.

\begin{lemma}\label{lem:muD-well-defined}
For a fixed insertion point, fixed lifted insertion basepoint, and fixed transported based class \(\alpha_p\), the value \(\mu_D(\alpha_p)\) is independent of the admissible representative, collar size, insertion ball \(B'\), larger support ball \(B\), and collapse map.  Consequently \(\operatorname{im}\mu_D\) is a well-defined subgroup of \(\Ftwo\langle u_D\rangle\).
\end{lemma}

\begin{proof}
First fix \(B'\), the lifted insertion basepoint \(\widehat A(p)\), and a common collar convention.  Let
\[
\widehat b^0,\widehat b^1:(B',\partial B')\to(X_H\times I,\widehat A(p))
\]
be two admissible lifted representatives of the same transported based class.  After shrinking the collar if necessary, assume that both are constant on the same collar \(C\) and that the same active region \(B'_{\operatorname{act}}\) and interface \(S\) are used.

Since the two representatives define the same based relative \(\pi_3\)-class, they are joined by a based homotopy
\[
\widehat b^r:(B',\partial B')\to(X_H\times I,\widehat A(p)),
\qquad
0\le r\le1,
\]
rel the collar \(C\).  By relative general position, the homotopy may be chosen so that its nonconstant part avoids the discrete set of \(D\)-related lifts over \(A(p)\), rel a neighbourhood of the interface, and the parametrized ordered map
\[
B'_{\operatorname{act}}\times B'_{\operatorname{act}}\times I\setminus\Delta
\longrightarrow (X\times I)^2,
\qquad
(x,y,r)\longmapsto (b^r(x),b^r(y))
\]
is transverse to the diagonal.  The avoidance is possible because the homotopy domain has dimension \(4\), while \(X_H\times I\) has dimension \(5\), and the forbidden \(D\)-related lifts over \(A(p)\) are discrete. The parametrized local \(D\)-fiber product
\[
\mathcal P_D^\mu=
\{(x,y,r)\in B'_{\operatorname{act}}\times B'_{\operatorname{act}}\times I\setminus\Delta:
 b^r(x)=b^r(y),
 \ell(\widehat b^r(x),\widehat b^r(y))=D\}
\]
is therefore a compact smooth orientable surface.  The collar is fixed throughout the homotopy, and the interface-avoidance conditions exclude boundary on \(S\).  Thus
\[
\partial\mathcal P_D^\mu
=
P_{II}(\widehat b^0)\sqcup P_{II}(\widehat b^1).
\]
Ordered-pair exchange is free and sends the label to its inverse.  Since \(D=D^{-1}\), exchange preserves \(\mathcal P_D^\mu\).  Therefore
\[
\chi(\mathcal P_D^\mu)\equiv0\pmod2.
\]
The surface is orientable by the same transverse-inverse-image orientation argument used for \(P_D(A)\).  Hence
\[
\#\pi_0(P_{II}(\widehat b^0))
\equiv
\#\pi_0(P_{II}(\widehat b^1))
\pmod2.
\]
Thus \(\mu_D(\alpha_p)\) is independent of the admissible representative.

Changing the collar size only adds or removes a region on which the representative is constant.  By the collar--active and interface-avoidance conditions, no \(D\)-labelled insertion-insertion point lies in the added or removed collar, nor can a component acquire boundary at the new interface.  Identifying the two collar conventions by a rel-boundary collar expansion reduces this case to the preceding parametrized argument.

Changing the insertion ball \(B'\) is handled by choosing a smaller ball contained in the intersection of the two choices, or equivalently by precomposing with an orientation-preserving diffeomorphism of \(3\)-balls fixed near the boundary and then using the preceding homotopy argument.  The quotient \(B'/\partial B'\cong S^3\) and the lifted basepoint are unchanged, so the represented based \(\pi_3\)-class is unchanged.

The larger support ball \(B\) and the collapse map \(c\) do not enter the insertion-insertion fiber product defining \(\mu_D(\alpha_p)\); they only prepare the complement for comparison with a global track.  Any two such preparations are homotopic rel \(\partial B\) through preparations which are constant near the chosen insertion ball.  Therefore changing \(B\) or \(c\) changes neither the based local representative nor the parity of \(P_{II}\).

Since \(\operatorname{im}\mu_D\) is defined as the subgroup generated over all insertion points, transport paths, and transported based classes, the subgroup is well-defined.
\end{proof}

The following lemma describes the effect of an admissible insertion on the raw count.
\begin{lemma}[Effect of an admissible insertion]\label{lem:insertion-effect-D}
Let \(A^{\operatorname{ins}}\) be obtained from a prepared track \(A_{\operatorname{prep}}\) by an admissible insertion representing \(\alpha_p\).  Then
\[
FQ^{0,D}_\rho(A^{\operatorname{ins}})-FQ^{0,D}_\rho(A)
=
\mu_D(\alpha_p)
\]
in \(\Ftwo\langle u_D\rangle\).
\end{lemma}

\begin{proof}
Use the source decomposition from Definition~\ref{def:muD}:
\[
M=B'_{\operatorname{act}}\sqcup M_{\operatorname{prep}},
\qquad
M_{\operatorname{prep}}=M\setminus B'_{\operatorname{act}}.
\]
On \(B'_{\operatorname{act}}\), the inserted track is \(b\).  On \(M_{\operatorname{prep}}\), the inserted track agrees with \(A_{\operatorname{prep}}\); on the constant collar this is because both maps are constant equal to \(A(p)\).

Thus every ordered pair in \(P_D(A^{\operatorname{ins}})\) belongs to exactly one of the following four disjoint regions:
\[
B'_{\operatorname{act}}\times B'_{\operatorname{act}},
\quad
B'_{\operatorname{act}}\times M_{\operatorname{prep}},
\quad
M_{\operatorname{prep}}\times B'_{\operatorname{act}},
\quad
M_{\operatorname{prep}}\times M_{\operatorname{prep}}.
\]
Consequently
\[
P_D(A^{\operatorname{ins}})
=
P_{II}\sqcup P_{IP}\sqcup P_{PI}\sqcup P_{PP},
\]
with the four pieces defined in Definition~\ref{def:muD}.  The pieces are disjoint by construction and exhaustive by the source decomposition.

The insertion-insertion piece \(P_{II}\) is exactly the local fiber product used in the definition of \(\mu_D(\alpha_p)\).  Therefore its contribution is
\[
\bigl(\#\pi_0(P_{II})\bmod2\bigr)u_D=\mu_D(\alpha_p).
\]

The prepared-prepared piece is the old fiber product.  Outside \(B\), \(A_{\operatorname{prep}}=A\).  Suppose a prepared-prepared \(D\)-labelled point has at least one source coordinate in \(B\).  Replace each coordinate lying in \(B\) by its image under the collapse map \(c\).  Since \(A_{\operatorname{prep}}=A\circ c\) on \(B\) and \(A_{\operatorname{prep}}=A\) outside \(B\), this gives an equality for the old track \(A\) with at least one source coordinate in \(c(B)\subset B\), unless the two prepared lifted points are actually the same lifted point.  If the two prepared lifted points are the same, the ordered label is the trivial double coset \(H\), not \(D\).  Otherwise one obtains an old \(D\)-labelled point with a source coordinate in \(B\), contradicting the choice of \(B\).  Therefore \(P_{PP}\) has no point with either source coordinate in \(B\).  On the complement of \(B\), the prepared track is literally equal to \(A\), so \(P_{PP}=P_D(A)\) canonically and contributes the old raw count \(FQ^{0,D}_\rho(A)\).

The mixed pieces \(P_{IP}\) and \(P_{PI}\) are compact smooth one-manifolds without boundary.  Smoothness follows from the mixed transversality assumptions.  Compactness follows because any limit point on the artificial interface \(S\) is excluded by the no-interface equations in Definition~\ref{def:muD}; any limit point on the constant collar would be a collar--active \(D\)-labelled equality, also excluded; endpoint source-boundary points are impossible because the insertion is supported in the interior and its image lies in an interior time slab; target-boundary points are excluded by Convention~\ref{conv:interior-target}.

Ordered-pair exchange sends \(P_{IP}\) homeomorphically to \(P_{PI}\).  Downstairs equality is symmetric, and Lemma~\ref{lem:label-well-defined} gives
\[
\ell(y,x)=\ell(x,y)^{-1}.
\]
Since \(D=D^{-1}\), the exchanged pair is again \(D\)-labelled.  Hence
\[
\#\pi_0(P_{IP})=\#\pi_0(P_{PI}),
\]
so the total mixed contribution is zero in \(\Ftwo\langle u_D\rangle\).

Adding the four disjoint contributions gives
\[
FQ^{0,D}_\rho(A^{\operatorname{ins}})
=
FQ^{0,D}_\rho(A)+\mu_D(\alpha_p),
\]
which is the asserted formula over \(\Ftwo\).
\end{proof}

The following lemma shows that this indeterminacy vanishes when \(D\)
contains no element of order two.
\begin{lemma}\label{lem:muD-zero}
If
\(
D\cap\{\gamma\in\pi:\gamma^2=1\}=\varnothing,
\)
then
\(
\operatorname{im}\mu_D=0.
\)
\end{lemma}

\begin{proof}
It is enough to prove that every admissible local representative has zero \(D\)-count.  Let
\(
b_{\operatorname{loc}}:B'\to X\times I
\)
be an admissible local insertion map, and let
\(
\widehat b_{\operatorname{loc}}:B'\to X_H\times I
\)
be its preferred lift.  Put
\[
f_X=\operatorname{pr}_X\circ b_{\operatorname{loc}},
\qquad
\widehat f_X=\operatorname{pr}_{X_H}\circ \widehat b_{\operatorname{loc}}.
\]

The exchange involution sends \(P_{II}(\widehat b_{\operatorname{loc}})\) to itself because \(D=D^{-1}\).  Components exchanged in distinct pairs cancel modulo two.  Suppose, for contradiction, that a connected component
\(
C\subset P_{II}(\widehat b_{\operatorname{loc}})
\)
is invariant under exchange.  Since \(C\) is a compact one-manifold without boundary, it is a circle.  Choose a point
\(
(x_*,y_*)\in C
\)
and a path
\[
\delta(t)=(x(t),y(t)),\qquad 0\le t\le1,
\]
in \(C\) from \((x_*,y_*)\) to \((y_*,x_*)\). Set
\[
z=f_X(x_*)=f_X(y_*),
\qquad
\widehat p=\widehat f_X(x_*),
\qquad
\widehat q=\widehat f_X(y_*).
\]
Let
\(
r(t)=f_X(x(t)).
\) Since \(\delta(t)\in P_{II}\), we have
\(
b_{\operatorname{loc}}(x(t))=b_{\operatorname{loc}}(y(t))
\)
in \(X\times I\), and hence
\(
f_X(x(t))=f_X(y(t)).
\)
Thus \(r\) is a loop in \(X\) based at \(z\).

Choose a base path \(\sigma\) in \(X\) from \(x_0\) to \(z\) whose lift to \(X_H\), starting at \(\widehat x_0\), ends at \(\widehat p\).  Define
\[
\gamma=[\sigma\cdot r\cdot\overline\sigma]\in\pi.
\]
The path \(\widehat f_X(x(t))\) is a lift of \(r(t)\) starting at \(\widehat p\) and ending at \(\widehat q\).  Therefore, if the label of \((\widehat p,\widehat q)\) is computed using the path \(\sigma\) to \(\widehat p\) and the path \(\sigma*r\) to \(\widehat q\), one obtains
\[
\ell(\widehat p,\widehat q)
=
H[\sigma\,\overline{\sigma*r}]H
=
H\gamma^{-1}H.
\]
Since \((x_*,y_*)\in P_{II}\), this label is \(D\).  Because \(D=D^{-1}\), it follows that
\(
\gamma\in D.
\)

Now consider the two factor projection paths
\(
u(t)=x(t)\) and \(v(t)=y(t).
\) The path \(u\) runs from \(x_*\) to \(y_*\), and \(v\) runs from \(y_*\) to \(x_*\).  Hence \(u*v\) is a loop in \(B'\) based at \(x_*\).  The ball \(B'\) is simply connected, so \(u*v\) is nullhomotopic in \(B'\).  Its image under \(f_X\) is
\(
r*r,
\)
because \(f_X(y(t))=f_X(x(t))=r(t)\) for every \(t\).  Therefore \(r*r\) is nullhomotopic in \(X\), and consequently
\[
\gamma^2
=
[\sigma\cdot r\cdot r\cdot\overline\sigma]
=
1.
\]
We have produced an element \(\gamma\in D\) with \(\gamma^2=1\), contradicting the hypothesis.

Thus no component of \(P_{II}(\widehat b_{\operatorname{loc}})\) is invariant under exchange.  All components are paired with distinct exchanged components, so the mod-two local count is zero.  Hence every generator of \(\operatorname{im}\mu_D\) is zero, and \(\operatorname{im}\mu_D=0\).
\end{proof}

\subsection{Marked mapping-space loop indeterminacy}\label{sec:theta}

In this section, we define the second indeterminacy: concatenation with a loop in the marked mapping-space component.

\begin{definition}[Marked mapping-space component]\label{def:M-rho}
Fix \((F_0,\lambda_0,\rho)\) and \(\widehat x_0\in X_H\).  Let \(\mathcal M_\rho(F_0)\) be the path component containing \((F_0,\lambda_0)\) of the space of pairs \((f,\lambda)\), where \(f:\Sigma_g\to X\), \(\lambda:x_0\to f(z_0)\), and
\(
f_*^\lambda=i_H\circ\rho.
\)
Each pair determines a preferred lift \(\widehat f\) by lifting \(\lambda\) from \(\widehat x_0\) and then applying the covering lifting criterion to \(f\).  A loop \(\Lambda\) represented by \((f_s,\lambda_s)\), \(s\in S^1\), gives a level-preserving loop-track
\[
A_\Lambda:\Sigma_g\times S^1\to X\times S^1,
\qquad
A_\Lambda(z,s)=(f_s(z),s),
\]
with preferred lift \(\widehat A_\Lambda(z,s)=(\widehat f_s(z),s)\).  Because the base path is part of the loop, the preferred lift returns to the same sheet after one circuit of \(S^1\).
\end{definition}
We give the map and path spaces in Definition~\ref{def:M-rho} their compactly generated compact-open topologies.  All homotopies used below are taken in these spaces; before a labelled double-locus count is made, relative smooth approximation and general position are applied.

\begin{definition}\label{def:ThetaD}
For a generic loop-track \(A_\Lambda\), define
\[
\Theta^0_{\rho,D}(\Lambda)=
\bigl(\#\pi_0(P_D(A_\Lambda))\bmod2\bigr)u_D
\in\Ftwo\langle u_D\rangle.
\]
Let \(R_{\Theta,D}\subset\Ftwo\langle u_D\rangle\) be the subgroup generated by all raw loop contributions \(\Theta^0_{\rho,D}(\Lambda)\), as \(\Lambda\) ranges over generic marked mapping-space loops.
\end{definition}

\begin{lemma}\label{lem:ThetaD-well-defined}
The raw loop contribution \(\Theta^0_{\rho,D}\) is additive under concatenation and is invariant under homotopy of marked loops modulo \(\operatorname{im}\mu_D\).  Hence there is a well-defined homomorphism
\[
\Theta_{\rho,D}:\pi_1(\mathcal M_\rho(F_0),F_0)
\to
\Ftwo\langle u_D\rangle/\operatorname{im}\mu_D.
\]
The image of \(R_{\Theta,D}\) in the quotient is \(\operatorname{im}\Theta_{\rho,D}\).
\end{lemma}

\begin{proof}
For additivity, concatenate two generic loop-tracks with stationary collars at the base pair.  The resulting loop-track is level-preserving on each parameter arc.  Equality in \(X\times S^1\) forces equality of the \(S^1\)-parameter, so no double point has one source coordinate in the first loop arc and the other in the second loop arc.  Stationary collars contribute no nontrivial \(D\)-components because \(F_0\) is embedded and any endpoint equality has the trivial double coset.  Thus raw counts add in \(\Ftwo\langle u_D\rangle\).

For homotopy invariance, take a generic homotopy of loops in \(\mathcal M_\rho(F_0)\).  The parametrized \(D\)-labelled fiber product is a compact orientable surface, and the proof of Lemma~\ref{lem:relative-invariance-D} applies with \(M=\Sigma_g\times S^1\).  The boundary is the union of the two endpoint loop fiber products, so the two endpoint counts agree modulo two.

If a representative is changed by an admissible local insertion, Lemma~\ref{lem:insertion-effect-D} says that the raw count changes by \(\mu_D(\alpha_p)\).  A finite sequence of such insertions changes the raw count by an element of \(\operatorname{im}\mu_D\).  Hence \(\Theta_{\rho,D}\) is well-defined in the quotient, and the statement about \(R_{\Theta,D}\) follows from its definition as the subgroup generated by raw loop contributions.
\end{proof}

\subsection{The absolute marked obstruction}\label{sec:absolute}

The raw count depends on the chosen track.  In this section, we prove that
the difference between two marked tracks with the same endpoints lies in the
subgroup generated by local insertions and marked mapping-space loops.  We
then define the quotient target and the absolute marked obstruction whose
nonzero values obstruct marked concordance.

We use the following group notations.  The centre of \(H\) is
\[
Z(H)=\{h\in H:hk=kh\text{ for all }k\in H\},
\]
and the centralizer of \(H\) in \(\pi\) is
\[
C_\pi(H)=\{\gamma\in\pi:\gamma h\gamma^{-1}=h\text{ for all }h\in H\}.
\]

\begin{lemma}[Absolute path-difference]\label{lem:path-difference-D}
Let \(A_0,A_1:\Sigma_g\times I\to X\times I\) be two generic
\(\rho\)-marked tracks from \(F_0\) to \(F_1\).  Then
\[
FQ^{0,D}_\rho(A_0)-FQ^{0,D}_\rho(A_1)
\in
R_{\Theta,D}+\operatorname{im}\mu_D.
\]
\end{lemma}

\begin{proof}
By Lemma~\ref{lem:level-straightening-D}, straighten both tracks to generic
level-preserving marked tracks without changing the raw counts.  Write
\[
A_i^{\operatorname{lev}}(z,s)=(a_i(z,s),s).
\]
Each straightened track determines a path \(\Gamma_i\) in
\(\mathcal M_\rho(F_0)\).  Namely, set \(f_{i,s}(z)=a_i(z,s)\), and let
\(\lambda_{i,s}\) be the fixed bottom base path \(\lambda_0\) followed by
the path traced by the source basepoint under the track from time \(0\) to
\(s\).  The preferred lift of \((f_{i,s},\lambda_{i,s})\) is the slice of
\(\widehat A_i^{\operatorname{lev}}\).  Thus \(\Gamma_i\) begins at
\((F_0,\lambda_0)\) and ends at \((F_1,\lambda_i^{\operatorname{top}})\),
where \(\lambda_i^{\operatorname{top}}=\lambda_{i,1}\).

Because both tracks are marked with the same prescribed top lift
\(\widehat F_1\), the lifts of \(\lambda_0^{\operatorname{top}}\) and
\(\lambda_1^{\operatorname{top}}\) starting at \(\widehat x_0\) end at the
same point \(\widehat F_1(z_0)\).  Hence
\[
h=[\lambda_0^{\operatorname{top}}\overline{\lambda_1^{\operatorname{top}}}]
\in H.
\]
Both endpoint pairs induce the same based marking, so for every
\(\alpha\in\pi_1(\Sigma_g,z_0)\),
\[
[\lambda_0^{\operatorname{top}}F_1(\alpha)\overline{\lambda_0^{\operatorname{top}}}]
=
[\lambda_1^{\operatorname{top}}F_1(\alpha)\overline{\lambda_1^{\operatorname{top}}}]
=\rho(\alpha).
\]
With the convention above for \(h\), this implies
\(h\rho(\alpha)h^{-1}=\rho(\alpha)\).  Therefore \(h\in C_\pi(H)\).  Since
also \(h\in H\), we have \(h\in Z(H)\).

We now join the two top endpoint pairs inside \(\mathcal M_\rho(F_0)\) by a
connector \(E_h\) whose preferred lifted image lies in
\(\widehat F_1(\Sigma_g)\).  If \(g\ge2\), then
\(H\cong\pi_1\Sigma_g\) has trivial centre, so \(h=1\).  The two top base
paths are homotopic rel endpoints; take the map coordinate fixed equal to
\(F_1\) and vary only the base path.

If \(g=1\), choose a based loop \(\beta\subset\Sigma_1\) representing
\(\rho^{-1}(h^{-1})\).  Since \((F_1,\lambda_0^{\operatorname{top}})\)
induces \(\rho\), the path \(\lambda_0^{\operatorname{top}}F_1(\beta)\) is
homotopic rel endpoints to \(\lambda_1^{\operatorname{top}}\).  Identify
\(\Sigma_1\) with \(\mathbb R^2/\mathbb Z^2\), and let \(T_t\) be the path
of translations with \(T_0=T_1=\operatorname{id}\) whose basepoint trace is
\(\beta\).  The path of pairs
\[
\left(F_1\circ T_t,
 \lambda_0^{\operatorname{top}}F_1(\beta|_{[0,t]})\right)
\]
followed by the chosen rel-endpoint homotopy of base paths gives the
required connector.  Its preferred lift is \(\widehat F_1\circ T_t\).

In both cases the map coordinate of the connector has embedded image in
\(F_1(\Sigma_g)\), so the connector has no off-diagonal double points and
therefore contributes no \(D\)-labelled components.

The concatenation
\[
\Lambda=\Gamma_0*E_h*\overline{\Gamma_1}
\]
is a loop in \(\mathcal M_\rho(F_0)\).  Insert stationary collars between
the three stages and parametrize the loop-track so that it is
level-preserving on each stage.  Equality in \(X\times S^1\) forces equality
of the parameter, so no mixed components occur between the stages.  The
connector and stationary collars contribute zero.  Reversing \(\Gamma_1\)
reverses the parameterization of its fiber product but not the number of
connected components.  Hence
\[
\Theta^0_{\rho,D}(\Lambda)=
FQ^{0,D}_\rho(A_0^{\operatorname{lev}})+FQ^{0,D}_\rho(A_1^{\operatorname{lev}})
\]
in \(\Ftwo\langle u_D\rangle\).  Since subtraction is addition over
\(\Ftwo\), this is the required difference.  The loop contribution lies in
\(R_{\Theta,D}\), and the straightening steps changed no raw counts by
Lemma~\ref{lem:relative-invariance-D}.  The conclusion follows.
\end{proof}

We are now ready to define the obstruction.
\begin{definition}\label{def:QD}
Define
\[
Q^D_\rho(X,F_0)=
\Ftwo\langle u_D\rangle/
\bigl(\operatorname{im}\mu_D+R_{\Theta,D}\bigr).
\]
For any generic \(\rho\)-marked track \(A\) from \(F_0\) to \(F_1\), define
\[
FQ^D_\rho(F_0,F_1)=
[FQ^{0,D}_\rho(A)]
\in Q^D_\rho(X,F_0).
\]
Lemma~\ref{lem:path-difference-D} implies that this class is independent of
\(A\).
\end{definition}

\medskip
\noindent\textbf{Concatenation and reversal of marked tracks.}
Let \(F_0,F_1,F_2\) be \(\rho\)-marked surfaces with the fixed subgroup
\(H\), fixed endpoint base paths, and preferred lifts.  Let
\(A_{01}\) be a generic \(\rho\)-marked track from \(F_0\) to \(F_1\),
and let \(A_{12}\) be a generic \(\rho\)-marked track from \(F_1\) to
\(F_2\).  After applying Lemma~\ref{lem:level-straightening-D}, reparametrize
the two tracks so that they have stationary product collars at \(F_1\), and
form the concatenation \(A_{02}=A_{01}*A_{12}\).  Then
\[
   P_D(A_{02})\cong P_D(A_{01})\sqcup P_D(A_{12})
\]
by the time-rescaling identifications, and hence
\[
   FQ^{0,D}_\rho(A_{02})
   =FQ^{0,D}_\rho(A_{01})+FQ^{0,D}_\rho(A_{12}).
\]
Indeed, equality in \(X\times I\) forces equality of the target time
coordinates.  Thus a double point cannot have one source point in the
interior of the first stage and the other in the interior of the second
stage.  On the stationary junction collar, the map coordinate is the
embedding \(F_1\), so every equality is diagonal and has the trivial label
\(H\).  The concatenated preferred lift is obtained by gluing the preferred
lifts of the two tracks along their common endpoint \(\widehat F_1\).
Consequently the labels on the second stage are exactly the labels computed
for \(A_{12}\); no endpoint-basepath action is introduced.  The
reparametrizations and the smoothing of the junction are supported in the
stationary collar, so they create no additional \(D\)-labelled component.
This proves the displayed decomposition and equality.

If \(A_{01}(z,s)=(a_{01}(z,s),s)\) is level-preserving, define its reverse
by
\[
   \overline A_{01}(z,s)=(a_{01}(z,1-s),s).
\]
Its preferred lift is obtained by the same reversal of the preferred lift of
\(A_{01}\).  Reversal identifies the two ordered double loci by reversing
the common time coordinate and preserves every double-coset label.  Hence
\[
   FQ^{0,D}_\rho(\overline A_{01})
   =FQ^{0,D}_\rho(A_{01}).
\]
Only the parametrization and, in a signed theory, the orientation of a
component are reversed; the present count has coefficients in \(\Ftwo\), so
no sign remains.

In particular, if \(C_{12}\) is a \(\rho\)-marked embedded concordance from
\(F_1\) to \(F_2\), then \(P_D(C_{12})=\varnothing\), and the
concatenation formula above, followed by Definition~\ref{def:QD}, gives the
equality
\[
   FQ^D_\rho(F_0,F_2)=FQ^D_\rho(F_0,F_1)
   \qquad\text{in }Q^D_\rho(X,F_0).
\]
This is the fixed-target statement used below.  In general
\(FQ^D_\rho(F_1,F_2)\) is defined in
\(Q^D_\rho(X,F_1)\), and no canonical identification of that target with
\(Q^D_\rho(X,F_0)\) is asserted here.

\begin{theorem}\label{thm:A-marked-obstruction}
Let \(F_0\) be a \(\rho\)-marked \(\pi_1\)-injective surface with image
subgroup \(H\leq\pi\), and let \(D=HgH\neq H\) satisfy \(D=D^{-1}\).  Then
\(FQ^D_\rho(F_0,F_1)\in Q^D_\rho(X,F_0)\) is a well-defined marked
concordance obstruction.  If \(F_0\) and \(F_1\) are \(\rho\)-marked
concordant, then
\(
FQ^D_\rho(F_0,F_1)=0.
\)
\end{theorem}

\begin{proof}
Well-definedness is Definition~\ref{def:QD}.  If \(C\) is a
\(\rho\)-marked embedded concordance, then there are no off-diagonal
self-intersections downstairs.  Hence \(P_D(C)=\varnothing\), the raw count
of \(C\) is zero, and the class of \(FQ^D_\rho(F_0,F_1)\) in
\(Q^D_\rho(X,F_0)\) is zero.
\end{proof}

\subsection{Cellular computation of \texorpdfstring{\(R_{\Theta,D}\)}{RThetaD}}\label{sec:cellular}

We compute the marked mapping-space-loop indeterminacy by first fixing the
coefficient conventions.  Put
\[
   G=\pi_1(\Sigma_g,z_0),
   \qquad
   \rho_0=i_H\circ\rho:G\longrightarrow\pi.
\]
For \(k\ge2\), let \(M_k=\pi_k(X,x_0)\), regarded as a left
\(\mathbb Z[\pi]\)-module by the standard change-of-basepoint action.  More
explicitly, choose a universal cover \(\widetilde X\to X\), a point
\(\widetilde x_0\) over \(x_0\), and, for \(\gamma\in\pi\), the lift
\(\widetilde\gamma\) of a representative loop beginning at
\(\widetilde x_0\).  If \(T_\gamma\) is the corresponding deck
transformation, then \(\gamma\cdot m\) is represented by
\(T_\gamma\) applied to a representative of \(m\), followed by change of
basepoint along \(\overline{\widetilde\gamma}\).  We extend \(\rho_0\) linearly to group rings and restrict this left
module along \(\rho_0\).

Let \(p_\Sigma:\widetilde\Sigma_g\to\Sigma_g\) be the universal cover,
with its left \(G\)-action.  We use the cellular cochain convention
\[
   C^j(\Sigma_g;(M_k)_\rho)
   =\operatorname{Hom}_{\mathbb Z[G]}
      \bigl(C_j(\widetilde\Sigma_g),M_k\bigr).
\]
Choose the standard CW structure with one vertex \(v=z_0\), oriented
one-cells
\[
   a_1,b_1,\ldots,a_g,b_g,
\]
and one two-cell attached by
\[
   r=\prod_{i=1}^g[a_i,b_i],
   \qquad [a,b]=aba^{-1}b^{-1}.
\]
Choose a lift \(\widetilde v\), a lift \(\widetilde e_x\) of each
one-cell running from \(\widetilde v\) to \(x\widetilde v\), and the lift
\(\widetilde e^2\) whose attaching word starts at \(\widetilde v\).  Then
\[
   \partial\widetilde e_x=x\widetilde v-\widetilde v,
   \qquad
   \partial\widetilde e^2
   =\sum_x\frac{\partial r}{\partial x}\,\widetilde e_x,
\]
where the second formula uses the left Fox derivatives.  Thus, for
\(m\in C^0(\Sigma_g;(M_k)_\rho)=M_k\),
\begin{equation}\label{eq:cellular-delta-zero}
   (\delta^0m)(e_x)=\rho_0(x)\cdot m-m,
\end{equation}
and, for \(c\in C^1(\Sigma_g;(M_k)_\rho)\),
\begin{equation}\label{eq:cellular-delta-one}
   (\delta^1c)(e^2)
   =\sum_x \rho_0\!\left(\frac{\partial r}{\partial x}\right)\cdot c(e_x).
\end{equation}
Writing \(q_i=\prod_{j<i}[a_j,b_j]\), the last formula is
\begin{align}
(\delta^1c)(e^2)
={}&\sum_{i=1}^g \rho_0(q_i)\cdot
\bigl(
 c(e_{a_i})
 -\rho_0(a_ib_ia_i^{-1})\cdot c(e_{a_i}) \notag\\
&\hspace{36mm}
 +\rho_0(a_i)\cdot c(e_{b_i})
 -\rho_0([a_i,b_i])\cdot c(e_{b_i})
\bigr).
\label{eq:surface-cellular-coboundary}
\end{align}
These conventions define the local system denoted below by
\((\pi_kX)_\rho\).

\begin{lemma}\label{lem:cellular-normalization-D}
Let
\[
   \Lambda\in
   \pi_1\bigl(\mathcal M_\rho(F_0),(F_0,\lambda_0)\bigr).
\]
The based evaluation loop is nullhomotopic.  A choice of evaluation
normalization determines a cellular cocycle
\[
   c_\Lambda\in Z^1(\Sigma_g;(\pi_2X)_\rho).
\]
Changing the evaluation normalization changes \(c_\Lambda\) by a cellular
coboundary.  Hence
\[
   \kappa_1(\Lambda)=[c_\Lambda]
   \in H^1(\Sigma_g;(\pi_2X)_\rho)
\]
is an invariant of the based marked loop.  If \(\kappa_1(\Lambda)=0\),
then \(\Lambda\) is homotopic through based loops in
\(\mathcal M_\rho(F_0)\) to a loop \((f_s,\lambda_s)\) for which
\[
   f_s|_N=F_0|_N
   \qquad\text{for every }s\in S^1,
\]
where \(N\) is a closed regular neighbourhood of the one-skeleton and
contains \(z_0\).
\end{lemma}

\begin{proof}
Represent \(\Lambda\) by \((f_s,\lambda_s)\), with
\(s\in S^1\), based at \((F_0,\lambda_0)\), and put
\(e(s)=f_s(z_0)\).  After cutting \(S^1\) at its basepoint, the map
\[
   (u,s)\longmapsto\lambda_s(u)
\]
is a square whose boundary is
\(\lambda_0e\overline{\lambda_0}\), with its opposite side constant at
\(x_0\).  Thus the based evaluation element is trivial.

Let \(P_{x_0}X\) be the space of paths beginning at \(x_0\), and let
\[
   P_{x_0}^{\operatorname{int}}X
   =\{\lambda\in P_{x_0}X:\lambda(1)\in\operatorname{int}X\}.
\]
Since \(\{z_0\}\hookrightarrow\Sigma_g\) is a cofibration, the restriction
map
\[
   \operatorname{ev}_{z_0}:
   \operatorname{Map}(\Sigma_g,\operatorname{int}X)
   \longrightarrow\operatorname{int}X
\]
is a Hurewicz fibration in the compactly generated compact-open category.
The endpoint map
\(\operatorname{end}_1:P_{x_0}^{\operatorname{int}}X
\to\operatorname{int}X\) is the pullback of the usual path fibration.
Consequently the projection
\[
   \mathcal P^{\operatorname{int}}=
   \operatorname{Map}(\Sigma_g,\operatorname{int}X)
   \times_{\operatorname{int}X}P_{x_0}^{\operatorname{int}}X
   \longrightarrow\operatorname{int}X
\]
from the pullback space of pairs is a Hurewicz fibration.  The square above
shows that \(e\) is nullhomotopic as a loop based at
\(y_0=F_0(z_0)\).  If \(X\) has boundary, choose the nullhomotopy in
\(\operatorname{int}X\), using a collar-pushing deformation which is fixed
near the compact loop \(e\).  Choose it constant on
\(\{s_0\}\times I\), and lift it through
\(\mathcal P^{\operatorname{int}}\to\operatorname{int}X\), relative to
\[
   (S^1\times\{0\})\cup(\{s_0\}\times I).
\]
The two lifted coordinates have the same projected endpoint, so the lifted
homotopy remains in the pullback and produces a based loop of pairs, still
denoted \((f_s,\lambda_s)\), with
\[
   f_s(z_0)=y_0
   \qquad\text{for all }s.
\]
A path of pairs preserves the actual induced homomorphism: the homotopy
cylinder for a based source loop gives the moving-basepoint identity.
Hence the lift stays in the subspace on which
\(f_*^\lambda=\rho_0\); no centralizer condition is involved.  Finally,
each normalized \(\lambda_s\) lies in the same component of the fixed-endpoint
path space \(P_{x_0,y_0}X\) as \(\lambda_0\).  Thus
\(\lambda_s\simeq\lambda_0\) rel endpoints and
\[
   (f_s)_*^{\lambda_0}=\rho_0
\]
for every \(s\).

Choose a universal cover \(\widetilde X\to X\), lift \(\lambda_0\) from
\(\widetilde x_0\), and denote its endpoint by \(\widetilde y_0\).  The
maps \(f_s\) have unique continuous lifts
\[
   \widetilde f_s:\widetilde\Sigma_g\longrightarrow\widetilde X,
   \qquad
   \widetilde f_s(\widetilde v)=\widetilde y_0,
\]
and these lifts are \(\rho_0\)-equivariant.  For a generator
\(x\in\{a_1,b_1,\ldots,a_g,b_g\}\), the restrictions
\(\widetilde f_s|_{\widetilde e_x}\) form a loop in the path space from
\(\widetilde y_0\) to \(\rho_0(x)\widetilde y_0\), based at
\(\widetilde F_0|_{\widetilde e_x}\).  Equivalently, collapse the two
boundary circles of
\(\widetilde e_x\times S^1\) separately to the two poles of a sphere.  The
resulting based map of \(S^2\), based at the pole corresponding to
\(\widetilde v\), determines an element of
\(\pi_2(\widetilde X,\widetilde y_0)\), and transport along
\(\lambda_0\) identifies it with an element
\[
   c_\Lambda(e_x)\in M_2=\pi_2(X,x_0).
\]
Equivariance determines a cellular one-cochain.  With the chosen lifts,
reversing the orientation of an edge gives
\begin{equation}\label{eq:edge-reversal-cocycle}
   c_\Lambda(\overline e_x)
   =-\rho_0(x)^{-1}\cdot c_\Lambda(e_x).
\end{equation}
Replacing a chosen lift by its \(g\)-translate multiplies its displayed
coefficient by \(\rho_0(g)\); this is exactly the change of cellular basis
for an equivariant cochain.

We verify the cocycle equation.  Lift the characteristic map of the
chosen two-cell and, for every \(s\), restrict \(\widetilde f_s\) to this
lift.  This is a loop of based maps
\[
   (D^2,*)\longrightarrow(\widetilde X,\widetilde y_0).
\]
The based mapping space of \((D^2,*)\) is contractible, because
\((D^2,*)\) contracts to its basepoint.  Its boundary restriction is
therefore nullhomotopic as a loop in
\(\operatorname{Map}_*(S^1,\widetilde X)\).  Under
\[
   \pi_1\operatorname{Map}_*(S^1,\widetilde X)
   \cong\pi_2(\widetilde X,\widetilde y_0),
\]
the boundary loop is obtained by concatenating the translated and
orientation-reversed edge loops in the attaching word
\(r=\prod_i[a_i,b_i]\).  Formula~\eqref{eq:edge-reversal-cocycle} gives
precisely the sum in \eqref{eq:surface-cellular-coboundary}.  Consequently
\[
   (\delta^1c_\Lambda)(e^2)=0,
\]
so \(c_\Lambda\) is a cocycle.

It remains to identify the choice dependence.  Let two evaluation
normalizations, including their lifts through the pullback fibration, produce
cocycles \(c_\Lambda\) and \(c_\Lambda'\).  Concatenate the reverse of the
first normalization with the second.  Its evaluation on
\(\{z_0\}\times S^1\times I\) is constant on
\[
   (S^1\times\partial I)\cup(\{s_0\}\times I),
\]
so it descends to a based sphere.  Orient this sphere so that the second
normalization minus the first represents \(m\in M_2\).  Applying the edge
construction to \(e_x\times S^1\times I\), the change at the terminal vertex
is the translate \(\rho_0(x)\cdot m\), while the change at the initial
vertex is \(m\) with the opposite orientation.  Therefore
\[
   c_\Lambda'(e_x)-c_\Lambda(e_x)
   =\rho_0(x)\cdot m-m
   =(\delta^0m)(e_x)
\]
by \eqref{eq:cellular-delta-zero}.  This argument also includes the
nonuniqueness of the fibration lifts.  Homotopic choices of \(\lambda_0\) rel
endpoints do not change the coefficients.  The cell-lift changes described
above are only changes of equivariant cellular basis.  Applying the same
prism calculation to a homotopy of marked loops shows that \([c_\Lambda]\)
depends only on the element of
\(\pi_1(\mathcal M_\rho(F_0),(F_0,\lambda_0))\).

Suppose now that \(\kappa_1(\Lambda)=0\).  Since there is one vertex, there
is an element \(m\in M_2=C^0(\Sigma_g;(\pi_2X)_\rho)\) with
\(c_\Lambda=\delta^0m\).  Modify the chosen evaluation normalization by
the sphere \(-m\).  Geometrically, this means gluing a representative of
\(-m\) into a small disk in the interior of the normalization square and
lifting the modified normalization through the pullback fibration, relative
to its boundary.  The boundary is unchanged, so the final evaluation loop is
still constant.  The moving-basepoint identity keeps the actual marking
\(\rho_0\) fixed.  The preceding formula shows that this single modification
changes all incident edge coefficients simultaneously and makes the new edge
cocycle zero.

For every one-cell, the corresponding loop in its endpoint-fixed path
space is now nullhomotopic.  Choose these nullhomotopies rel the basepoint
of the parameter circle.  Because they are fixed at the common vertex, they
combine to a homotopy of the restriction to the wedge
\(\Sigma_g^{(1)}\), through maps carrying every generator to the same
relative homotopy class as \(F_0\).  Thus the induced homomorphism remains
\(\rho_0\).  Since the inclusion of the one-skeleton is a cofibration, the
homotopy extension property extends this to a homotopy of the full loop of
maps; retain the accompanying path coordinate.  The resulting homotopy is
through based loops in \(\mathcal M_\rho(F_0)\) and is literally stationary
on \(\Sigma_g^{(1)}\).

Finally choose a closed regular neighbourhood \(N\) of the one-skeleton.
The inclusion
\[
   \bigl(\Sigma_g^{(1)}\times S^1\bigr)
   \cup\bigl(N\times\{s_0\}\bigr)
   \longrightarrow N\times S^1
\]
is a cofibration and a homotopy equivalence: its cofiber is the smash
product of the cofiber of
\(\Sigma_g^{(1)}\hookrightarrow N\) with \(S^1/\{s_0\}\), and the first
cofiber is contractible.  The loop map and the stationary map agree on the
left-hand subspace.  Since a cofibration which is a homotopy equivalence is
a strong deformation retract, precomposition with such a deformation
retraction shows that the two maps are homotopic rel that subspace on
\(N\times S^1\).  A second application of homotopy extension gives a based
marked loop which is literally equal to \(F_0\) on \(N\) for every
parameter value.
\end{proof}

\begin{lemma}[Relative extensions and local insertions]
\label{lem:relative-extension-local-D}
Let \(C\cong D^2\) be the closure of
\(\Sigma_g\setminus\operatorname{Int}N\), put \(B=C\times I\cong D^3\),
and let
\[
   U_0(z,s)=F_0(z).
\]
Suppose that \(U:B\to X\) agrees with \(U_0\) on \(\partial B\).  Choose
\(q_B\in\partial B\), put \(y_*=U_0(q_B)\), and form the oriented double
\[
   S^3=B_+\cup_{\operatorname{id}_{\partial B}}\overline{B_-},
\]
where the second copy has the opposite orientation.  The map
\begin{equation}\label{eq:relative-difference-map}
   \Delta(U,U_0)=U\cup_{\partial B}U_0:S^3\longrightarrow X
\end{equation}
is based at \(y_*\) and defines a class
\[
   \alpha(U,U_0)=[\Delta(U,U_0)]\in\pi_3(X,y_*).
\]
The following statements hold.
\begin{enumerate}[label=\textup{(\roman*)}]
\item For two extensions \(U_1,U_2\) with the same boundary values, define
\(\alpha(U_1,U_2)\) by the same oriented-double construction.  They are
homotopic rel \(\partial B\) if and only if
\(\alpha(U_1,U_2)=0\).  Relative
homotopy classes of extensions of \(U_0|_{\partial B}\) form a torsor for
\(\pi_3(X,y_*)\), with origin \(U_0\).  In particular, \(U\) is homotopic
rel \(\partial B\) to an interior local insertion representing
\(\alpha(U,U_0)\).
\item Every transported based class used in Definition~\ref{def:muD} has an
admissible representative.
\item Let a based marked loop be stationary on \(N\), and let its residual
map on \(B\) be \(U\).  After transporting \(\alpha(U,U_0)\) to an
interior insertion point \(p\), there is an admissible insertion
representing a class \(\alpha_p\) such that
\[
   \Theta^0_{\rho,D}(\Lambda)=\mu_D(\alpha_p).
\]
\end{enumerate}
\end{lemma}

\begin{proof}
We first prove the zero-difference criterion without collapsing the
nonconstant boundary map.  Let
\(g=U_1|_{\partial B}=U_2|_{\partial B}\), and define on the literal
boundary
\[
   \partial(B\times I)=
   (B\times\{0\})\cup(\partial B\times I)\cup(B\times\{1\})
\]
the map \(\Phi\) which is \(U_1\) on the bottom, \(g\) on the side
cylinder, and \(U_2\) on the top.  The side-cylinder map is constant only in
the \(I\)-direction, not in the \(\partial B\)-direction.  Therefore it
factors through the quotient
\[
   q:\partial(B\times I)\longrightarrow
   B_+\cup_{\partial B}\overline{B_-}
\]
which collapses each fiber \(\{x\}\times I\subset\partial B\times I\) to
the equatorial point \(x\), while leaving the two \(B\)-faces unchanged.
The quotient is the oriented double, and
\[
   \Phi=\Delta(U_1,U_2)\circ q.
\]
With the displayed orientations, \(q:S^3\to S^3\) has degree \(+1\): a
regular value in the interior of \(B_+\) has one positive preimage.
Consequently \([\Phi]=\alpha(U_1,U_2)\) in \(\pi_3(X,y_*)\).  Hence
\(\alpha(U_1,U_2)=0\) if and only if \(\Phi\) extends over \(B\times I\).
Such an extension is precisely a homotopy from \(U_1\) to \(U_2\) which is
fixed pointwise on \(\partial B\).

For \(\beta\in\pi_3(X,y_*)\), choose a small ball
\(W\Subset\operatorname{Int}B\) and an arc from \(q_B\) to \(W\).
Transport \(\beta\) along the image of this arc under \(U_0\), make
\(U_0\) constant on a smaller ball by precomposition with a collapse map
supported in \(W\), and insert the transported based representative there.
Denote the result by \(U_0\#\beta\).  The gluing orientations give
\[
   \alpha(U_0\#\beta,U_0)=\beta.
\]
Cutting and regluing three oriented copies of \(B\) gives the separation
identity
\[
   \alpha(U_1,U_3)=
   \alpha(U_1,U_2)+\alpha(U_2,U_3).
\]
The zero-difference criterion and this identity show that
\([U]\mapsto\alpha(U,U_0)\) is a bijection from relative homotopy classes of
extensions to \(\pi_3(X,y_*)\).  Equivalently, interior connected sum gives
a free and transitive action with origin \(U_0\).  This proves
\textup{(i)}.

We next prove \textup{(ii)}.  Let \(A\) and \(p\) be insertion data as in
Definition~\ref{def:muD}, and let \(\alpha_p\) be a transported based class.
Choose a sufficiently small time slab
\(J=(t_p-\epsilon,t_p+\epsilon)\).  A collar-pushing deformation of \(X\),
chosen fixed near \(\operatorname{pr}_X A(p)\), together with contraction of
the interval coordinate, shows that
\[
   \operatorname{int}X\times J\longrightarrow X\times I
\]
is a based homotopy equivalence.  Hence choose a based representative
\[
   b:(B',\partial B')\longrightarrow
   \bigl(\operatorname{int}X\times J,A(p)\bigr)
\]
which is constant on a collar.  Lift it from \(\widehat A(p)\).  Relative
transversality to the point \(A(p)\) makes the active part disjoint from
\(A(p)\), while preserving the collar and the based homotopy class, because
its dimension is \(3<5=\dim(X\times I)\).  Choose the interface inside this
constant collar.  Conditions \textup{(L2)} and \textup{(L3)} then hold.

For condition \textup{(L4)}, suppose that a point on the interface, which
maps to \(A(p)\), formed a mixed \(D\)-labelled equality with a point of the
prepared complement.  If the latter point lies outside the preparation
ball, this gives an old \(D\)-labelled pair with one coordinate equal to
\(p\), contrary to the choice of \(p\).  If it lies in the preparation
ball, composing with the collapse map gives the same contradiction unless
both lifted points are the preferred lift of \(A(p)\); in that exceptional
case the label is \(H\), not \(D\).  Thus \textup{(L4)} holds.  Finally,
relative multitransversality, supported in the active part, makes the three
ordered maps in \textup{(L5)} transverse to the diagonal without changing
the preceding avoidance conditions.  This supplies an admissible
representative.  The interface exclusions and transversality are exactly
those used in Lemma~\ref{lem:insertion-effect-D} to obtain compact
one-manifolds without artificial boundary.

For \textup{(iii)}, cut the parameter circle at its basepoint, outside the
chosen support, and choose a closed parameter interval containing the support
in its interior.  On that interval the stationary loop-track is a proper
track with product collars, so Definition~\ref{def:muD} and
Lemma~\ref{lem:insertion-effect-D} apply literally.  Transport
\(\alpha(U,U_0)\) along the stationary image to an insertion point \(p\),
and choose an admissible representative by \textup{(ii)}.  Insert it into
the stationary track after the collapse preparation of
Definition~\ref{def:muD}.  The preparation is homotopic to the identity rel
its boundary, so it has zero separation class; consequently the
\(X\)-component of the inserted residual block has separation class
\(\alpha(U,U_0)\) from \(U_0\).

The inserted track need not initially be level-preserving.  Apply
Lemma~\ref{lem:level-straightening-D} inside the chosen parameter interval.
The straightening leaves the \(X\)-coordinate unchanged and preserves the
product collars.  Hence the level-preserving residual block still satisfies
\[
   \alpha(U^{\operatorname{ins}},U_0)=\alpha(U,U_0).
\]
The separation identity from \textup{(i)} gives
\(\alpha(U^{\operatorname{ins}},U)=0\), so the two residual maps are
homotopic rel \(\partial B\).  Extend this homotopy by the stationary map on
\(N\), retain the original path coordinate \(\lambda_s\), and pair every
intermediate \(X\)-map with the source parameter.  This gives a homotopy of
level-preserving based marked loops.  It is fixed on \(N\), hence at
\(z_0\), so the actual homomorphism \(\rho_0\), the path coordinate, and the
preferred lifts are preserved.  Relative smooth approximation and general
position may be imposed without changing these fixed data.

The raw count is unchanged by this marked-loop homotopy and by level
straightening.  Lemma~\ref{lem:insertion-effect-D} computes the change from
the stationary loop-track as \(\mu_D(\alpha_p)\).  The stationary loop-track
has no nontrivial \(D\)-labelled component, and the mixed components in the
insertion comparison cancel in exchanged pairs, as proved in that lemma.
Hence
\[
   \Theta^0_{\rho,D}(\Lambda)=\mu_D(\alpha_p).
\]

After transporting the difference class to the fixed basepoint \(x_0\), if
the chosen lift of the two-cell is replaced by its \(g\)-translate, the
transported class is replaced by
\(\rho_0(g)\cdot\alpha(U,U_0)\).  A different transport path acts by the
corresponding element of \(\pi\).  Definition~\ref{def:muD} allows all
insertion points, transport paths, and transported classes, so neither
change affects membership in \(\operatorname{im}\mu_D\).
\end{proof}

\begin{lemma}\label{lem:H2-local-D}
If \(\kappa_1(\Lambda)=0\), then
\[
   \Theta^0_{\rho,D}(\Lambda)\in\operatorname{im}\mu_D.
\]
Consequently, if
\(H^1(\Sigma_g;(\pi_2X)_\rho)=0\), then
\begin{equation}\label{eq:RTheta-contained-in-mu}
   R_{\Theta,D}\subseteq\operatorname{im}\mu_D.
\end{equation}
\end{lemma}

\begin{proof}
By Lemma~\ref{lem:cellular-normalization-D}, the loop is homotopic through
based marked loops to one which is stationary on a regular neighbourhood
\(N\) of the one-skeleton.  Cutting the parameter circle at its basepoint,
the residual map on
\(B=\overline{\Sigma_g\setminus\operatorname{Int}N}\times I\)
agrees on \(\partial B\) with the stationary extension.  Apply
Lemma~\ref{lem:relative-extension-local-D}\textup{(iii)}.  This gives a
transported class \(\alpha_p\) with
\[
   \Theta^0_{\rho,D}(\Lambda)=\mu_D(\alpha_p),
\]
and proves the first assertion.  If the displayed cohomology group
vanishes, then \(\kappa_1(\Lambda)=0\) for every marked loop, so the subgroup
generated by all raw loop contributions satisfies
\eqref{eq:RTheta-contained-in-mu}.

Throughout the normalization the loop remains in
\(\mathcal M_\rho(F_0)\).  Thus \(H=\rho_0(G)\) is fixed as an actual
subgroup of \(\pi\), the preferred lifts are transported continuously and
return to their initial sheets, and every ordered double-point label remains
in \(H\backslash\pi/H\).  Exchange still sends \(HgH\) to
\(Hg^{-1}H\), and the fixed self-dual class \(D\) is not conjugated or
relabelled.
\end{proof}

\begin{theorem}\label{thm:B-survival}
Let \(D=HgH\neq H\) be self-dual.  If
\[
   D\cap\{\gamma\in\pi:\gamma^2=1\}=\varnothing
\]
and
\[
   H^1(\Sigma_g;(\pi_2X)_\rho)=0,
\]
then
\[
   \operatorname{im}\mu_D=0,
   \qquad
   R_{\Theta,D}=0,
\]
and consequently
\[
   Q^D_\rho(X,F_0)\cong\Ftwo\langle u_D\rangle.
\]
\end{theorem}

\begin{proof}
The order-two hypothesis gives
\(\operatorname{im}\mu_D=0\) by Lemma~\ref{lem:muD-zero}.  The
cohomological hypothesis and Lemma~\ref{lem:H2-local-D} give
\[
   R_{\Theta,D}\subseteq\operatorname{im}\mu_D=0.
\]
Thus the raw mapping-space-loop subgroup itself is zero, not merely its
image in the final quotient.  Definition~\ref{def:QD} now gives the stated
one-dimensional target.  No nonvanishing or surjectivity assertion for the
local-insertion contribution is used.
\end{proof}

\section{\texorpdfstring{\(D\)}{D}-crossed tracks}
\label{sec:D-crossed-blocks}

In this section, we isolate the local geometric input which realizes the generator
\(u_D\) of the raw self-dual double-coset count.  The main definition is a specified track: a compactly supported marked
finger-plus-Whitney movie whose labelled fiber product has exactly one
\(D\)-component. The finger-guide and Whitney-disk data provide a practical way to verify this condition in the concrete models.

\subsection{\texorpdfstring{\(D\)}{D}-crossed tracks}

\begin{definition}[\(D\)-crossed track]
\label{def:D-crossed-block}
Let
\(
F_0:\Sigma_g\hookrightarrow X
\)
be a \(\rho\)-marked embedded surface with base path
\(\lambda_0:x_0\to F_0(z_0)\), preferred lift
\(\widehat F_0:\Sigma_g\to X_H\), and image subgroup \(H\leq\pi\).  Let
\(D\neq H\) be a nontrivial self-dual double-coset.  A \(D\)-crossed track for \(F_0\) is a generic \(\rho\)-marked track
\(
G:\Sigma_g\times I\to X\times I
\)
with preferred lift \(\widehat G\), together with a compact support
\(K\subset\operatorname{int}X\), satisfying the following conditions.

\begin{enumerate}[label=\textup{(B\arabic*)}]
\item The support \(K\) is disjoint from
\(
\lambda_0([0,1])\cup F_0(z_0),
\)
and the track is product outside \(K\) and near the endpoint collars.

\item The non-product part of \(G\) is a standard finger-plus-framed-Whitney
movie.  More precisely, after reparametrizing the interval, \(G\) first
performs a finger move creating exactly two transverse double points of the
intermediate immersion, is stationary for a middle interval, and then
performs a framed Whitney move pairing exactly those two double points.  The
top endpoint is an embedded surface
\(
F_1:\Sigma_g\hookrightarrow X
\)
homotopic to \(F_0\).

\item The lift \(\widehat G\) has endpoint restrictions
\(
\widehat F_0\times\{0\}\) and 
\(\widehat F_1\times\{1\},
\)
where \(\widehat F_1\) is the preferred lift induced by the unchanged base
path.  Thus \(G\) is a \(\rho\)-marked track from \(F_0\) to
\(F_1\).

\item The labelled fiber product \(P_D(G)\), in the sense of
Definition~\ref{def:PD}, has exactly one connected component:
\(
\#\pi_0(P_D(G))=1.
\)
This component is supported in the compact set \(K\).  If
\(G_X=\operatorname{pr}_X\circ G\) and
\(M_K=G_X^{-1}(K)\subset\Sigma_g\times I\), then
\(
P_D(G)\subset M_K\times M_K.
\)
At a regular level immediately after the finger move, the unique component
meets the level slice in the two selected ordered branch pairs and, because
\(D=D^{-1}\), in their exchanged ordered pairs.
\end{enumerate}
\end{definition}
Thus, once a local track has exactly one \(D\)-labelled off-diagonal component, it realizes the generator of the raw \(D\)-coordinate.
\begin{theorem}
\label{thm:C-D-crossed-block}
If \(F_0\) admits a \(D\)-crossed track, then the associated
track
\(
G:\Sigma_g\times I\to X\times I
\)
is a generic \(\rho\)-marked track from \(F_0\) to an embedded surface
\(F_1\) homotopic to \(F_0\), and
\(
FQ^{0,D}_\rho(G)=u_D.
\)
\end{theorem}
The definition above assumes that the chosen label \(D\) is self-dual.  In
the examples, self-duality will come from a simple algebraic source: an
element outside the surface subgroup whose square lies in the surface
subgroup.  We record this elementary criterion now because it is the bridge
between the local \(D\)-crossed movie and the square-root neighbourhoods constructed below from M\"obius bands.

\begin{lemma}
\label{lem:algebraic-square-root-criterion}
Let \(H\leq\pi\), and let \(g\in\pi\).  Put
\(
D=HgH\in H\backslash\pi/H.
\)
Then
\(
HgH=Hg^{-1}H
\)
if and only if \(D\) contains an element \(t\) such that
\(
t^2\in H.
\)
More precisely, if
\(g^{-1}=h_1gh_2\) for \(h_1,h_2\in H\), then
\(
t:=gh_2=h_1^{-1}g^{-1}
\)
lies in \(HgH\) and satisfies
\(
t^2=h_1^{-1}h_2\in H.
\)
Conversely, if \(t\in HgH\) and \(t^2\in H\), then
\(
HgH=Hg^{-1}H.
\)
\end{lemma}

\begin{proof}
Assume first that \(HgH=Hg^{-1}H\).  Then \(g^{-1}\in HgH\), so there are
\(h_1,h_2\in H\) with
\(
g^{-1}=h_1gh_2.
\)
Set \(t=gh_2\).  Then \(t\in HgH\), and the displayed equation gives
\(
gh_2=h_1^{-1}g^{-1}.
\)
Therefore
\(
t^2=(gh_2)(gh_2)=(h_1^{-1}g^{-1})(gh_2)=h_1^{-1}h_2\in H.
\)

Conversely, suppose \(t\in HgH\) and \(t^2\in H\).  Write \(c=t^2\in H\).
Then
\(
t^{-1}=tc^{-1},
\)
so
\(
Ht^{-1}H=Htc^{-1}H=HtH.
\)
Since \(t\in HgH\), we have \(HtH=HgH\).  Taking inverses gives
\(
Ht^{-1}H=(HtH)^{-1}=(HgH)^{-1}=Hg^{-1}H.
\)
Thus \(HgH=Hg^{-1}H\).
\end{proof}

\subsection{The twofold annular movie}

The next proposition is a local movie in a twofold annular quotient.  It is
used below only after it is placed inside a square-root local neighbourhood.  The
proposition computes the relevant local ordered fiber product explicitly.

\begin{figure}[ht]
\centering
\begin{tikzpicture}[
   scale=1,
   line cap=round,
   line join=round,
   every node/.style={font=\small},
   main/.style={line width=0.65pt},
   guide/.style={line width=0.45pt,dashed},
   dot/.style={circle,fill=black,inner sep=1.25pt},
   arrow/.style={-{Stealth[length=5pt,width=4pt]},line width=0.45pt}
]

\draw[main] (0,0) rectangle (8,3.4);

\draw[guide] (0,1.7) -- (8,1.7);
\node[left] at (0,1.7) {\(\lambda=\frac12\)};

\draw[guide] (0,2.8) -- (8,2.8);
\node[left] at (0,2.8) {\(\frac12+\alpha\)};

\draw[guide] (0,0.6) -- (8,0.6);
\node[left] at (0,0.6) {\(\frac12-\alpha\)};

\draw[main,domain=0:8,samples=300,smooth]
   plot (\x,{1.7 + 1.1*cos(90*\x)});

\node[dot] at (1,1.7) {};
\node[dot] at (3,1.7) {};
\node[dot] at (5,1.7) {};
\node[dot] at (7,1.7) {};

\node[below] at (1,0) {\(\xi_+\)};
\node[below] at (3,0) {\(\xi_-\)};
\node[below] at (5,0) {\(J\xi_+\)};
\node[below] at (7,0) {\(J\xi_-\)};

\node[below] at (0,-0.38) {\(x=0\)};
\node[below] at (8,-0.38) {\(x=1\)};
\node[below] at (4,-0.38) {\(S^1_x=\mathbb R/\mathbb Z\)};

\node[above] at (4,3.48) {\(\lambda=T(x)\)};

\draw[arrow] (6.1,3.72) -- (7.25,3.72);
\node[above] at (6.67,3.72) {\(x\)};

\end{tikzpicture}
\caption{The \((x,\lambda)\)-projection of the local ordered fiber-product
component in Proposition~\ref{prop:cutoff-crossed-movie}.  The \(x\)-circle
is cut open as \(S^1_x=\mathbb R/\mathbb Z\), and the curve shown is
\[
   x\longmapsto (x,T(x)),
   \qquad
   T(x)=\frac12+\alpha\cos(4\pi x).
\]
The actual fiber-product component is
\[
   Z=
   \left\{
   \left((x,S(x)),\,J(x,S(x)),\,T(x)\right):x\in S^1
   \right\}
   \subset A\times A\times I .
\]
Thus the suppressed coordinates are determined by \(s=S(x)\) and by the
forced second ordered point \(J(x,S(x))\).  Since \(x\in S^1\) parametrizes
\(Z\), this component is a single circle.  At the regular level
\(\lambda=\frac12\), the four ordered local equalities occur at
\(\xi_+,\xi_-,J\xi_+\), and \(J\xi_-\), all on this same component.}
\label{fig:twofold-annular-movie}
\end{figure}
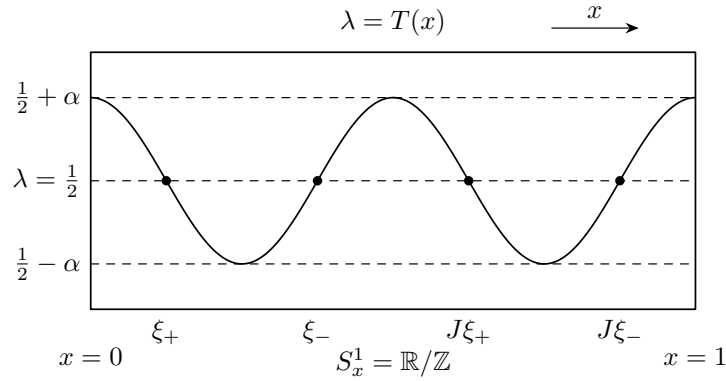

\begin{proposition}
\label{prop:cutoff-crossed-movie}
Let
\[
   S^1_x=\mathbb R/\mathbb Z,
   \qquad
   \widehat{\mathcal B}=S^1_x\times[-3,3]_s\times D^2_{(u,v)} .
\]
All additions in the \(x\)-coordinate are taken modulo \(1\).  Define
\[
   \iota(x,s,u,v)=\left(x+\frac12,-s,-u,v\right).
\]
Put
\[
A=S^1_x\times[-3,3]_s,
\qquad
J(x,s)=\left(x+\frac12,-s\right),
\qquad
R(u,v)=(-u,v),
\]
so that \(\iota(a,z)=(Ja,Rz)\).

Choose constants
\[
0<\alpha<\frac14,\qquad
0<\delta<\frac14,\qquad
0<\varepsilon\ll r\ll1,
\]
and set
\[
T(x)=\frac12+\alpha\cos(4\pi x),
\qquad
S(x)=\delta\sin(2\pi x).
\]
Let \(\chi:[-3,3]\to[0,1]\) be a smooth even cutoff function satisfying
\[
\chi(s)=1\quad\text{for } |s|\le1,
\qquad
\chi(s)=0\quad\text{for } |s|\ge2.
\]
For each \(0\le\lambda\le1\), define smooth maps
\[
\Delta_\lambda:A\to\mathbb R^2_{(u,v)},
\qquad
n_\lambda:A\to\mathbb R^2_{(u,v)}
\]
by
\[
\Delta_\lambda(x,s)
=
(1-\chi(s))(2r,0)
+
\chi(s)\varepsilon\bigl(T(x)-\lambda,\ s-S(x)\bigr),
\qquad
n_\lambda(x,s)=\frac12\Delta_\lambda(x,s).
\]
After choosing \(\varepsilon\ll r\ll1\) sufficiently small, we have
\(n_\lambda(A)\subset D^2_{(u,v)}\) for all \(\lambda\).
Let
\[
\widehat f_\lambda:A\to\widehat{\mathcal B},
\qquad
\widehat f_\lambda(x,s)=(x,s,n_\lambda(x,s)).
\]

After concatenating an \(\iota\)-free setup isotopy from
\(A\times\{(r,0)\}\) to \(\widehat f_0(A)\), the family
\(\widehat f_\lambda\), and a final product collar at \(\widehat f_1(A)\),
one obtains a smooth product-collared generic track
\(
\widehat G_A:A\times I\to\widehat{\mathcal B}\times I
\)
which is fixed for \(|s|\ge2\), starts at \(A\times\{(r,0)\}\), and ends at
the final graph \(\widehat f_1(A)\).  Its local ordered
\(\iota\)-fiber product
\[
P_\iota(\widehat G_A)
=
\{(a,b,\theta):
\widehat G_A(a,\theta)
=
(\iota\times\operatorname{id}_I)\widehat G_A(b,\theta)\}
\]
is a single connected compact one-manifold.  In the main movie interval,
where the movie parameter is \(\lambda\), it is the circle
\[
Z=
\left\{
\left((x,S(x)),\,J(x,S(x)),\,T(x)\right):x\in S^1
\right\}.
\]
No other local \(\iota\)-fiber-product component occurs.
\end{proposition}

\begin{proof}
First note that
\[
T\left(x+\frac12\right)=T(x),
\qquad
S\left(x+\frac12\right)=-S(x),
\qquad
\chi(-s)=\chi(s).
\]
Therefore
\(
\Delta_\lambda(Ja)=-R\Delta_\lambda(a)
\)
for every \(a=(x,s)\in A\).  Since \(n_\lambda=\frac12\Delta_\lambda\), we
obtain
\[
n_\lambda(a)-R n_\lambda(Ja)=\Delta_\lambda(a).
\]

Suppose
\(
\widehat f_\lambda(a)=\iota\widehat f_\lambda(b).
\)
The base coordinates force \(b=Ja\).  The normal coordinates then require
\(
n_\lambda(a)=R n_\lambda(Ja),
\)
which is equivalent to
\(
\Delta_\lambda(a)=0.
\)
Thus
\[
\widehat f_\lambda(a)=\iota\widehat f_\lambda(b)
\quad\Longleftrightarrow\quad
b=Ja\ \text{ and }\ \Delta_\lambda(a)=0.
\]

We compute the zero set of \(\Delta_\lambda\).  If \(\chi(s)=0\), then
\(
\Delta_\lambda(x,s)=(2r,0)\neq0.
\)
If \(1<|s|<2\) and \(\chi(s)>0\), then
\[
|s-S(x)|\ge |s|-\delta>0,
\]
so the second coordinate
\[
\chi(s)\varepsilon(s-S(x))
\]
is nonzero.  Hence no zero occurs in the transition region.

Thus every zero lies in the core \(|s|\le1\), where \(\chi=1\).  There
\[
\Delta_\lambda(x,s)=
\varepsilon\bigl(T(x)-\lambda,\ s-S(x)\bigr),
\]
so
\[
\Delta_\lambda(x,s)=0
\quad\Longleftrightarrow\quad
s=S(x),\qquad \lambda=T(x).
\]
Consequently the local \(\iota\)-fiber product of the main movie is exactly
\[
Z=
\left\{
\left((x,S(x)),\,J(x,S(x)),\,T(x)\right):x\in S^1
\right\}.
\]
This is a single embedded circle parametrized by \(x\in S^1\).  The point
\(x+\frac12\) gives the exchanged ordered pair in the same component, not a
second component.

We now attach the initial sheet.  Let
\(
n_{\operatorname{std}}(x,s)=(r,0).
\)
Then
\[
n_{\operatorname{std}}(a)-R n_{\operatorname{std}}(Ja)=(2r,0).
\]
For \(0\le\mu\le1\), set
\[
\Delta^{\operatorname{in}}_\mu
=
(1-\mu)(2r,0)+\mu\Delta_0,
\qquad
n^{\operatorname{in}}_\mu=\frac12\Delta^{\operatorname{in}}_\mu.
\]
The same equivariance relation holds:
\(
\Delta^{\operatorname{in}}_\mu(Ja)=-R\Delta^{\operatorname{in}}_\mu(a).
\)
In the core \(|s|\le1\), the first coordinate is
\[
(1-\mu)2r+\mu\varepsilon T(x)>0.
\]
In the transition region, if \(\chi(s)=0\) the first coordinate is \(2r\);
if \(\chi(s)>0\), the second coordinate is nonzero because
\(s-S(x)\neq0\).  In the outer collar \(|s|\ge2\), the vector is
\((2r,0)\).  Therefore
\(
\Delta^{\operatorname{in}}_\mu(a)\neq0
\)
for all \(a\) and all \(\mu\).  This gives an \(\iota\)-free setup isotopy
from \(A\times\{(r,0)\}\) to \(\widehat f_0(A)\).

The movie is fixed near the domain boundary.  Indeed, for \(|s|\ge2\),
\[
\Delta_\lambda(x,s)=(2r,0),
\qquad
n_\lambda(x,s)=(r,0),
\]
for every \(\lambda\).  Thus the local movie glues to the unchanged surface
outside \(A\).

The final local surface is the graph
\(
\widehat f_1(A).
\)
There is no return isotopy to the initial sheet appended.  It is \(\iota\)-free:
in the core, \(T(x)-1<0\) for every \(x\); in the transition region the
second coordinate is nonzero whenever \(\chi(s)>0\); and in the collar
\(\chi=0\), the first coordinate is \(2r\).  Hence
\(
\Delta_1(x,s)\neq0
\)
everywhere.

The genericity check is local to the core.  Near the zero set, \(\chi=1\),
and the defining map is
\[
\Phi(x,s,\lambda)
=
\varepsilon\bigl(T(x)-\lambda,\ s-S(x)\bigr).
\]
Its differential has rank \(2\) along the zero set because
\[
\frac{\partial\Phi}{\partial\lambda}=(-\varepsilon,0),
\qquad
\frac{\partial\Phi}{\partial s}=(0,\varepsilon).
\]
Thus \(P_\iota(\widehat G_A)\) is a smooth one-manifold.  The initial setup
isotopy and the final collar are \(\iota\)-free, so smoothing the
concatenation junctions inside these \(\iota\)-free regions neither creates
nor destroys local \(\iota\)-fiber-product points.  After reparametrizing
the three movie intervals with constant endpoint collars, the track is
smooth and product-collared.

The projection of \(Z\) to movie time has the standard birth and death.
The function
\[
T(x)=\frac12+\alpha\cos(4\pi x)
\]
has minima at \(x=\frac14,\frac34\), which are exchanged by \(J\), and
maxima at \(x=0,\frac12\), also exchanged by \(J\).  Hence downstairs there
is one finger birth and one Whitney death.  At the regular level
\(\lambda=\frac12\), the equation \(T(x)=\lambda\) has the four domain
points
\[
\xi_+=\left(\frac18,S\left(\frac18\right)\right),
\qquad
\xi_-=\left(\frac38,S\left(\frac38\right)\right),
\]
\[
J\xi_+=\left(\frac58,S\left(\frac58\right)\right),
\qquad
J\xi_-=\left(\frac78,S\left(\frac78\right)\right).
\]
Following \(Z\) in increasing \(x\)-order gives one connected ordered cycle
through these four ordered local equalities.

It remains to identify the Whitney disk at the death.  Let
\(
\lambda_+=\frac12+\alpha.
\)
For \(\lambda=\lambda_+-\eta\), with \(0<\eta\ll1\), the two local double
points near \(x=0\) occur at \(x=\pm y_\eta\), where
\(
T(\pm y_\eta)=\lambda.
\)
Work in a small coordinate neighbourhood of \(x=0\), write the coordinate
as \(y\), and use the core region \(\chi=1\).  At \(s=S(y)\),
\[
\widehat f_\lambda(y,S(y))
=
\left(y,S(y),\frac{\varepsilon}{2}(T(y)-\lambda),0\right).
\]
The local \(\iota\)-partner over the same downstairs point has normal
coordinate
\[
\left(-\frac{\varepsilon}{2}(T(y)-\lambda),0\right).
\]
Define
\(
W_\eta:[-y_\eta,y_\eta]\times[0,1]\to\widehat{\mathcal B}
\)
by
\[
W_\eta(y,\vartheta)
=
\left(
y,S(y),
(1-2\vartheta)\frac{\varepsilon}{2}(T(y)-\lambda),
0
\right).
\]
The edge \(\vartheta=0\) lies on the preferred graph, the edge
\(\vartheta=1\) lies on the local \(\iota\)-partner, and the sides
\(y=\pm y_\eta\) collapse to the two double points.  Thus the boundary is
a continuous loop in the lifted product chart.  After the standard corner
smoothing, \(W_\eta\) is an embedded disk in \(\widehat{\mathcal B}\).  Its
\(\iota\)-image lies near \(x=\frac12\), while \(W_\eta\) lies near \(x=0\),
so the projection to the quotient is an embedded downstairs Whitney disk.
The product coordinates give the Whitney framing.

Every possible local \(\iota\)-equality in the main movie is detected by
\(\Delta_\lambda=0\), and this zero set is the single circle \(Z\).  The
setup isotopy, domain-boundary collar, transition region, final collar, and
final surface are all \(\iota\)-free.  Since each \(\widehat f_\lambda\) is a graph over \(A\), the preferred
lifted sheet has no off-diagonal self-intersections.  Thus the only possible
local ordered equalities with nontrivial local label are the
\(\iota\)-equalities.  These have already been computed by the equation
\(\Delta_\lambda=0\), whose zero set is the single circle \(Z\).  Hence no
other local \(\iota\)-fiber-product component occurs.
\end{proof}

\subsection{Standard square-root neighbourhoods}
The cutoff movie of Proposition~\ref{prop:cutoff-crossed-movie} will be
placed in a four-dimensional local quotient neighbourhood.  The definition below records the local quotient chart, the annular surface patch, the square-root element, the associated sheet-exchange label, and the isolation condition
needed for the movie.  Lemma~\ref{lem:mobius-band-square-root-block} gives
the geometric source of such neighbourhoods from embedded M\"obius bands.

\begin{definition}[Standard square-root neighbourhood]
\label{def:standard-square-root-handle}
Let \(F:\Sigma_g\hookrightarrow X\) be a \(\rho\)-marked embedded surface
with base path \(\lambda_F:x_0\to F(z_0)\), preferred lift \(\widehat F\),
and image subgroup \(H\leq\pi=\pi_1(X,x_0)\).  Let
\[
   A=S^1_x\times[-3,3]_s\subset\Sigma_g\setminus\{z_0\}
\]
be an annular neighbourhood of an oriented essential curve.  Using
\(\lambda_F\) and a chosen domain path from \(z_0\) to \((0,0)\in A\), let
\(c\in H\) be the based element represented by the oriented core
\(S^1_x\times\{0\}\).  A normal push-off \(c^+\) of this core is read with
the same base path and represents the same element \(c\).

A standard square-root neighbourhood for \(c\) along \(F\) consists of a
compact local neighbourhood \(U\subset X\) and a component
\(\widehat{\mathcal U}\subset q^{-1}(U)\) with the following properties.

\begin{enumerate}[label=\textup{(S\arabic*)}]
\item
\(U\cap(\lambda_F([0,1])\cup F(z_0))=\varnothing\).

\item The lifted neighbourhood admits coordinates
\[
   \widehat{\mathcal U}
   =
   S^1_x\times[-3,3]_s\times D^2_{(u,v)}
\]
in which
\(
   q|_{\widehat{\mathcal U}}:\widehat{\mathcal U}\to U
\)
is the quotient by the free involution
\[
   \iota(x,s,u,v)=\left(x+\frac12,-s,-u,v\right).
\]
Thus the local quotient chart is the standard twofold quotient obtained from
the annular double cover of a M\"obius band, with one normal coordinate
reversed and the product coordinate fixed.  This is a local normal form, not
a global deck transformation of \(X_H\to X\).

\item The preferred local lifted sheet of the annular patch is
\(
   \widehat F_+=A\times\{(r,0)\}\subset\widehat{\mathcal U}
\)
for some \(0<r\ll1\), and \(q(\widehat F_+)=F(A)\cap U\).  The restriction
\(q|\widehat F_+\) is a diffeomorphism onto \(F(A)\cap U\), and
\(\widehat F_-=\iota(\widehat F_+)\) is the local \(\iota\)-partner.  The
unchanged surface \(F(\Sigma_g\setminus A)\) is disjoint from \(U\), except
along the fixed collar where the movie is glued to the stationary surface.

\item There is an element \(t\in\pi\), called the square-root element, such
that \(t\notin H\), \(t^2=c\in H\), and every prescribed local
sheet-exchange pair \((p_H,\iota p_H)\) arising from the preferred lifted
movie support has double-coset label
\(
   \ell(p_H,\iota p_H)=HtH .
\)

\item The preferred lifted cutoff movie is supported in
\(\widehat{\mathcal U}\times I\), is stationary near the gluing collar, and
has no projected equalities involving the moving annulus except same-sheet
equalities and the prescribed
\((\iota\times\operatorname{id}_I)\)-equalities.  Moreover, there is no
mixed \(HtH\)-labelled equality with one domain coordinate in the moving
annulus and the other on the unchanged surface outside \(A\).
\end{enumerate}
\end{definition}

\begin{lemma}
\label{lem:square-root-crossed-block}
Let \(F:\Sigma_g\hookrightarrow X\) be a \(\rho\)-marked embedded surface
with image subgroup \(H\leq\pi\).  Suppose \(F\) contains an annular
neighbourhood \(A=S^1_x\times[-3,3]_s\) of an essential oriented curve \(c\),
and that \(X\) contains a standard square-root neighbourhood for \(c\) along
\(F\).  Let \(t\) be the square-root element and put
\(
   D=HtH.
\)
Then \(D=HtH=Ht^{-1}H\neq H\), and the cutoff movie gives a \(D\)-crossed
track \(G\) with
\(
   FQ^{0,D}_\rho(G)=u_D.
\)
\end{lemma}

\begin{proof}
Since \(t^2=c\in H\), we have \(t^{-1}=tc^{-1}\), hence
\(Ht^{-1}H=HtH\).  Since \(t\notin H\), this double coset is nontrivial.

Project the cutoff movie of Proposition~\ref{prop:cutoff-crossed-movie} to
the quotient neighbourhood \(U\), glue it to the stationary surface outside
\(A\), and insert product collars.  The support is away from the marked base
data, so the resulting track is \(\rho\)-marked.  The projected movie is the
standard finger-plus-framed-Whitney movie supplied by
Proposition~\ref{prop:cutoff-crossed-movie}; the isolation condition (S5) in
Definition~\ref{def:standard-square-root-handle} makes the top surface
embedded.

By Definition~\ref{def:standard-square-root-handle}\textup{(S4)}, every
prescribed local sheet-exchange pair has label \(D\).  By
\textup{(S5)}, there are no other \(D\)-labelled local or mixed pairs
involving the moving annulus.  Thus the \(D\)-labelled part of the fiber
product is exactly the local ordered \(\iota\)-fiber product of the preferred
lifted cutoff movie.  Proposition~\ref{prop:cutoff-crossed-movie} computes
this fiber product as the single circle
\[
   Z=
   \left\{
   \left((x,S(x)),\,J(x,S(x)),\,T(x)\right):x\in S^1
   \right\}.
\]
Therefore \(G\) is a \(D\)-crossed track, and
Theorem~\ref{thm:C-D-crossed-block} gives
\(
   FQ^{0,D}_\rho(G)=u_D .
\)
\end{proof}

\begin{lemma}[M\"obius bands supply standard square-root neighbourhoods]
\label{lem:mobius-band-square-root-block}
Let \(M\) be a compact orientable smooth three-manifold, and let
\(
   \Sigma_g\subset\operatorname{int}M\), \(g\ge2,\)
be a closed oriented \(\pi_1\)-injective surface.  Put
\(H=\pi_1\Sigma_g\le \pi_1M=\pi\), using the based convention fixed in
Convention~\ref{conv:basepoints-markings}.  Let
\(
   B\subset M\setminus\operatorname{Int}\nu\Sigma_g
\)
be an embedded M\"obius band with
\(\partial B=c^+\subset\partial\nu\Sigma_g\), where \(c^+\) is a normal
push-off of an oriented essential curve \(c\subset\Sigma_g\).  Let
\(t\in\pi\) be represented by the core of \(B\), oriented so that
\(\partial B\) represents \(t^2\), and assume \(t\notin H\).

Let \(X_M=M\times[-1,1]\), and let
\(F_0:\Sigma_g\hookrightarrow X_M\) be the copy \(\Sigma_g\times\{0\}\).
After choosing an annular patch \(A\subset\Sigma_g\) around \(c\) and
shrinking a product neighbourhood of \(B\times\{0\}\) together with the
normal collar from \(c\) to \(c^+\), the four-manifold \(X_M\) contains a
standard square-root neighbourhood for \(c\) along \(F_0\), with
square-root element \(t\).  Consequently the cutoff movie gives a
\(D\)-crossed track with raw value \(u_D\), where
\(
   D=HtH=Ht^{-1}H\neq H .
\)
\end{lemma}

\begin{proof}
The boundary of a M\"obius band represents twice its core.  Hence
\(\partial B=c^+\) gives \(t^2=c\in H\), with \(c^+\) read as the normal
push-off of \(c\) using the fixed based convention.  Since \(t\notin H\), the
label \(HtH\) is nontrivial, and \(t^{-1}=tc^{-1}\) makes it self-dual.

Let \(N(B)\) be an orientable regular neighbourhood of \(B\) in
\(M\setminus\operatorname{Int}\nu\Sigma_g\).  Then \(N(B)\) is the
orientable \(I\)-bundle over the M\"obius band and has fundamental group
\(\langle t\rangle\).  Since \(t^2\in H\) and \(t\notin H\),
\(
   H\cap\langle t\rangle=\langle t^2\rangle .
\)
Thus the relevant component of the \(H\)-cover over \(N(B)\) is the twofold
cover associated to \(\langle t^2\rangle<\langle t\rangle\).  On the
zero-section this is the orientation double cover of the M\"obius band,
namely an annulus.

Choose lifted coordinates
\(\
   S^1_x\times[-3,3]_s\times[-\varepsilon,\varepsilon]_u
\)
on this three-dimensional cover.  The deck transformation is
\((x,s)\mapsto(x+\frac12,-s)\) on the annular double cover.  This deck transformation reverses the orientation of the annular double
cover.  Since \(M\) is orientable, the lifted transformation on the
three-dimensional neighbourhood also reverses the normal coordinate \(u\).  Crossing with the product
coordinate \(v\) in \(M\times[-1,1]\) gives the local quotient chart
\(
   S^1_x\times[-3,3]_s\times D^2_{(u,v)}
\)
with involution
\[
   \iota(x,s,u,v)=\left(x+\frac12,-s,-u,v\right).
\]

Attach the normal collar from the annular patch \(A\subset\Sigma_g\) to the
push-off \(c^+=\partial B\).  In the lifted chart the preferred lifted sheet
of \(F_0(A)\) is \(A\times\{(r,0)\}\) for \(0<r\ll1\).  The
sheet-exchanging loop is represented by running once in the core direction
of the M\"obius band, so its double-coset label is \(HtH\).  Finally,
shrinking the chart inside the chosen neighbourhood of \(B\times\{0\}\) and
the collar makes the cutoff movie isolated: no unchanged part of the surface
outside \(A\) enters the neighbourhood except along the fixed gluing collar,
and the only local off-diagonal equalities involving the moving annulus are
the prescribed \((\iota\times\operatorname{id}_I)\)-equalities.  Hence this
is a standard square-root neighbourhood, and the final assertion follows
from Lemma~\ref{lem:square-root-crossed-block}.
\end{proof}

\subsection{Normal quotient dictionary}
\label{sec:specializations}

Assume that \(H\triangleleft\pi\), and put
\(G=\pi/H\).  Then double-coset labels are ordinary quotient labels.  

\begin{corollary}
\label{cor:finite-normal}
If \(H\triangleleft\pi\) and \(G=\pi/H\), then
\(
H\backslash\pi/H\cong G.
\)
A nontrivial quotient label \(\bar g\in G\) is self-dual exactly when
\(
\bar g^2=1.
\)
For each such label, the self-dual double-coset construction defines the
marked obstruction associated to \(D=HgH\), where \(g\) is any representative
of \(\bar g\).  The survival theorem still requires both the order-two
exclusion
\(
D\cap\{\gamma\in\pi:\gamma^2=1\}=\varnothing
\)
and the cohomological vanishing
\(
H^1(\Sigma_g;(\pi_2X)_\rho)=0.
\)
\end{corollary}

\begin{proof}
When \(H\) is normal, \(HgH\) depends only on the quotient class
\(\bar g\in G\), and every quotient class gives one double-coset.  The
inverse label is \(Hg^{-1}H\), corresponding to \(\bar g^{-1}\).  Hence
\(HgH=Hg^{-1}H\) if and only if \(\bar g=\bar g^{-1}\), equivalently
\(\bar g^2=1\).  The obstruction and the stated survival hypotheses are
those of Theorems~\ref{thm:A-marked-obstruction} and~\ref{thm:B-survival}.
\end{proof}

\begin{corollary}
\label{cor:index-two}
Assume \(H\triangleleft\pi\) and \(\pi/H\cong\Z/2\).  Let \(\tau\) be the
nontrivial deck transformation of \(X_H\to X\), and let
\(
D=\pi\setminus H
\)
denote the unique nontrivial double-coset label.  Then the double-coset
fiber product is
\[
P_D(A)=
\{(x,y):\widehat A(x)=\tau\widehat A(y)\}.
\]
The local order-two hypothesis in the survival theorem is the requirement
that the nontrivial coset \(\pi\setminus H\) contain no element of order two;
the loop indeterminacy also requires
\(
H^1(\Sigma_g;(\pi_2X)_\rho)=0.
\)
\end{corollary}

\begin{proof}
The quotient has one nontrivial element, and that element is its own
inverse.  Since the cover is regular with deck group \(\Z/2\), two lifted
points have nontrivial label precisely when one is the \(\tau\)-translate of
the other.  The statements about the two survival hypotheses are the
specialization of Theorem~\ref{thm:B-survival}.
\end{proof}

\begin{remark}
\label{rem:no-involutions}
If a normal quotient \(G=\pi/H\) has no nontrivial involutions, then this
mod-two single-coordinate construction supplies no nontrivial normal-quotient labels.
Non-self-dual labels require a different signed, oriented, or paired-label
invariant.
\end{remark}

\section{Endpoint-rigid image nonconcordance}\label{sec:image}

The preceding sections produce marked obstruction classes and local tracks
with prescribed nonzero values.  This section explains when such marked
nonconcordance promotes to image nonconcordance.  The issue is endpoint
reparametrization: an image concordance may replace \(F_1\) by
\(F_1\circ\phi\).  In the non-normal setting, compatible endpoint
reparametrizations are controlled by the normalizer of the marked subgroup.

We denote the normalizer by
\(
N_\pi(H)=\{\gamma\in\pi:\gamma H\gamma^{-1}=H\}.
\)
Conjugation by an element of \(N_\pi(H)\) restricts to an automorphism of
\(H\), giving the natural homomorphism
\[
N_\pi(H)\longrightarrow\Out(H),
\qquad
\gamma\longmapsto[c_\gamma|_H].
\]

\begin{lemma}\label{lem:centralizer-promotion}
Assume \(C_\pi(H)\subseteq H\).  If two \(\rho\)-marked surfaces with the
same marking are parametrized concordant, then they are \(\rho\)-marked
concordant.
\end{lemma}

\begin{proof}
Let \(C:\Sigma_g\times I\hookrightarrow X\times I\) be a parametrized
concordance from \(F_0\) to \(F_1\), and write
\(C_X=\operatorname{pr}_X\circ C\).  Let
\[
a(t)=C_X(z_0,t),\qquad 0\le t\le1,
\]
be the path traced by the domain basepoint.  Concatenating this path with
the endpoint base paths gives a loop based at \(x_0\), whose class is
\(
\gamma=[\lambda_0a\overline{\lambda_1}]\in\pi.
\)
For a based loop \(\alpha\subset\Sigma_g\), the cylinder \(\alpha\times I\)
gives the moving-basepoint identity
\(
[aF_1(\alpha)\overline a]=[F_0(\alpha)]
\in\pi_1(X,F_0(z_0)).
\)
Using the endpoint markings,
\[
[\lambda_iF_i(\alpha)\overline{\lambda_i}]=\rho(\alpha),
\qquad i=0,1,
\]
we obtain
\[
\gamma\rho(\alpha)\gamma^{-1}
=
[\lambda_0aF_1(\alpha)\overline a\,\overline{\lambda_0}]
=
[\lambda_0F_0(\alpha)\overline{\lambda_0}]
=\rho(\alpha).
\]
Thus \(\gamma\in C_\pi(H)\).  By the hypothesis \(C_\pi(H)\subseteq H\), we
have \(\gamma\in H\).  Therefore the loop
\(\lambda_0a\overline{\lambda_1}\) lifts to a closed loop in the \(H\)-cover
starting at \(\widehat x_0\).  Equivalently, the lift of \(a\) starting at
the preferred point \(\widehat F_0(z_0)\) ends at the preferred point
\(\widehat F_1(z_0)\).  The concordance therefore has the prescribed
endpoint lifts, and hence is a marked concordance.
\end{proof}

\begin{lemma}
\label{lem:normalizer-endpoint-full}
Suppose \(F_0\) is parametrized concordant to \(F_1\circ\phi\), where
\(
   \phi\in\Diff(\Sigma_g)
\). Then the basepoint track gives an
element \(\gamma\in N_\pi(H)\) and
\(
   [\rho\phi_*\rho^{-1}]=[c_\gamma|_H]\in\Out(H).
\)
\end{lemma}

\begin{proof}
Let \(A\) be a parametrized concordance from \(F_0\) to \(F_1\circ\phi\),
and write \(A_X=\operatorname{pr}_X\circ A\).  Let
\(
   a(t)=A_X(z_0,t)
\)
be the path traced by the domain basepoint.  Choose a path
\(\delta:z_0\to\phi(z_0)\) in \(\Sigma_g\), and set
\(
   \lambda_1^\phi=\lambda_1F_1(\delta).
\)
Define
\(
   \gamma=[\lambda_1^\phi\,\overline a\,\overline{\lambda_0}]\in\pi.
\)
For every based loop \(\alpha\subset\Sigma_g\), the cylinder
\(\alpha\times I\) gives
\(
   [\overline a\,F_0(\alpha)\,a]
   =
   [(F_1\circ\phi)(\alpha)]
\)
in \(\pi_1(X,F_1(\phi(z_0)))\).  Hence
\[
   c_\gamma(\rho(\alpha))
   =
   [\lambda_1^\phi(F_1\circ\phi)(\alpha)\overline{\lambda_1^\phi}]
   =
   \rho([\delta\,\phi(\alpha)\,\overline\delta]).
\]
Therefore \(c_\gamma(H)=H\), so \(\gamma\in N_\pi(H)\), and
\(
   [c_\gamma|_H]=[\rho\phi_*\rho^{-1}]
   \in\Out(H).
\)
\end{proof}

\begin{corollary}[Full endpoint-rigid image nonconcordance]
\label{cor:full-endpoint-rigid-image}
Let \(F_0\) be a \(\rho\)-marked \(\pi_1\)-injective surface with image
subgroup \(H\le\pi\), and let \(D=HgH\neq H\) be self-dual.  Assume:
\begin{enumerate}[label=\textup{(J\arabic*)}]
\item
\(
   FQ^D_\rho(F_0,F_1)\neq0 \text{ in }Q^D_\rho(X,F_0).
\)
\item
\(
   C_\pi(H)\subseteq H.
\)
\item The full endpoint-rigidity set is trivial:
\[
\left\{[\phi]\in\Mod(\Sigma_g):
[\rho\phi_*\rho^{-1}]
\in
\operatorname{im}\!\left(N_\pi(H)\to\Out(H)\right)
\right\}
=
\{1\}.
\]
\end{enumerate}
Then \(F_0(\Sigma_g)\) and \(F_1(\Sigma_g)\) are not smoothly
image-concordant.
\end{corollary}

\begin{proof}
If the images were smoothly image-concordant, then \(F_0\) would be
parametrized concordant to \(F_1\circ\phi\) for some
\(\phi\in\Diff(\Sigma_g)\).  By
Lemma~\ref{lem:normalizer-endpoint-full}, the mapping class \([\phi]\) lies
in the displayed full endpoint-rigidity set.  Hence \([\phi]=1\).
Choose a smooth isotopy
\(\phi_t\), constant near its endpoints, with
\(\phi_0=\phi\) and \(\phi_1=\operatorname{id}_{\Sigma_g}\).  The trace
\[
   T_\phi:\Sigma_g\times I\longrightarrow X\times I,
   \qquad
   T_\phi(z,t)=(F_1(\phi_t(z)),t),
\]
is an embedded product-collared track from \(F_1\circ\phi\) to \(F_1\).
Concatenating \(T_\phi\) to the given concordance, and smoothing the
junction inside the product collars, gives a parametrized concordance from
\(F_0\) to \(F_1\).  Since
\(C_\pi(H)\subseteq H\), Lemma~\ref{lem:centralizer-promotion} promotes this
parametrized concordance to a \(\rho\)-marked concordance.  This contradicts
the nonzero marked obstruction in \textup{(J1)}, by
Theorem~\ref{thm:A-marked-obstruction}.
\end{proof}

The following short-cut arguments are only needed in Section~\ref{sec:orientation-family}.
\begin{lemma}\label{lem:orientation-reversing-exclusion}
Assume \([F_0]\neq-[F_0]\in H_2(X;\Z)\).  If \(F_1\) is homotopic to
\(F_0\), then no orientation-reversing \(\phi\in\Diff(\Sigma_g)\) can make
\(F_0\) parametrized concordant to \(F_1\circ\phi\).
\end{lemma}

\begin{proof}
A parametrized concordance gives equality of homology classes.  Since
\(F_1\) is homotopic to \(F_0\), \([F_1]=[F_0]\).  If \(\phi\) reverses
orientation, then \([F_1\circ\phi]=-[F_1]=-[F_0]\).  A concordance from
\(F_0\) to \(F_1\circ\phi\) would imply \([F_0]=-[F_0]\), contrary to the
hypothesis.
\end{proof}

\begin{corollary}\label{cor:primitive-dual-exclusion}
If \(X\) is compact with boundary and there is a relative class
\([D_0]\in H_2(X,\partial X;\Z)\) such that
\(
[F_0]\cdot[D_0]=1,
\)
then \([F_0]\neq-[F_0]\).
\end{corollary}

\begin{proof}
If \([F_0]=-[F_0]\), then \(2[F_0]=0\).  Pairing with \([D_0]\) gives
\(0=2([F_0]\cdot[D_0])=2\), impossible in \(\Z\).
\end{proof}

\begin{theorem}\label{thm:D-image}
Let \(F_0\) be a \(\rho\)-marked \(\pi_1\)-injective surface with image
subgroup \(H\leq\pi\), and let \(D=HgH\neq H\) be self-dual.  Assume:
\begin{enumerate}[label=\textup{(I\arabic*)}]
\item \( FQ^D_\rho(F_0,F_1)\neq 0 \quad\text{in }Q^D_\rho(X,F_0). \)
\item \(C_\pi(H)\subseteq H\).
\item Endpoint rigidity holds:
\[
\left\{[\phi]\in\Mod^+(\Sigma_g):
[\rho\phi_*\rho^{-1}]\in
\operatorname{im}(N_\pi(H)\to\Out(H))
\right\}=\{1\}.
\]
\item \([F_0]\neq-[F_0]\). 
\end{enumerate}
Then \(F_0(\Sigma_g)\) and \(F_1(\Sigma_g)\) are not smoothly
image-concordant.
\end{theorem}

\begin{proof} 
Suppose that the images were smoothly image-concordant. Then \(F_0\) is parametrized concordant to \(F_1\circ\phi\) for some \(\phi\in\Diff(\Sigma_g)\). If \(\phi\) is orientation-preserving, the proof of Corollary~\ref{cor:full-endpoint-rigid-image}, with \textup{(I3)} in place of the full endpoint-rigidity hypothesis, gives a contradiction. If \(\phi\) is orientation-reversing, Lemma~\ref{lem:orientation-reversing-exclusion} and \textup{(I4)} give a contradiction. Hence no smooth image concordance exists. 
\end{proof}

\section{Applications}\label{sec:applications}
We now turn the marked double-coset obstruction into examples.  The
three-dimensional input is an embedded M\"obius band in the complement of a
\(\pi_1\)-injective surface: by
Lemma~\ref{lem:mobius-band-square-root-block}, such a band supplies the
local neighbourhood which realizes a nontrivial self-dual label.  We first state a
general criterion for families of such bands.  We then verify the criterion
in a compact Klein-bottle \(I\)-bundle model and use the same local input to construct the closed graph-manifold mapping-torus family of the main theorem.

\subsection{A M\"obius-band square-root criterion}
\label{subsec:characteristic-square-roots}
In this section we provide a criterion that converts a three-dimensional supply of M\"obius-band
square roots into stabilized four-dimensional image-nonconcordance examples.
Its input is a family of embedded M\"obius bands in the surface complement
with distinct self-dual double-coset labels; its output is the corresponding
family in
\(
   (M\times[-1,1])\#(S^2\times S^2).
\)
Throughout this subsection, \(M\) is a compact orientable smooth
three-manifold, \(\Sigma=\Sigma_g\subset\operatorname{int}M\), \(g\ge2\), is
a closed \(\pi_1\)-injective surface, and
\(
   H=\pi_1\Sigma_g\le \pi=\pi_1M
\)
is fixed using Convention~\ref{conv:basepoints-markings}.  Put
\(
   C=M\setminus\operatorname{Int}\nu\Sigma .
\)

\begin{definition}[M\"obius-band square-root family]
\label{def:mobius-square-root-family}
A \emph{M\"obius-band square-root family} for \((M,\Sigma)\) is a family of
properly embedded M\"obius bands
\(
   B_\alpha\subset C\), \(\alpha\in A,
\)
such that
\(
   \partial B_\alpha=c_\alpha^+\subset\partial\nu\Sigma
\)
is a normal push-off of an oriented essential simple closed curve
\(c_\alpha\subset\Sigma\).  Let \(t_\alpha\in\pi\) be the core element of
\(B_\alpha\), oriented so that the boundary of \(B_\alpha\) represents
\(t_\alpha^2\).  We require
\[
   t_\alpha^2=c_\alpha\in H,
   \qquad
   t_\alpha\notin H.
\]
The associated self-dual double-coset is
\(
   \mathcal D_\alpha=Ht_\alpha H.
\)
Indeed,
\(
   t_\alpha^{-1}=t_\alpha c_\alpha^{-1}
\)
implies
\(
   Ht_\alpha^{-1}H=Ht_\alpha H.
\)
We call the family \emph{essential} if every \(B_\alpha\) is
incompressible, boundary-incompressible, and not properly homotopic into
\(\partial C\).  A single properly embedded M\"obius band with these three
properties will likewise be called essential.
\end{definition}

The next proposition locates the three-dimensional source of the square-root
mechanism.  A M\"obius band in
\(C=M\setminus\operatorname{Int}\nu\Sigma\) turns a boundary curve
\(c_\alpha\in H\) into a square-root element \(t_\alpha\notin H\) with
\(t_\alpha^2=c_\alpha\), and hence into a self-dual double-coset
\(Ht_\alpha H=Ht_\alpha^{-1}H\).  By the characteristic-pair theorem of
Jaco--Shalen--Johannson, equivalently the engulfing property of the
characteristic submanifold, every essential occurrence of this mechanism in
an irreducible complement with incompressible boundary is carried by the
characteristic submanifold
\cite[Chapter~V]{JacoShalen1979}
\cite[Chapters~5--7]{Johannson1979}. In Haken surface complements, the square-root phenomenon is carried by the
canonical Seifert or \(I\)-bundle part of the complement. The
later examples realize this localized mechanism explicitly, using a
Klein-bottle \(I\)-bundle piece.

\begin{proposition}[Characteristic localization]
\label{prop:characteristic-localization}
Let \(C\) be compact, orientable, irreducible, and with incompressible
boundary.  Let
\(
   \Psi(C)\subset C
\)
be the Jaco--Shalen--Johannson characteristic submanifold, taken relative to
the full boundary pattern \(\partial C\).  If
\(
   B\subset C
\)
is a properly embedded essential M\"obius band, then \(B\) is properly
homotopic into \(\Psi(C)\).
\end{proposition}

\begin{proof}
Let \(N(B)\) be a regular neighbourhood of \(B\) in \(C\).  Since \(C\) is
orientable, \(N(B)\) is the orientable \(I\)-bundle over a M\"obius band; in
particular it is a Seifert pair with boundary pattern
\(N(B)\cap\partial C\). The essentiality of \(B\) implies that the induced Seifert-pair map
\[
   (N(B),N(B)\cap\partial C)\longrightarrow (C,\partial C)
\]
is essential.  Equivalently, the frontier annulus of \(N(B)\) is essential:
a compression, boundary-compression, or boundary-parallelism of that annulus
would give the corresponding compression, boundary-compression, or proper
boundary homotopy for \(B\).

By the characteristic-pair theorem of Jaco--Shalen--Johannson, equivalently
the engulfing property of the characteristic submanifold, every essential Seifert-pair map into \((C,\partial C)\) is properly
homotopic into the characteristic pair.  Applying this to
\((N(B),N(B)\cap\partial C)\), the regular neighbourhood \(N(B)\), and hence
the core M\"obius band \(B\), is properly homotopic into \(\Psi(C)\).
\end{proof}

\begin{remark}
\label{rem:no-characteristic-classification}
Proposition~\ref{prop:characteristic-localization} is used as a
localization statement: in an irreducible surface complement with
incompressible boundary, essential M\"obius bands occur in the characteristic
submanifold.  Several characteristic pieces can contain such bands, including
orientable \(I\)-bundles over nonorientable surfaces and Seifert pieces with
vertical M\"obius bands.  For the applications below, the relevant model is
the Klein-bottle \(I\)-bundle of
Subsection~\ref{subsec:klein-bottle-infinite-square-roots}, where the core
elements \(a^nt\) give infinitely many square roots with a common boundary
square:
\(
   (a^nt)^2=t^2 .
\)
\end{remark}

\begin{lemma}[Dual sphere kills the meridian]
\label{lem:dual-sphere-kills-meridian}
Let \(F\subset X\) be a closed embedded oriented surface with a framed dual
sphere \(D_0\).  Suppose that
\(
D_0\setminus\operatorname{Int}\nu F
\)
is a disk in \(X\setminus\nu F\) bounding a meridian \(\mu\) of \(F\).  Then
the inclusion
\(
X\setminus\nu F\longrightarrow X
\)
induces an isomorphism on \(\pi_1\).
\end{lemma}

\begin{proof}
Let \(C=X\setminus\operatorname{Int}\nu F\).  Van Kampen for
\(X=C\cup\nu F\) gives
\(
\pi_1X\cong \pi_1C/\langle\!\langle\mu\rangle\!\rangle,
\)
where \(\mu\) is a meridian of \(F\).  The punctured dual sphere is a disk in
\(C\) with boundary \(\mu\), so \(\mu=1\in\pi_1C\).  Therefore the quotient
map \(\pi_1C\to\pi_1X\) is already injective as well as surjective.
\end{proof}

\begin{lemma}[Survival after finite stabilization]
\label{lem:aspherical-finite-stabilization-survival}
Let \(Z\) be an aspherical connected smooth four-manifold with
\(\Pi=\pi_1Z\), and assume that \(\Pi\) is torsion-free.  Let
\[
   X=Z\# q(S^2\times S^2),
   \qquad q\ge1.
\]
Let \(F_0:\Sigma_g\hookrightarrow X\), \(g\ge1\), be a \(\rho\)-marked
embedded surface whose marked image subgroup is
\(
   F_{0*}\pi_1\Sigma_g=H\le \Pi=\pi_1X .
\)
Let \(D\in H\backslash\Pi/H\) be a nontrivial self-dual double-coset.  Then
\(
   H^1(\Sigma_g;(\pi_2X)_\rho)=0
\)
and
\(
   Q^D_\rho(X,F_0)\cong \Ftwo\langle u_D\rangle .
\)
\end{lemma}

\begin{proof}
Since \(Z\) is aspherical, \(\pi_2Z=0\).  The connected sum with
\(q(S^2\times S^2)\) contributes \(2q\) free
\(\mathbb Z[\Pi]\)-summands, so
\(
   \pi_2X\cong \mathbb Z[\Pi]^{2q}
\)
as a left \(\mathbb Z[\Pi]\)-module.  The local coefficient system
\((\pi_2X)_\rho\) depends on the \(\Pi\)-module \(\pi_2X\) and the marking
\(\rho:\pi_1\Sigma_g\cong H\), not on whether the image of \(F_0\) lies
inside the \(Z\)-summand.  Restricting to \(H\), we have
\[
   \operatorname{Res}^{\Pi}_{H}\mathbb Z[\Pi]
   \cong
   \bigoplus_{H\backslash\Pi}\mathbb Z[H],
\]
and therefore
\[
   (\pi_2X)_\rho
   \cong
   \bigoplus_{H\backslash\Pi}\mathbb Z[H]^{2q}.
\]
The surface group \(H\cong\pi_1\Sigma_g\) is an orientable \(PD_2\)-group, so
\(
   H^1(H;\mathbb Z[H])=0
\)
\cite[Chapter~VIII]{BrownCohomology}.  Since \(H\) has a finite projective
resolution, \(H^1(H;-)\) commutes with direct sums of free
\(\mathbb Z[H]\)-modules.  Hence
\(
   H^1(\Sigma_g;(\pi_2X)_\rho)=0.
\)
Because \(\Pi\) is torsion-free and \(D\neq H\), the double-coset \(D\)
contains no element of order two.  The survival theorem,
Theorem~\ref{thm:B-survival}, gives
\(
   Q^D_\rho(X,F_0)\cong \Ftwo\langle u_D\rangle .
\)
\end{proof}

\begin{theorem}
\label{thm:general-mobius-square-root-consequence}
Let \(M\) be a compact orientable aspherical smooth three-manifold, and let
\(
   \Sigma=\Sigma_g\subset\operatorname{int}M\), \(g\ge2,
\)
be a closed \(\pi_1\)-injective surface.  Put
\(
   H=\pi_1\Sigma_g\le \pi=\pi_1M.
\)
Assume that \(C=M\setminus\operatorname{Int}\nu\Sigma\) contains a
M\"obius-band square-root family
\(
   \{B_\alpha\}_{\alpha\in A}
\)
with core elements \(t_\alpha\), and write
\(
   \mathcal D_\alpha=Ht_\alpha H.
\)
Assume:
\[
   \mathcal D_\alpha\ne\mathcal D_\beta
   \quad(\alpha\ne\beta),
\]
\[
   \pi \text{ is torsion-free},
   \qquad
   N_\pi(H)=H.
\]
Assume also that the square-root neighbourhoods can be chosen away from the fixed
basepoint, base path, and the stabilizing tube used below.  Equivalently, the
annular patches \(c_\alpha\subset\Sigma_g\) are taken disjoint from a fixed
small disk in \(\Sigma_g\) used for the marking and stabilization; in the
applications this is arranged by construction.

Let
\[
   X_M=(M\times[-1,1])\#(S^2\times S^2).
\]
Then \(X_M\) contains a base \(\rho\)-marked embedded surface
\(
   F_0:\Sigma_g\hookrightarrow X_M
\)
and, for each \(\alpha\in A\), a \(\rho\)-marked embedded surface
\(
   F_\alpha:\Sigma_g\hookrightarrow X_M
\)
such that:

\begin{enumerate}[label=\textup{(\roman*)}]
\item all \(F_\alpha\) are homotopic to \(F_0\);
\item \(F_{\alpha *}\pi_1\Sigma_g=H\) for every \(\alpha\);
\item the surfaces have a common framed embedded dual sphere;
\item the complement maps
\(
   \pi_1(X_M\setminus\nu F_\alpha)\longrightarrow\pi_1X_M
\)
are isomorphisms, including for \(F_0\);
\item for every \(\beta\in A\),
\(
   Q^{\mathcal D_\beta}_\rho(X_M,F_0)
   \cong
   \Ftwo\langle u_{\mathcal D_\beta}\rangle;
\)
\item for all \(\alpha,\beta\in A\),
\[
   FQ^{\mathcal D_\beta}_\rho(F_0,F_\alpha)
   =
   \begin{cases}
   u_{\mathcal D_\beta},& \alpha=\beta,\\
   0,& \alpha\ne\beta.
   \end{cases}
\]
\end{enumerate}

Consequently the embedded images
\(
   F_0(\Sigma_g)\), \(F_\alpha(\Sigma_g)\), \(\alpha\in A
\)
are pairwise not smoothly image-concordant.
\end{theorem}

\begin{proof}
Take the connected sum with \(S^2\times S^2\) in the interior of
\(M\times[-1,1]\), away from the surface, the marking data, and the regions
where the square-root neighbourhoods will be supported.  Let
\[
   S_1=S^2\times\{\operatorname{pt}\},
   \qquad
   D_{\mathrm{dual}}=\{\operatorname{pt}\}\times S^2
\]
in the stabilizing summand.  Define \(F_0\) to be
\(\Sigma\times\{0\}\), connected-summed with \(S_1\) along a tube chosen away
from the marking data and from the annular patches used by the square-root
neighbourhoods.  Connected sum with \(S^2\) does not change the genus or the image
subgroup, so
\(
   F_{0*}\pi_1\Sigma_g=H.
\)
The sphere \(D_{\mathrm{dual}}\) is a framed embedded dual sphere for
\(F_0\).

For each \(\alpha\), apply
Lemma~\ref{lem:mobius-band-square-root-block} to the single band
\(B_\alpha\).  This gives a standard square-root neighbourhood in
\(M\times[-1,1]\), with square-root core \(t_\alpha\), and hence a
\(\mathcal D_\alpha\)-crossed track.  Let \(F_\alpha\) be the
top surface.  The construction is supported away from the basepoint, base
path, stabilizing tube, and stabilizing summand.  Hence \(F_\alpha\) is
\(\rho\)-marked, homotopic to \(F_0\), and has image subgroup \(H\).  The
sphere \(D_{\mathrm{dual}}\) is disjoint from the support and
therefore remains a common framed embedded dual sphere.

For every \(\alpha\), the punctured dual sphere
\(
   D_{\mathrm{dual}}\setminus\operatorname{Int}\nu F_\alpha
\)
is a disk in \(X_M\setminus\nu F_\alpha\) bounding a meridian of
\(F_\alpha\).  Lemma~\ref{lem:dual-sphere-kills-meridian} gives the
complement \(\pi_1\)-isomorphism, including for \(F_0\).

The survival statement follows from
Lemma~\ref{lem:aspherical-finite-stabilization-survival}, applied to
\[
   Z=M\times[-1,1],
   \qquad
   \Pi=\pi_1M=\pi .
\]
Indeed \(Z\) is aspherical because \(M\) is aspherical, and \(\pi\) is
torsion-free by hypothesis.  Hence, for every \(\alpha\in A\),
\(
   Q^{\mathcal D_\alpha}_\rho(X_M,F_0)
   \cong
   \Ftwo\langle u_{\mathcal D_\alpha}\rangle .
\)

The distinguishing formula follows from the local calculation.  The
\(\alpha\)-track uses only the square-root neighbourhood associated to \(B_\alpha\).
By Lemma~\ref{lem:mobius-band-square-root-block} and
Lemma~\ref{lem:square-root-crossed-block}, this neighbourhood contributes exactly one
local component with label \(\mathcal D_\alpha\), and no component with any
other nontrivial label.  Since the labels \(\mathcal D_\alpha\) are pairwise
distinct,
\[
   FQ^{\mathcal D_\beta}_\rho(F_0,F_\alpha)
   =
   \begin{cases}
   u_{\mathcal D_\beta},& \alpha=\beta,\\
   0,& \alpha\ne\beta.
   \end{cases}
\]

It remains to promote marked nonconcordance to image nonconcordance.  Since
\(N_\pi(H)=H\), we have \(C_\pi(H)\subseteq H\).  Moreover the image of
\(
   N_\pi(H)\longrightarrow\Out(H)
\)
is trivial, because conjugation by an element of \(H\) is inner on \(H\).
Since \(g\ge2\), the extended Dehn--Nielsen--Baer theorem
\cite[Theorem~8.1]{FarbMargalit2012} identifies the full mapping class group
with the outer automorphism group:
\(
   \Mod(\Sigma_g)\xrightarrow{\cong}\Out(\pi_1\Sigma_g).
\)
Hence the full endpoint-rigidity set in
Corollary~\ref{cor:full-endpoint-rigid-image} is trivial.

For each \(\alpha\), the nonzero value
\(
   FQ^{\mathcal D_\alpha}_\rho(F_0,F_\alpha)=u_{\mathcal D_\alpha}
\)
and Corollary~\ref{cor:full-endpoint-rigid-image} show that \(F_0\) is not
smoothly image-concordant to \(F_\alpha\).

Now let \(\alpha\ne\beta\).  If \(F_\alpha\) and \(F_\beta\) were smoothly
image-concordant, then Lemma~\ref{lem:normalizer-endpoint-full}, the
triviality of the full endpoint-rigidity set, and
Lemma~\ref{lem:centralizer-promotion} would reduce the image concordance to a
\(\rho\)-marked concordance from \(F_\alpha\) to \(F_\beta\).  Concatenating
such a marked concordance with a marked track from \(F_0\) to \(F_\alpha\)
would force
\(
   FQ^{\mathcal D_\alpha}_\rho(F_0,F_\alpha)
   =
   FQ^{\mathcal D_\alpha}_\rho(F_0,F_\beta),
\)
contradicting the displayed Kronecker formula.  Hence the embedded images are
pairwise not smoothly image-concordant.
\end{proof}

The next two elementary lemmas will be used in the concrete models later to
verify the two group-theoretic hypotheses in
Theorem~\ref{thm:general-mobius-square-root-consequence}: pairwise
distinctness of the double-coset labels and the normalizer condition
\(N_\pi(H)=H\).

\begin{lemma}
\label{lem:quotient-double-coset-test}
Let \(H\le \pi\), and let \(t_\alpha,t_\beta\in\pi\).  Suppose there is a
homomorphism
\(
   q:\pi\to Q
\)
such that \(q(H)=1\).  If
\(
   q(t_\alpha)\ne q(t_\beta),
\)
then
\(
   Ht_\alpha H\ne Ht_\beta H.
\)
More generally, if \(q(H)=L\le Q\) and
\(
   Lq(t_\alpha)L\ne Lq(t_\beta)L
   \quad\text{in }L\backslash Q/L,
\)
then
\(
   Ht_\alpha H\ne Ht_\beta H.
\)
\end{lemma}

\begin{proof}
If \(Ht_\alpha H=Ht_\beta H\), then
\(
   t_\alpha=h_1t_\beta h_2
\)
for some \(h_1,h_2\in H\).  Applying \(q\) gives
\(
   q(t_\alpha)=q(h_1)q(t_\beta)q(h_2).
\)
If \(q(H)=1\), this gives \(q(t_\alpha)=q(t_\beta)\).  In the general case,
it gives equality of the \(L\)-double-cosets
\(
   Lq(t_\alpha)L=Lq(t_\beta)L.
\)
The contrapositive proves both assertions.
\end{proof}

\begin{lemma}
\label{lem:bass-serre-normalizer-criterion}
Let \(G\) act without inversions on a tree \(T\).  Let \(v\in T\) be a
vertex, and let \(H\le G_v\).  Suppose that \(H\) is noncyclic and that no
edge stabilizer of \(T\) contains a conjugate of \(H\).  Then
\(
   N_G(H)\le G_v.
\)
Consequently
\(
   N_G(H)=N_{G_v}(H).
\)
In particular, if \(N_{G_v}(H)=H\), then
\(
   N_G(H)=H.
\)
\end{lemma}

\begin{proof}
Let \(x\in N_G(H)\).  Since \(H\le G_v\), the subgroup \(H\) fixes \(v\).
For \(h\in H\),
\(
   h(xv)=x(x^{-1}hx)v=xv,
\)
because \(x^{-1}hx\in H\).  Hence \(H\) fixes both \(v\) and \(xv\).

If \(xv\ne v\), then \(H\) fixes the geodesic segment from \(v\) to \(xv\),
and therefore fixes every edge in that segment.  This puts \(H\) inside an
edge stabilizer, contrary to the hypothesis.  Thus \(xv=v\), so
\(x\in G_v\).  Therefore
\(
   N_G(H)\le G_v,
\)
and intersecting with the normalizer inside \(G_v\) gives
\(
   N_G(H)=N_{G_v}(H).
\)
The final assertion follows immediately.
\end{proof}

\subsection{Infinitely many square roots in a Klein-bottle \texorpdfstring{\(I\)}{I}-bundle}
\label{subsec:klein-bottle-infinite-square-roots}

We now verify the hypotheses of
Theorem~\ref{thm:general-mobius-square-root-consequence} in one compact
aspherical three-manifold.

\begin{lemma}[The Klein-bottle \(I\)-bundle]
\label{lem:klein-bottle-I-bundle}
Let
\[
   A=[0,1]_r\times S^1_\theta,
   \qquad
   \varphi(r,\theta)=(1-r,-\theta),
\]
and define
\[
   V=A\times[0,1]_s/(x,1)\sim(\varphi(x),0).
\]
Then \(V\) is the orientable twisted \(I\)-bundle over the Klein bottle.  It
is compact, orientable, and aspherical, and
\[
   \pi_1V
   =
   K
   =
   \langle a,t\mid tat^{-1}=a^{-1}\rangle .
\]
Moreover, for every \(n\in\Z\), \(V\) contains a properly embedded M\"obius
band \(B_n^V\) such that:
\begin{enumerate}[label=\textup{(\roman*)}]
\item the curves \(\partial B_n^V\subset\partial V\) are all the same
boundary slope;
\item this boundary slope represents \(t^2\in\pi_1V\);
\item the core of \(B_n^V\) represents
\(
   t_n=a^n t;
\)
\item
\(
   t_n^2=(a^n t)^2=t^2.
\)
\end{enumerate}
\end{lemma}

\begin{proof}
The diffeomorphism \(\varphi\) preserves the orientation of the annulus,
because it reverses both the \(r\)-coordinate and the \(\theta\)-coordinate.
Hence its mapping torus \(V\) is orientable.  Equivalently, \(V\) is the
\(I\)-bundle over
\[
   K\!b=S^1_\theta\times[0,1]_s/(\theta,1)\sim(-\theta,0),
\]
whose fiber coordinate is \(r\) and whose transition function is
\(r\mapsto1-r\).  The base \(K\!b\) is the Klein bottle, and the fiber
transition reverses orientation exactly when the base transition does; hence
the total space is the orientable twisted \(I\)-bundle over the Klein
bottle.

Choose the basepoint \((1/2,0,0)\).  Let \(a\) be the angular loop in the
annulus fiber, and let \(t\) be the mapping-torus loop at the fixed point
\((1/2,0)\) of \(\varphi\).  Since \(\varphi_*(a)=a^{-1}\), Van Kampen gives
\[
   \pi_1V
   =
   \langle a,t\mid tat^{-1}=a^{-1}\rangle .
\]
The universal cover is
\(
   \widetilde V=[0,1]_r\times\R_x\times\R_s,
\)
with deck transformations
\[
   a(r,x,s)=(r,x+1,s),
   \qquad
   t(r,x,s)=(1-r,-x,s+1).
\]
This universal cover is contractible, so \(V\) is aspherical.

For \(n\in\Z\), define an arc
\[
   \alpha_n:[0,1]\longrightarrow A,
   \qquad
   \alpha_n(r)=(r,nr\bmod 1).
\]
It is embedded because the first coordinate is \(r\).  Also
\[
   \varphi(\alpha_n(r))
   =
   (1-r,-nr)
   =
   (1-r,n(1-r))
   =
   \alpha_n(1-r)
   \quad\text{in }[0,1]\times S^1.
\]
Therefore \(\alpha_n\times[0,1]\subset A\times[0,1]\) descends to a properly
embedded surface
\(
   B_n^V\subset V.
\)
Its domain is
\[
   [0,1]_r\times[0,1]_s/(r,1)\sim(1-r,0),
\]
which is a M\"obius band.  This proves the embeddedness of
\(B_n^V\).

The boundary is independent of \(n\).  Indeed,
\(
   \alpha_n(0)=(0,0)\) and \( \alpha_n(1)=(1,0)
\)
for every \(n\).  Hence \(\partial B_n^V\) is the mapping-torus curve
obtained from the two boundary points \((0,0)\) and \((1,0)\), and this
curve does not depend on \(n\).

We now compute the boundary and core elements.  Let
\[
   \widetilde B_n
   =
   \{(r,nr,s):0\le r\le1,\ s\in\R\}
   \subset \widetilde V.
\]
Put \(g_n=a^n t\).  Then
\[
   g_n(r,nr,s)
   =
   a^n(1-r,-nr,s+1)
   =
   (1-r,n(1-r),s+1),
\]
so \(g_n\) preserves \(\widetilde B_n\).  On the strip coordinates, it acts
by
\[
   (r,s)\longmapsto(1-r,s+1).
\]
The quotient by this action is precisely the M\"obius band \(B_n^V\).

The core line
\[
   \left\{ \left(\frac12,\frac n2,s\right):s\in\R\right\}
   \subset\widetilde B_n
\]
is carried to itself by \(g_n=a^n t\), so the core of \(B_n^V\) represents
\(a^n t\).

The boundary line \(r=0\), \(x=0\), is carried to the other boundary line
\(r=1\), \(x=n\), by \(g_n\), and is carried back to itself by \(g_n^2\).
Since
\(
   g_n^2=(a^n t)^2
   =
   a^n(t a^n t^{-1})t^2
   =
   a^n a^{-n}t^2
   =
   t^2,
\)
the boundary of \(B_n^V\) represents \(t^2\).  This also proves
\((a^n t)^2=t^2\).
\end{proof}

\begin{theorem}
\label{thm:fixed-klein-bottle-infinite-family}
Let \(g\ge2\), and let
\(
   c\subset\Sigma_g
\)
be an oriented embedded essential simple closed curve.  Put
\(H=\pi_1\Sigma_g\), and let \(c\in H\) denote the based element represented
by the curve \(c\), using the based convention fixed in
Convention~\ref{conv:basepoints-markings} and
Definition~\ref{def:standard-square-root-handle}.

There is a compact orientable aspherical smooth three-manifold
\(M_{g,c}^{K}\), containing
\(
   \Sigma=\Sigma_g\times\{0\}
\)
as a closed \(\pi_1\)-injective surface, such that
\[
   \pi_1M_{g,c}^{K}
   \cong
   H*_{\langle c=t^2\rangle}
   \langle a,t\mid tat^{-1}=a^{-1}\rangle .
\]
Here the edge group is identified with \(\langle c\rangle\le H\) and with
\(\langle t^2\rangle\le \langle a,t\mid tat^{-1}=a^{-1}\rangle\).

Moreover,
\(
   M_{g,c}^{K}\setminus\operatorname{Int}\nu\Sigma
\)
contains embedded M\"obius bands
\(
   B_n\), \(n\in\Z,
\)
such that
\(
   \partial B_n=c^+\subset\partial\nu\Sigma
\)
is a fixed normal push-off of \(c\), and the core element of \(B_n\) is
\(
   t_n=a^n t.
\)
Thus
\[
   t_n^2=c\in H,
   \qquad
   t_n\notin H.
\]
For \(n\ge1\), write
\(
   \mathcal D_n=Ht_nH=H(a^n t)H.
\)
The double-cosets \(\mathcal D_n\), \(n\ge1\), are pairwise distinct,
nontrivial, and self-dual:
\(
   \mathcal D_n=\mathcal D_n^{-1}.
\)
The group \(\pi_1M_{g,c}^{K}\) is torsion-free, and
\(
   N_{\pi_1M_{g,c}^{K}}(H)=H.
\)

Consequently, for
\[
   X_{g,c}^{K}
   =
   (M_{g,c}^{K}\times[-1,1])\#(S^2\times S^2),
\]
there is a base \(\rho\)-marked embedded surface
\(
   F_0:\Sigma_g\hookrightarrow X_{g,c}^{K}
\), and there are \(\rho\)-marked embedded surfaces
\[
   F_n:\Sigma_g\hookrightarrow X_{g,c}^{K},
   \qquad n\ge1,
\]
where \(F_n\) is obtained from the square-root neighbourhood associated to \(B_n\).
The surfaces \(F_n\), \(n\ge0\), are all homotopic, satisfy
\(
   F_{n*}\pi_1\Sigma_g=H,
\)
have a common framed embedded dual sphere, and have complement
\(\pi_1\)-isomorphisms
\[
   \pi_1(X_{g,c}^{K}\setminus\nu F_n)
   \xrightarrow{\cong}
   \pi_1X_{g,c}^{K}.
\]
For every \(i\ge1\),
\(
   Q^{\mathcal D_i}_\rho(X_{g,c}^{K},F_0)
   \cong
   \Ftwo\langle u_{\mathcal D_i}\rangle,
\)
and, for all \(i,n\ge1\),
\[
   FQ^{\mathcal D_i}_\rho(F_0,F_n)
   =
   \begin{cases}
   u_{\mathcal D_i},& i=n,\\
   0,& i\ne n.
   \end{cases}
\]
The embedded images of
\(
   F_0,F_1,F_2,\ldots
\)
are pairwise not smoothly image-concordant.
\end{theorem}

\begin{proof}
We divide the proof into the three-dimensional construction, the group theory calculations, and the four-dimensional application.

\medskip
\noindent\emph{Construction of \(M_{g,c}^{K}\).}
Let
\[
   P=\Sigma_g\times[-1,1],
   \qquad
   \Sigma=\Sigma_g\times\{0\}\subset P.
\]
Choose an annular neighbourhood
\(
   A_c\subset\Sigma_g\times\{1\}
\)
of \(c\times\{1\}\).  Let \(V\) be the annulus mapping torus from
Lemma~\ref{lem:klein-bottle-I-bundle}.  Let
\(
   A_V\subset\partial V
\)
be an annular neighbourhood of the boundary slope \(\partial B_n^V\).  This
annulus is independent of \(n\), and its core represents \(t^2\).

Choose an orientation-reversing gluing diffeomorphism
\(
   A_V\to A_c
\)
which identifies the oriented core \(t^2\) with \(c\).  Define
\(
   M_{g,c}^{K}=P\cup_{A_c=A_V}V,
\)
smoothing corners after the gluing.  The resulting three-manifold is compact,
smooth, and orientable. The Van Kampen theorem gives
\[
   \pi_1M_{g,c}^{K}
   \cong
   H*_{\langle c=t^2\rangle}
   \langle a,t\mid tat^{-1}=a^{-1}\rangle .
\]
Bass--Serre normal form for the free product amalgam gives injectivity of the vertex
groups \cite{SerreTrees}.  In particular, the inclusion
\(
   \pi_1\Sigma=\pi_1(\Sigma_g\times\{0\})=H
   \longrightarrow
   \pi_1M_{g,c}^{K}
\)
is injective.  Hence \(\Sigma\subset M_{g,c}^{K}\) is \(\pi_1\)-injective.

The manifold \(M_{g,c}^{K}\) is aspherical.  The pieces \(P\) and \(V\) are
aspherical, and the gluing annulus is aspherical.  The two edge inclusions
are \(\pi_1\)-injective: the inclusion into \(P\) is generated by the
essential curve \(c\), and the inclusion into \(V\) is generated by the
infinite-order element \(t^2\).  The universal cover of
\(M_{g,c}^{K}\) is the Bass--Serre tree of spaces with vertex spaces the
universal covers of \(P\) and \(V\), and edge spaces universal covers of
annuli.  All vertex and edge spaces are contractible, and the underlying
nerve is a tree.  Thus the universal cover is contractible.

\medskip
\noindent\emph{The embedded M\"obius bands in the complement of \(\Sigma\).}
Choose
\(
   \nu\Sigma=\Sigma_g\times[-\varepsilon,\varepsilon]\subset P
\)
with \(0<\varepsilon<1\).  For each \(n\in\Z\), the band
\(B_n^V\subset V\) has boundary the core of \(A_V\), which is identified with
\(c\times\{1\}\subset A_c\).  Attach to \(B_n^V\) the collar annulus
\(
   c\times[\varepsilon,1]\subset\Sigma_g\times[\varepsilon,1].
\)
After smoothing the corner along \(c\times\{1\}\), the union is an embedded
M\"obius band
\(
   B_n\subset
   M_{g,c}^{K}\setminus\operatorname{Int}\nu\Sigma
\)
with
\(
   \partial B_n=c\times\{\varepsilon\}=c^+.
\)
The core of \(B_n\) may be chosen in the original \(B_n^V\) piece, so by
Lemma~\ref{lem:klein-bottle-I-bundle} it represents
\(
   t_n=a^n t.
\)
Also
\(
   t_n^2=(a^n t)^2=t^2=c\in H.
\)
The bands \(B_n\) share the boundary curve \(c^+\) and need not be
pairwise disjoint; each four-dimensional construction below uses only one
band at a time.

\medskip
\noindent\emph{Torsion-freeness.}
Let
\(
   K=\langle a,t\mid tat^{-1}=a^{-1}\rangle .
\)
The group \(K\) is torsion-free.  Indeed, the homomorphism
\[
   K\longrightarrow \Z,
   \qquad
   a\longmapsto0,\quad t\longmapsto1,
\]
has kernel \(\langle a\rangle\cong\Z\).  A finite-order element maps
trivially to \(\Z\), hence lies in \(\langle a\rangle\), and is therefore
trivial.  The surface group \(H\) is also torsion-free.

Now let \(\pi=\pi_1M_{g,c}^{K}\).  The group \(\pi\) acts on its
Bass--Serre tree.  A finite-order element of a group acting on a tree fixes
a vertex, hence is conjugate into a vertex group \cite{SerreTrees}.  Since both vertex groups
\(H\) and \(K\) are torsion-free, \(\pi\) is torsion-free.

\medskip
\noindent\emph{The normalizer.}
Let \(T\) be the Bass--Serre tree of
\(
   \pi=
   H*_{\langle c=t^2\rangle}K,
\)
and let \(v_H\) be the vertex fixed by \(H\).  The edge stabilizers are
conjugate to the cyclic group
\(
   \langle c\rangle=\langle t^2\rangle.
\)
Since \(H=\pi_1\Sigma_g\) is noncyclic for \(g\ge2\), no edge stabilizer
contains a conjugate of \(H\).  Applying
Lemma~\ref{lem:bass-serre-normalizer-criterion} with \(G_v=H\), we obtain
\(
   N_\pi(H)=H.
\)

\medskip
\noindent\emph{The double-coset labels.}
For \(n\ge1\), the relation \(t_n^2=c\in H\) implies
\(
   t_n^{-1}=t_n c^{-1}.
\)
Hence
\[
   Ht_n^{-1}H=Ht_n c^{-1}H=Ht_nH.
\]
Thus \(\mathcal D_n=Ht_nH\) is self-dual.  Define
\[
   \psi:\pi\longrightarrow
   D_\infty
   =
   \langle r,s\mid s^2=1,\ srs^{-1}=r^{-1}\rangle
\]
by
\[
   \psi(H)=1,\qquad
   \psi(a)=r,\qquad
   \psi(t)=s.
\]
This respects \(tat^{-1}=a^{-1}\), because \(srs^{-1}=r^{-1}\), and it
respects the amalgamating relation \(c=t^2\), because both \(c\) and \(t^2\)
map to \(1\).

For every \(n\ge1\),
\(
   \psi(t_n)=\psi(a^n t)=r^n s\ne1.
\)
Since \(\psi(H)=1\), this proves \(t_n\notin H\).  Therefore
\(\mathcal D_n\ne H\).

If \(n,m\ge1\) and
\(
   H(a^n t)H=H(a^m t)H,
\)
then there are \(h_1,h_2\in H\) such that
\(
   a^n t=h_1(a^m t)h_2.
\)
Applying \(\psi\) gives
\(
   r^n s=r^m s
\)
in \(D_\infty\), and hence \(r^n=r^m\).  Since \(r\) has infinite order in
\(D_\infty\), we get \(n=m\).  Thus the double-cosets
\(
   \mathcal D_n=H(a^n t)H\), \(n\ge1,
\)
are pairwise distinct.

\medskip
\noindent\emph{Application of the general square-root consequence.}
The family \(\{B_n\}_{n\ge1}\) is a M\"obius-band square-root family in
the sense of Definition~\ref{def:mobius-square-root-family}.  The quotient
\(\psi:\pi\to D_\infty\) and
Lemma~\ref{lem:quotient-double-coset-test} show that the labels
\(
   \mathcal D_n=H(a^nt)H\), \(n\ge1,
\)
are pairwise distinct.  The torsion-freeness and normalizer calculations
above give the remaining hypotheses of
Theorem~\ref{thm:general-mobius-square-root-consequence}.  Applying that
theorem to
\(
   X_{g,c}^{K}
\)
gives the asserted surfaces, the common framed dual sphere, the complement
\(\pi_1\)-isomorphisms, the displayed obstruction formula, and the
image-nonconcordance conclusion.
\end{proof}

\subsection{A closed graph-manifold family}
\label{subsec:closed-graph-manifold-family}

We now give a closed graph-manifold version of the square-root construction.
The local Klein-bottle \(I\)-bundle piece is the same one used in
Subsection~\ref{subsec:klein-bottle-infinite-square-roots}, but it is placed
inside a closed graph manifold carrying an active Dehn twist along a
nonseparating JSJ torus.  The quotient detecting the labels is slightly
different from the compact case: the Klein-bottle generator \(a\) maps to an
even power of the rotation in \(D_\infty\).  This still detects all labels.

Fix \(g\ge2\), and choose an oriented nonseparating simple closed curve
\(c\subset\Sigma_g\).  Let
\[
   A=S^1_x\times[-1,1]_y\subset\Sigma_g
\]
be an annular neighbourhood of \(c=S^1_x\times\{0\}\), and put
\(
   R=\overline{\Sigma_g\setminus A}.
\)
Then \(R\) has genus \(g-1\) and two boundary components
\(d_+\sqcup d_-\).  Let
\(
   W_R=R\times S^1_u.
\)
Its boundary consists of two tori
\(
   T^R_\pm=d_\pm\times S^1_u,
\)
and the Seifert fiber slope of \(W_R\) is the \(u\)-slope.

Let \(V\) be the orientable twisted \(I\)-bundle over the Klein bottle, as in
Lemma~\ref{lem:klein-bottle-I-bundle}.  Thus
\[
   \pi_1V=K=\langle a,t\mid tat^{-1}=a^{-1}\rangle.
\]
We use the Seifert fibration of \(V\) over \(D^2(2,2)\) whose regular fiber
is represented by the central element
\(
   h=t^2.
\)
The Seifert fibration can be seen directly from the universal-cover model
\[
   \widetilde V=[0,1]_r\times\mathbb R_x\times\mathbb R_s,
   \qquad
   a(r,x,s)=(r,x+1,s),\quad
   t(r,x,s)=(1-r,-x,s+1).
\]
The central element \(h=t^2\) acts by
\(
   h(r,x,s)=(r,x,s+2),
\)
so the images of the vertical lines
\(\{(r,x)\}\times\mathbb R_s\) are the regular Seifert fibers and represent
\(h=t^2\).  The induced action on the transverse annulus
\([0,1]_r\times S^1_x\) is generated by
\(
   (r,x)\longmapsto(1-r,-x).
\)
Its quotient is a disk with two fixed points, represented by
\((1/2,0)\) and \((1/2,1/2)\), and each fixed point has stabilizer of order
two.  Hence the base orbifold is \(D^2(2,2)\), the disk with two cone points
of order \(2\).

Choose the boundary transverse generator \(a\) so that the boundary torus
subgroup of \(V\) is \(\langle a,h\rangle\).  Replacing \(a\) by \(a^{-1}\),
if necessary, only reindexes the bands of Lemma~\ref{lem:klein-bottle-I-bundle}.

Choose two disjoint regular fibers \(F_+\) and \(F_-\) on \(\partial V\).
They cut \(\partial V\) into two closed vertical annuli, denoted \(A_+\) and
\(A_-\).  The boundary of each \(A_\pm\) is \(F_+\sqcup F_-\).  Let
\(
   P_A=A\times[-1,1]_z.
\)
Glue \(A\times\{1\}\) fiberwise to \(A_+\), and glue
\(A\times\{-1\}\) fiberwise to \(A_-\), matching every circle
\(S^1_x\times\{y\}\) with a regular fiber of \(V\).  The gluing is chosen so
that the two boundary circles
\[
   S^1_x\times\{1\}\times\{1\},
   \qquad
   S^1_x\times\{1\}\times\{-1\}
\]
are attached to the same fiber \(F_+\), while the two boundary circles with
\(y=-1\) are attached to \(F_-\).  After smoothing corners, denote the
resulting compact manifold by \(W_K\).

The Seifert fibration on \(V\) and the product fibration of \(P_A\) by the
\(S^1_x\)-circles glue to a Seifert fibration of \(W_K\).  The base orbifold
is obtained from \(D^2(2,2)\) by attaching a one-handle to the boundary along
the two boundary arcs corresponding to \(A_+\) and \(A_-\).  Hence the base
orbifold is an annulus with two cone points of order two.  The regular fiber
of \(W_K\) is represented by \(h=t^2\).  The boundary of \(W_K\) consists of
two tori, denoted \(T^K_+\) and \(T^K_-\).  The torus \(T^K_+\) is obtained
from the side annulus
\[
   S^1_x\times\{1\}\times[-1,1]_z\subset P_A,
\]
and \(T^K_-\) is obtained similarly from the side annulus with \(y=-1\).  Let
\(m_\pm\subset T^K_\pm\) denote the regular fiber slope.

\begin{lemma}
\label{lem:WK-boundary-subgroups}
With the choices above,
\[
   \pi_1W_K\cong
   \langle a,t,w\mid tat^{-1}=a^{-1},\ [w,t^2]=1\rangle .
\]
Writing \(h=t^2\), and using the basepoint and orientation convention fixed
in the proof, the two boundary torus subgroups are
\(
   \pi_1T^K_+=\langle h,w\rangle
\)
and
\(
   \pi_1T^K_-=\langle h,a^{-1}w^{-1}\rangle .
\)
\end{lemma}

\begin{proof}
Let the basepoint lie on \(F_+\), viewed as a common boundary fiber of the
annuli \(A_+\) and \(A_-\).  The piece \(P_A=A\times[-1,1]\) is an annulus
cross an interval, so its fundamental group is generated by the fiber
\(S^1_x\).  Both attaching annuli identify this generator with \(h=t^2\) in
\(\pi_1V\).  The Van Kampen theorem for attaching the two annuli is therefore the HNN
extension of \(\pi_1V\) which identifies the subgroup \(\langle h\rangle\)
on \(A_+\) with the same subgroup on \(A_-\).  If \(w\) denotes the stable
letter represented by the path in the side annulus
\(S^1_x\times\{1\}\times[-1,1]_z\), then
\(
   whw^{-1}=h.
\)
Together with \(\pi_1V=\langle a,t\mid tat^{-1}=a^{-1}\rangle\), this gives
\(
   \pi_1W_K\cong
   \langle a,t,w\mid tat^{-1}=a^{-1},\ [w,t^2]=1\rangle .
\)

It remains to identify the two boundary tori.  Let \(\delta_+\) be a
transverse arc in \(A_+\) from \(F_+\) to \(F_-\), and let \(\delta_-\) be a
transverse arc in \(A_-\) from \(F_+\) to \(F_-\).  We orient the boundary
transverse generator \(a\) of \(\partial V\) so that the loop
\(
   \delta_-\,\overline{\delta_+}
\)
represents \(a\).  This convention is compatible with replacing the old
Klein-bottle generator by its inverse, which only reindexes the elements
\(a^nt\).

The torus \(T^K_+\) is the mapping torus of the fiber \(F_+\) obtained from
\(S^1_x\times\{1\}\times[-1,1]_z\).  Its fiber is \(h\), and its transverse
loop is precisely \(w\).  Thus
\(
   \pi_1T^K_+=\langle h,w\rangle .
\)

For \(T^K_-\), orient the transverse loop in the opposite closing direction:
it runs in the side annulus \(S^1_x\times\{-1\}\times[-1,1]_z\) from the
\(A_-\)-side back to the \(A_+\)-side.  Let this loop, based at \(F_-\), be
\(\beta\).  The rectangle \([-1,1]_y\times[-1,1]_z\) in the base of \(P_A\)
gives the boundary relation
\(
   w\,\delta_-\,\beta\,\overline{\delta_+}=1.
\)
Hence
\(
   \beta=\overline{\delta_-}\,w^{-1}\delta_+.
\)
Transporting \(\beta\) to the basepoint on \(F_+\) by \(\delta_+\) gives
\(
   \delta_+\,\beta\,\overline{\delta_+}
   =
   \delta_+\overline{\delta_-}\,w^{-1}.
\)
Because \(\delta_-\overline{\delta_+}\) represents \(a\), the loop
\(\delta_+\overline{\delta_-}\) represents \(a^{-1}\).  Therefore the
transverse boundary class of \(T^K_-\), expressed at the fixed basepoint, is
\(
   a^{-1}w^{-1}.
\)
The fiber class is again \(h\).  Thus
\(
   \pi_1T^K_-=
   \langle h,a^{-1}w^{-1}\rangle .
\)
Since these are boundary tori of a Seifert manifold, the displayed generators
commute; this also follows directly from the rectangle model.
\end{proof}

Now glue the two boundary tori of \(W_R\) to the two boundary tori of
\(W_K\) by orientation-reversing diffeomorphisms
\[
   T^R_+\longrightarrow T^K_+,
   \qquad
   T^R_-\longrightarrow T^K_-
\]
satisfying
\[
   d_+\longmapsto m_+,
   \qquad
   u\longmapsto w,
\]
and
\[
   d_-\longmapsto m_-,
   \qquad
   u\longmapsto a^{-1}w^{-1}.
\]
We choose the gluing maps so that the circle
\(d_\pm\times\{u_0\}\subset T^R_\pm\) is identified with the corresponding
boundary circle of \(A\times\{0\}\subset W_K\) lying in \(T^K_\pm\).  Define
\[
   N_g^K=W_R\cup_{T_+\sqcup T_-}W_K.
\]
The two gluing tori will be denoted \(T_+\) and \(T_-\).

\begin{lemma}
\label{lem:closed-graph-manifold-topology}
The manifold \(N_g^K\) is closed, orientable, irreducible, and aspherical.  It
is a graph manifold, and \(T_+\) and \(T_-\) are nonseparating JSJ tori.
\end{lemma}

\begin{proof}
The two boundary tori of \(W_R\) are glued to the two boundary tori of
\(W_K\), so \(N_g^K\) is closed.  The gluing maps are orientation-reversing on
boundary tori, so the closed manifold is orientable.

Both pieces are aspherical Seifert manifolds with incompressible torus
boundary.  For \(W_R=R\times S^1\), this follows from \(\chi(R)<0\).  For
\(W_K\), the base orbifold is an annulus with two cone points of order two;
it has nonempty boundary and negative orbifold Euler characteristic.  Hence
\(W_K\) has contractible universal cover and incompressible boundary tori.
Gluing the two pieces along \(\pi_1\)-injective tori gives a Bass--Serre
tree of spaces whose vertex spaces and edge spaces in the universal cover are
contractible.  Therefore the universal cover of \(N_g^K\) is contractible.
Thus \(N_g^K\) is aspherical and irreducible.

The pieces are Seifert fibered, so \(N_g^K\) is a graph manifold.  The
Seifert fiber of \(W_R\) is the \(u\)-slope.  Under the gluing it maps to the
transverse slopes \(w\) and \(a^{-1}w^{-1}\), not to the regular fiber
\(m_\pm\), of \(W_K\).  Conversely, the regular fiber \(m_\pm\) of \(W_K\)
maps to the boundary curve \(d_\pm\) of \(R\), not to the \(u\)-fiber of
\(W_R\).  The Seifert fibrations therefore do not match across either gluing
torus.  Since neither piece is a solid torus or \(T^2\times I\), these tori
are JSJ tori by the Jaco--Shalen--Johannson decomposition theory for graph
manifolds \cite{JacoShalen1979,Johannson1979}.

The JSJ graph has two vertices joined by two edges.  Removing either edge
leaves the graph connected.  Equivalently, cutting \(N_g^K\) along either
\(T_+\) or \(T_-\) leaves the two Seifert pieces connected through the other
gluing torus.  Hence both \(T_+\) and \(T_-\) are nonseparating.  In
particular, \([T_\pm]\neq0\in H_2(N_g^K;\Z)\).
\end{proof}

\begin{figure}[ht]
\centering
\begin{tikzpicture}[scale=1]

\node[draw,circle,minimum size=1.15cm,inner sep=0pt] (WR) at (0,0) {$W_R$};
\node[draw,circle,minimum size=1.15cm,inner sep=0pt] (WK) at (5.5,0) {$W_K$};

\draw[thick] (WR) to[bend left=28] node[midway,above] {$T_+$} (WK);
\draw[thick] (WR) to[bend right=28] node[midway,below] {$T_-$} (WK);

\end{tikzpicture}
\caption{The JSJ graph of \(N_g^K\).  The two vertices correspond to the
Seifert-fibered pieces
\(
W_R=R\times S^1_u\text{ and }
W_K,
\)
where \(R\) is the genus-\((g-1)\) surface with boundary
\(d_+\sqcup d_-\), and \(W_K\) is the Seifert piece built from the twisted
\(I\)-bundle \(V\) over the Klein bottle together with
\(P_A=A\times[-1,1]\).  The two edges correspond to the gluing tori
\(T_+\) and \(T_-\).  The gluing maps are determined by
\(
d_+\mapsto m_+ \text{ and } u\mapsto w
\)
on \(T_+\), and
\(
d_-\mapsto m_-\text{ and } u\mapsto a^{-1}w^{-1}
\)
on \(T_-\).  Thus the JSJ graph has two vertices and two parallel edges, so
\(N_g^K\) is a closed graph manifold.  Moreover each torus \(T_\pm\) is
nonseparating, since deleting one edge leaves the graph connected.}
\label{fig:jsj-graph-ngk}
\end{figure}
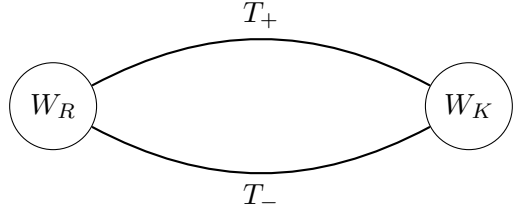
Define
\[
   \Sigma=(R\times\{u_0\})\cup(A\times\{0\})\subset N_g^K,
\]
where \(A\times\{0\}\) means \(A\times\{0\}_z\subset P_A\subset W_K\).  The
chosen gluing maps identify the boundary circles of \(R\times\{u_0\}\) with
the boundary circles of \(A\times\{0\}\).  Thus \(\Sigma\) is a closed
embedded genus-\(g\) surface.  Put
\(
   G=\pi_1N_g^K\) and \(H=\pi_1\Sigma\le G\).

\begin{lemma}
\label{lem:closed-surface-properties}
The surface \(\Sigma\subset N_g^K\) is closed, two-sided, nonseparating, and
\(\pi_1\)-injective.
\end{lemma}

\begin{proof}
The surface is obtained by gluing the annulus \(A\times\{0\}\) to the two
boundary components of \(R\times\{u_0\}\), hence is homeomorphic to
\(\Sigma_g\).  It is embedded by the compatibility of the gluing maps, and it
is two-sided because it is orientable in the orientable three-manifold
\(N_g^K\).

It is nonseparating because the curve
\[
   \gamma_u=\{p\}\times S^1_u\subset R\times S^1_u,
   \qquad p\in\operatorname{int}R,
\]
intersects \(R\times\{u_0\}\), and hence \(\Sigma\), algebraically once.

We prove \(\pi_1\)-injectivity by normal form.  Write
\(
   \Sigma=R\cup A,
\)
where \(A\) is the annulus attached to the two boundary components
\(d_+\) and \(d_-\) of \(R\).  View this as a graph of spaces whose
underlying graph has two vertices, corresponding to \(R\) and \(A\), and two
parallel edges, corresponding to the two boundary circles \(d_+\) and
\(d_-\).  Choose the edge corresponding to \(d_+\) as a spanning-tree edge.
Since \(A\) is an annulus, \(\pi_1A\cong\mathbb Z\), generated by the
annular core, and the tree-edge relation identifies this core with
\(d_+\in\pi_1R\).  The other edge, corresponding to \(d_-\), is the single
non-tree edge and therefore contributes a stable letter \(v\).  Its edge
relation identifies the \(d_-\)-boundary subgroup of \(R\), after crossing
the annulus, with the \(d_+\)-boundary subgroup.  With the chosen orientations
this gives
\(
   \pi_1\Sigma
   \cong
   \langle \pi_1R, v\mid v d_- v^{-1}=d_+\rangle .
\)
Geometrically, \(v\) is represented by a based loop which crosses the annulus: it follows a path in \(R\) to \(d_+\), then a transverse arc
\(\{x_0\}\times[-1,1]\subset A\times\{0\}\) from the \(d_+\)-boundary
component to the \(d_-\)-boundary component, and then a path in \(R\) back to the basepoint.

For the ambient splitting, choose the edge torus \(T_+\) in a maximal tree.
Then \(G\) is an HNN extension of the amalgamated group
\[
   G_0=
   \pi_1W_R*_{\langle d_+,u\rangle=\langle h,w\rangle}\pi_1W_K
\]
with stable letter, again denoted \(v\), for the edge \(T_-\).  The associated
subgroups are
\[
   \langle d_-,u\rangle\le \pi_1W_R
   \quad\text{and}\quad
   \langle h,a^{-1}w^{-1}\rangle\le \pi_1W_K,
\]
and the relation includes
\[
   v d_- v^{-1}=h=d_+,
   \qquad
   v u v^{-1}=a^{-1}w^{-1}.
\]
The inclusion \(R\hookrightarrow W_R=R\times S^1\) is \(\pi_1\)-injective.
The annulus \(A\times\{0\}\hookrightarrow W_K\) is also \(\pi_1\)-injective,
because its core is the regular fiber \(h=t^2\), which has infinite order in
\(\pi_1W_K\).

Let
\(
   \iota_*:\pi_1\Sigma\longrightarrow G
\)
be the homomorphism induced by inclusion.  We show that
\(\ker\iota_*=1\).  Let \(1\ne x\in\pi_1\Sigma\).  Using the HNN
description
\(
   \pi_1\Sigma
   \cong
   \langle \pi_1R, v\mid v d_- v^{-1}=d_+\rangle,
\)
write \(x\) in reduced HNN normal form.  If this reduced word contains no
stable letter \(v^{\pm1}\), then \(x\) is a nontrivial element of
\(\pi_1R\).  Its image is nontrivial in \(\pi_1W_R=\pi_1R\times\langle u\rangle\),
and hence also nontrivial in \(G\), by the normal-form theorem for the
ambient graph-of-groups decomposition.

It remains to consider the case where the reduced normal form of \(x\)
contains at least one stable letter.  With the relation
\(v d_-v^{-1}=d_+\), a reduced word in the surface HNN extension has no
subword \(v r v^{-1}\) with \(r\in\langle d_-\rangle\), and no subword
\(v^{-1}rv\) with \(r\in\langle d_+\rangle\).  In the ambient HNN
extension the corresponding associated subgroups are
\[
   A_-:=\langle d_-,u\rangle\le\pi_1W_R,
   \qquad
   B_-:=\langle h,a^{-1}w^{-1}\rangle\le\pi_1W_K.
\]
The first required intersection is immediate from
\(\pi_1W_R=\pi_1R\times\langle u\rangle\):
\[
   (\pi_1R\times\{1\})\cap A_-=\langle d_-\rangle.
\]

We calculate the second intersection inside the amalgamated vertex group
\[
   G_0=
   \pi_1W_R*_{E}\pi_1W_K,
   \qquad
   E=\langle d_+,u\rangle=\langle h,w\rangle.
\]
Amalgam normal form gives
\(\pi_1W_R\cap\pi_1W_K=E\) in \(G_0\).  It remains to calculate the
intersection of the two boundary-torus subgroups
\(\langle h,w\rangle\) and
\(\langle h,a^{-1}w^{-1}\rangle\) in \(\pi_1W_K\).
In the presentation
\[
   \pi_1W_K=
   \langle a,t,w\mid tat^{-1}=a^{-1},\ [w,t^2]=1\rangle ,
\]
the element \(h=t^2\) is central: it commutes with \(a\) because
\(t^2at^{-2}=a\), with \(t\) trivially, and with \(w\) by the
defining relation.
Quotienting by \(\langle h\rangle\) gives
\[
   \pi_1W_K/\langle h\rangle
   \cong
   \langle a,\bar t\mid \bar t^2=1,\
      \bar t a\bar t^{-1}=a^{-1}\rangle
      *\langle w\rangle
   \cong D_\infty*\langle w\rangle .
\]
The images of the two boundary subgroups are respectively
\(\langle w\rangle\) and \(\langle a^{-1}w^{-1}\rangle\).
Their intersection is trivial.  Indeed, for \(m\ne0\), the reduced free
product normal form of \(w^m\) has one syllable, whereas, for \(n\ne0\),
the reduced normal form of \((a^{-1}w^{-1})^n\) has \(2|n|\) syllables.
Therefore
\[
   \langle h,w\rangle
   \cap\langle h,a^{-1}w^{-1}\rangle
   =\langle h\rangle
   \qquad\text{in }\pi_1W_K.
\]
Since
\((\pi_1R\times\{1\})\cap E=\langle d_+\rangle\) in
\(\pi_1W_R\), while \(h=d_+\) in \(G_0\), the amalgam intersection
and the preceding boundary-subgroup calculation give
\[
   (\pi_1R\times\{1\})\cap B_-=\langle d_+\rangle
   \qquad\text{in }G_0.
\]

The two displayed associated-subgroup intersections show that neither type of
surface-HNN-reduced subword becomes a pinch in the ambient HNN extension.
Hence the image of the reduced HNN word for \(x\) is still reduced in the
ambient HNN extension defining \(G\).  By the Bass--Serre normal-form
theorem, equivalently Britton's lemma for this HNN extension, a nonempty
reduced HNN word is nontrivial.  Thus \(\iota_*(x)\ne1\).

Since every nontrivial element \(x\in\pi_1\Sigma\) has nontrivial image in
\(G\), the kernel of \(\iota_*\) is trivial.  Therefore
\(\pi_1\Sigma\to G\) is injective.
\end{proof}

\begin{lemma}
\label{lem:closed-bands-and-quotient}
For every \(n\in\Z\), the complement
\(N_g^K\setminus\operatorname{Int}\nu\Sigma\) contains an embedded M\"obius
band \(B_n\) with
\[
   \partial B_n=c^+,
   \qquad
   t_n=a^nt,
   \qquad
   t_n^2=c\in H,
   \qquad
   t_n\notin H.
\]
Moreover, for \(n\ge1\), the double-cosets
\(
   \mathcal D_n=Ht_nH
\)
are pairwise distinct, nontrivial, and self-dual.
\end{lemma}

\begin{proof}
Use the M\"obius bands \(B_n^V\subset V\) from
Lemma~\ref{lem:klein-bottle-I-bundle}, choosing their common boundary fiber
inside the interior of the annulus \(A_+\subset\partial V\).  Attach to
\(B_n^V\) the collar annulus
\[
   S^1_x\times\{0\}_y\times[\epsilon,1]_z\subset A\times[-1,1]_z
\]
for a small \(\epsilon>0\).  The curve
\(
   c^+=S^1_x\times\{0\}_y\times\{\epsilon\}
\)
is a normal push-off of
\(c=S^1_x\times\{0\}_y\times\{0\}\subset A\times\{0\}\subset\Sigma\).  After
smoothing the corner at \(z=1\), the union is a properly embedded M\"obius
band
\(
   B_n\subset N_g^K\setminus\operatorname{Int}\nu\Sigma.
\)
The collar lies in the region \(z\ge\epsilon\), while \(\Sigma\cap P_A\) is
\(A\times\{0\}\), so the band is disjoint from \(\operatorname{Int}\nu\Sigma\)
and has boundary \(c^+\).

The core may be chosen in the original \(B_n^V\), so it represents
\(t_n=a^nt\).  By Lemma~\ref{lem:klein-bottle-I-bundle},
\(
   (a^nt)^2=t^2.
\)
In \(W_K\), the regular fiber \(t^2=h\) is identified with the curve
\(c\subset A\times\{0\}\subset\Sigma\).  Therefore
\(
   t_n^2=c\in H.
\)

It remains to prove \(t_n\notin H\) and label distinctness.  Use the
Bass--Serre presentation of \(G\) described in
Lemma~\ref{lem:closed-surface-properties} and Lemma~\ref{lem:WK-boundary-subgroups}:
\[
\begin{aligned}
G=\langle &\pi_1R,u,a,t,w,v\mid [u,\pi_1R]=1,
   \ tat^{-1}=a^{-1},\ [w,t^2]=1,\\
& d_+=t^2,
   \ u=w,
   \ v d_- v^{-1}=t^2,
   \ vuv^{-1}=a^{-1}w^{-1}\rangle,
\end{aligned}
\]
with the usual surface relation inside \(\pi_1R\).  Define
\[
   q:G\longrightarrow
   D_\infty=\langle r,s\mid s^2=1,\ srs^{-1}=r^{-1}\rangle
\]
by
\[
   q(\pi_1R)=1,
   \qquad
   q(v)=1,
   \qquad
   q(u)=r,
   \qquad
   q(w)=r,
   \qquad
   q(t)=s,
   \qquad
   q(a)=r^{-2}.
\]
This is well defined.  The relation \(tat^{-1}=a^{-1}\) maps to
\(
   s r^{-2}s^{-1}=r^2=q(a^{-1}),
\)
and \([w,t^2]=1\) maps to \([r,s^2]=1\).  The \(T_+\)-gluing relations
\(d_+=t^2\) and \(u=w\) map to \(1=s^2\) and \(r=r\).  The \(T_-\)-gluing
relations \(v d_- v^{-1}=t^2\) and \(vuv^{-1}=a^{-1}w^{-1}\) map to
\(1=s^2\) and
\[
   r=q(u)=q(a^{-1}w^{-1})=r^2r^{-1}=r.
\]
Thus all gluing relations are respected.

The subgroup \(H=\pi_1\Sigma\) is generated by \(\pi_1R\) and the surface
stable letter \(v\).  Both are killed by \(q\), so \(q(H)=1\).  Moreover
\(
   q(G)=\langle r,s\rangle=D_\infty,
\)
and
\(
   q(t_n)=q(a^nt)=r^{-2n}s.
\)
The elements \(r^{-2n}s\) are pairwise distinct in \(D_\infty\).  Hence, if
\(Ht_nH=Ht_mH\), then applying \(q\) gives \(r^{-2n}s=r^{-2m}s\), so
\(n=m\).  Thus the double-cosets \(Ht_nH\) are pairwise distinct.  Since
\(q(t_n)\neq1\), we also have \(t_n\notin H\), and the labels are nontrivial.

Finally, \(t_n^2=c\in H\) implies \(t_n^{-1}=t_nc^{-1}\), so
\(
   Ht_n^{-1}H=Ht_nH.
\)
Thus each \(\mathcal D_n=Ht_nH\) is self-dual.
\end{proof}

\begin{lemma}
\label{lem:closed-nonfiber-normalizer}
The surface \(\Sigma\) is not a fiber of a fibration \(N_g^K\to S^1\), and
\(
   N_G(H)=H.
\)
\end{lemma}

\begin{proof}
The quotient \(q:G\to D_\infty\) from Lemma~\ref{lem:closed-bands-and-quotient}
shows that \(\Sigma\) is not a fiber.  If \(\Sigma\) were the fiber of a
fibration \(N_g^K\to S^1\), then the fibration exact sequence would give
\[
   1\longrightarrow H\longrightarrow G\longrightarrow \Z\longrightarrow1,
\]
so \(H\triangleleft G\) and \(G/H\cong\Z\).  Since \(q(H)=1\), the map
\(q\) would factor through \(G/H\cong\Z\), and its image would be cyclic.
But \(q(u)=r\) and \(q(t)=s\), so \(q(G)=D_\infty\), which is not cyclic.
This contradiction proves that \(\Sigma\) is not a fiber.

Heil's Theorem~1 states that if \(F\) is a closed two-sided
incompressible surface in a \(P^2\)-irreducible three-manifold \(M\), and
the normalizer of \(i_*\pi_1F\) is strictly larger than \(i_*\pi_1F\), then
one of the following holds: \(M\) fibers over \(S^1\) with fiber \(F\);
\(M\) is a twisted line bundle over a closed surface; or \(F\) separates
\(M\) and a complementary component is such a twisted line bundle
\cite[Theorem~1, p.~147]{Heil1981}.

Here \(N_g^K\) is closed and orientable by construction, and it is
irreducible and aspherical by
Lemma~\ref{lem:closed-graph-manifold-topology}.  In particular it is
\(P^2\)-irreducible: an embedded projective plane would give, by the
projective-plane theorem, an element of order two in \(\pi_1N_g^K\), while
\(\pi_1N_g^K\) is torsion-free because \(N_g^K\) is aspherical.  The surface
\(\Sigma\) is closed, two-sided, nonseparating, and incompressible by
Lemma~\ref{lem:closed-surface-properties}.  The twisted-line-bundle
alternative for the whole ambient manifold is impossible because such a
line-bundle total space is not a closed three-manifold, and the complementary
twisted-line-bundle alternative
is impossible because \(\Sigma\) is nonseparating.  Thus a normalizer
strictly larger than \(H\) would force \(\Sigma\) to be a fiber, contrary to
the first paragraph.  Therefore
\(
   N_G(H)=H.
\)
\end{proof}

Let \(T=T_+\), and let \(\lambda\subset T\) be the primitive curve in the
\(S^1_u\)-direction inherited from \(W_R\).  Let
\[
   \tau=\tau_{T,\lambda}:N_g^K\to N_g^K
\]
be the Dehn twist supported in a product collar of \(T\), twisting in the
\(\lambda\)-direction.  Such torus twists are the standard Dehn
homeomorphisms of Johannson \cite{Johannson1979}.

\begin{lemma}
\label{lem:closed-torus-twist-homology}
After fixing the orientation convention for \(\tau\), every class
\([S]\in H_2(N_g^K;\Z)\) satisfies
\[
   \tau_*[S]=[S]+([\lambda]\cdot[S])[T].
\]
For the surface \(\Sigma\),
\[
   \tau_*^k[\Sigma]=[\Sigma]\pm k[T].
\]
\end{lemma}

\begin{proof}
It is enough to compute the dual action on \(H_1(N_g^K;\Z)\).  Let \(x\) be
an oriented closed curve transverse to \(T\).  Each positive intersection of
\(x\) with \(T\) acquires one copy of the twist curve \(\lambda\) after
applying \(\tau\), and each negative intersection acquires one copy with the
opposite orientation.  Thus, up to the global sign determined by the chosen
orientation convention for the twist,
\[
   \tau_*[x]=[x]+([x]\cdot[T])[\lambda].
\]
Equivalently,
\[
   \tau_*^{-1}[x]=[x]-([x]\cdot[T])[\lambda]
\]
with the same convention.

Let \([S]\in H_2(N_g^K;\Z)\).  For every \([x]\in H_1(N_g^K;\Z)\),
intersection naturality gives
\[
   [x]\cdot\tau_*[S]=\tau_*^{-1}[x]\cdot[S].
\]
Using the displayed formula for \(\tau_*^{-1}\), and absorbing the global
sign into the convention for \(\lambda\), gives
\[
   [x]\cdot\tau_*[S]
   =
   [x]\cdot[S]+([x]\cdot[T])([\lambda]\cdot[S]).
\]
Since this holds for every \([x]\), Poincare duality gives
\[
   \tau_*[S]=[S]+([\lambda]\cdot[S])[T].
\]

Now \(\Sigma\cap T\) is the curve \(m_+\), the regular fiber of \(W_K\) and
the \(d_+\)-curve from \(W_R\).  On the torus \(T\), the curves \(m_+\) and
\(\lambda\) have algebraic intersection \(\pm1\).  Therefore
\([\lambda]\cdot[\Sigma]=\pm1\), and the iterated formula is
\(
   \tau_*^k[\Sigma]=[\Sigma]\pm k[T].
\)
\end{proof}

\begin{lemma}
\label{lem:closed-active-twist}
For every \(k\ne0\),
\(
   \tau_*^k[\Sigma]\ne \pm[\Sigma]
   \text{ in }H_2(N_g^K;\Z).
\)
Consequently \(\tau_*^k(H)\) is not conjugate to \(H\) in \(G\) for every
\(k\ne0\).
\end{lemma}

\begin{proof}
The torus \(T\) is nonseparating by
Lemma~\ref{lem:closed-graph-manifold-topology}, so \([T]\ne0\).  We first
show that \([\Sigma]\) and \([T]\) are linearly independent.  The curve
\[
   \gamma_u=\{p\}\times S^1_u\subset R\times S^1_u,
   \qquad p\in\operatorname{int}R,
\]
is disjoint from \(T\) and intersects \(\Sigma\) algebraically once.  Hence
\[
   \gamma_u\cdot[\Sigma]=\pm1,
   \qquad
   \gamma_u\cdot[T]=0.
\]
A relation \(A[\Sigma]+B[T]=0\) therefore gives \(A=0\) after intersecting
with \(\gamma_u\).  Then \(B[T]=0\).  Since \(N_g^K\) is a closed orientable
three-manifold, Poincare duality and the universal coefficient theorem give
\[
   H_2(N_g^K;\Z)\cong H^1(N_g^K;\Z)
      \cong \operatorname{Hom}(H_1(N_g^K;\Z),\Z),
\]
which is torsion-free.  Because \([T]\ne0\), this implies \(B=0\).  Thus
\([\Sigma]\) and \([T]\) are linearly independent.

By Lemma~\ref{lem:closed-torus-twist-homology},
\(
   \tau_*^k[\Sigma]=[\Sigma]\pm k[T].
\)
Linear independence implies that this is never equal to \([\Sigma]\) or to
\(-[\Sigma]\) when \(k\ne0\).

Suppose now that \(\tau_*^k(H)\) were conjugate to \(H\) in \(G\).  Let
\(i:\Sigma_g\hookrightarrow N_g^K\) denote the inclusion.  Then there is
\(\gamma\in G\) such that
\(
   c_\gamma\bigl(\tau_*^k(H)\bigr)=H.
\)
Hence
\[
   i_*^{-1}\circ c_\gamma\circ \tau_*^k\circ i_*:
   \pi_1\Sigma_g\longrightarrow \pi_1\Sigma_g
\]
is an automorphism.  Since \(\Sigma_g\) is aspherical, it is a \(K(H,1)\).  Hence the automorphism
\(
   i_*^{-1}\circ c_\gamma\circ \tau_*^k\circ i_*
\)
is realized by a self-map \(\phi:\Sigma_g\to\Sigma_g\).  Since this
homomorphism is an isomorphism and \(\Sigma_g\) has no higher homotopy
groups, Whitehead's theorem implies that \(\phi\) is a self-homotopy
equivalence.  The maps \(\tau^k\circ i\) and
\(i\circ\phi\) then induce conjugate homomorphisms to
\(\pi_1N_g^K\).  Since \(N_g^K\) is aspherical, they are freely homotopic.
Therefore
\[
   \tau_*^k[\Sigma]
   =
   (\tau^k\circ i)_*[\Sigma_g]
   =
   (i\circ\phi)_*[\Sigma_g]
   =
   \deg(\phi)[\Sigma]
   =
   \pm[\Sigma].
\]
This contradicts the first part of the proof, which showed that
\(\tau_*^k[\Sigma]\ne\pm[\Sigma]\) for \(k\ne0\).  Thus
\(\tau_*^k(H)\) is not conjugate to \(H\) in \(G\) for every \(k\ne0\).
\end{proof}

We now collect the needed mapping-torus group facts for later use.

\begin{lemma}
\label{lem:closed-mapping-torus-group-facts}
Let
\[
   Y_\tau=N_g^K\rtimes_\tau S^1=N_g^K\times[0,1]/(x,1)\sim(\tau(x),0)
\]
be the mapping torus.  Put
\(
   \Gamma=\pi_1Y_\tau .
\)
Let \(s\in\Gamma\) denote the class of the loop going once in the
\(S^1\)-direction. Put \(
   G=\pi_1N_g^K.
\)
We have
\[
   \Gamma
   \cong
   G\rtimes_{\tau_*}\langle s\rangle
   =
   \langle G,s\mid sgs^{-1}=\tau_*(g)\text{ for all }g\in G\rangle .
\]
Then \(Y_\tau\) is aspherical, \(\Gamma\) is torsion-free,
\(
   N_\Gamma(H)=H,
\)
and the double-cosets
\[
   \mathcal D_n=Ht_nH\in H\backslash\Gamma/H,
   \qquad
   t_n=a^nt,\qquad n\ge1,
\]
are pairwise distinct, nontrivial, and self-dual.
\end{lemma}

\begin{proof}
The mapping torus \(Y_\tau\) is aspherical because its infinite cyclic cover
is \(N_g^K\times\mathbb R\), and its universal cover is
\(\widetilde N_g^K\times\mathbb R\), which is contractible by
Lemma~\ref{lem:closed-graph-manifold-topology}.

The group \(\Gamma\) is torsion-free.  If \(gs^k\in\Gamma\) has finite
order, its image under the projection \(\Gamma\to\mathbb Z\) has finite
order, so \(k=0\).  The element then lies in \(G\), which is torsion-free because \(N_g^K\) is a closed aspherical manifold.

For \(h\in H\),
\(
   (gs^k)h(gs^k)^{-1}=g\tau_*^k(h)g^{-1}.
\)
Thus \(gs^k\in N_\Gamma(H)\) if and only if
\(
   g\tau_*^k(H)g^{-1}=H.
\)
If \(k\ne0\), this contradicts Lemma~\ref{lem:closed-active-twist}.  Hence
\(k=0\), and then \(g\in N_G(H)=H\) by
Lemma~\ref{lem:closed-nonfiber-normalizer}.  Therefore
\(
   N_\Gamma(H)=H.
\)

By Lemma~\ref{lem:closed-bands-and-quotient}, the labels
\(\mathcal D_n=Ht_nH\) are pairwise distinct, nontrivial, and self-dual in
\(H\backslash G/H\).  They remain pairwise distinct in
\(H\backslash\Gamma/H\): if
\[
   Ht_nH=Ht_mH
   \quad\text{in }H\backslash\Gamma/H,
\]
then \(t_n=h_1t_mh_2\) for some \(h_1,h_2\in H\).  Since
\(t_n,t_m,h_1,h_2\in G\), the same equality already holds in
\(H\backslash G/H\), and hence \(n=m\).  Nontriviality and self-duality also
hold in \(H\backslash\Gamma/H\), since \(t_n\notin H\) and
\(t_n^2=c\in H\).
\end{proof}

\begin{theorem}[Unstabilized closed mapping-torus nonconcordance]
\label{thm:unstabilized-closed-mapping-torus-family}
For every \(g\ge2\), with the closed graph-manifold data
\[
   N_g^K,\qquad
   \Sigma=\Sigma_g\subset N_g^K,\qquad
   H=\pi_1\Sigma\le G=\pi_1N_g^K,
\]
and the Dehn twist \(\tau:N_g^K\to N_g^K\) constructed above, put
\[
   Y_\tau=N_g^K\rtimes_\tau S^1,
   \qquad
   \Gamma=\pi_1Y_\tau .
\]
Then \(Y_\tau\) contains a base \(\rho\)-marked embedded surface
\(
   F_0:\Sigma_g\hookrightarrow Y_\tau
\)
and \(\rho\)-marked embedded surfaces
\(
   F_n:\Sigma_g\hookrightarrow Y_\tau\), \(n\ge1,
\)
such that:
\begin{enumerate}[label=\textup{(\roman*)}]
\item all \(F_n\) are homotopic to \(F_0\);
\item \(F_{n*}\pi_1\Sigma_g=H\le\Gamma\) for all \(n\ge0\);
\item for \(n\ge1\), there are pairwise distinct nontrivial self-dual double
cosets
\[
   \mathcal D_n=Ht_nH\in H\backslash\Gamma/H,
   \qquad
   t_n=a^nt,
   \qquad
   t_n^2=c\in H;
\]
\item for every \(i\ge1\),
\(
   Q^{\mathcal D_i}_\rho(Y_\tau,F_0)
   \cong
   \Ftwo\langle u_{\mathcal D_i}\rangle;
\)
\item for all \(i,n\ge1\),
\[
   FQ^{\mathcal D_i}_\rho(F_0,F_n)
   =
   \begin{cases}
   u_{\mathcal D_i},& i=n,\\
   0,& i\ne n.
   \end{cases}
\]
\end{enumerate}
Consequently the embedded images of
\(
   F_0,F_1,F_2,\ldots
\)
are pairwise not smoothly image-concordant in \(Y_\tau\).
\end{theorem}

\begin{proof}
Let \(F_0\) be the original surface
\(
   \Sigma=\Sigma_g\subset N_g^K\subset Y_\tau
\)
with the marking fixed in the construction above.  The fiber
\(N_g^K\subset Y_\tau\) has a product neighbourhood
\(
   N_g^K\times(-\epsilon,\epsilon)\subset Y_\tau .
\)
For each \(n\ge1\), choose \(0<\delta<\epsilon\) and work inside the closed
subproduct
\[
   N_g^K\times[-\delta,\delta]
   \subset
   N_g^K\times(-\epsilon,\epsilon)
   \subset Y_\tau .
\]
After rescaling \([-\delta,\delta]\) to \([-1,1]\),
Lemma~\ref{lem:mobius-band-square-root-block} applies to the M\"obius band
\(
   B_n\subset N_g^K\setminus\operatorname{Int}\nu\Sigma .
\)
Thus this closed subproduct contains a standard square-root neighbourhood for
\(c\) along \(F_0\), with square-root element
\[
   t_n=a^nt,
   \qquad
   t_n^2=c\in H .
\]
The construction is compactly supported in the interior of
\(N_g^K\times[-\delta,\delta]\), so it defines a local track after inclusion
into \(Y_\tau\).  The support is chosen away from the basepoint and base
path, hence the track is \(\rho\)-marked.

Put
\(
   \mathcal D_n=Ht_nH\in H\backslash\Gamma/H .
\)
By Lemma~\ref{lem:closed-mapping-torus-group-facts}, the labels
\(\mathcal D_n\), \(n\ge1\), are pairwise distinct, nontrivial, and
self-dual in \(H\backslash\Gamma/H\), and \(\Gamma\) is torsion-free.
Together with Lemma~\ref{lem:square-root-crossed-block}, the track associated
to \(B_n\) is a \(\mathcal D_n\)-crossed track with raw value
\(u_{\mathcal D_n}\).

Let \(F_n\) be the top surface of this track.  Then \(F_n\) is embedded,
homotopic to \(F_0\), and has marked image subgroup \(H\).

Since \(Y_\tau\) is aspherical by
Lemma~\ref{lem:closed-mapping-torus-group-facts}, we have
\(
   \pi_2Y_\tau=0.
\)
Therefore
\(
   H^1(\Sigma_g;(\pi_2Y_\tau)_\rho)=0.
\)
Because \(\Gamma\) is torsion-free, no label \(\mathcal D_i\) contains an
element of order two.  The survival theorem, Theorem~\ref{thm:B-survival},
gives
\(
   Q^{\mathcal D_i}_\rho(Y_\tau,F_0)
   \cong
   \Ftwo\langle u_{\mathcal D_i}\rangle
\)
for every \(i\ge1\).  Also \(\pi_3Y_\tau=0\), since \(Y_\tau\) is
aspherical; this is consistent with the absence of local
\(\pi_3Y_\tau\)-insertion classes, although the formal survival input is the
order-two exclusion and the vanishing of
\(H^1(\Sigma_g;(\pi_2Y_\tau)_\rho)\).

The distinguishing formula is the same local calculation as before.  The
track from \(F_0\) to \(F_n\) uses only the square-root neighbourhood associated to
\(B_n\).  By Lemma~\ref{lem:mobius-band-square-root-block} and
Lemma~\ref{lem:square-root-crossed-block}, this neighbourhood contributes exactly
one local component with label \(\mathcal D_n\), and no component with any
other nontrivial label.  Since the labels are pairwise distinct,
\[
   FQ^{\mathcal D_i}_\rho(F_0,F_n)
   =
   \begin{cases}
   u_{\mathcal D_i},& i=n,\\
   0,& i\ne n.
   \end{cases}
\]

It remains to promote the marked obstruction values to image nonconcordance.
By Lemma~\ref{lem:closed-mapping-torus-group-facts}, \(N_\Gamma(H)=H\).
Hence \(C_\Gamma(H)\subseteq H\), and the image of
\(
   N_\Gamma(H)\longrightarrow\Out(H)
\)
is trivial.  Since \(g\ge2\), the extended Dehn--Nielsen--Baer theorem
\cite[Theorem~8.1]{FarbMargalit2012} identifies
\(\Mod(\Sigma_g)\) with \(\Out(\pi_1\Sigma_g)\).  Thus the full
endpoint-rigidity set in Corollary~\ref{cor:full-endpoint-rigid-image} is
trivial.

The argument from the proof of
Theorem~\ref{thm:general-mobius-square-root-consequence} now applies, with
\(\pi\) replaced by \(\Gamma\).  The displayed formula gives a nonzero
\(\mathcal D_n\)-coordinate for \(F_n\), and the \(\mathcal D_m\)-coordinate
separates \(F_m\) from \(F_n\) when \(m\ne n\).  Hence the embedded images of
\(F_0,F_1,F_2,\ldots\) are pairwise not smoothly image-concordant in
\(Y_\tau\).
\end{proof}

The next theorem adds one \(S^2\times S^2\)-summand to the same local
construction in order to supply a common framed dual sphere and to kill the
meridian in the complements.

\begin{theorem}[Stabilized closed graph-manifold mapping-torus family]
\label{thm:closed-graph-manifold-mapping-torus-family}
For every \(g\ge2\), there exist a closed orientable aspherical graph
manifold \(N_g^K\), a closed two-sided nonseparating \(\pi_1\)-injective
surface
\[
   \Sigma=\Sigma_g\subset N_g^K,
   \qquad
   H=\pi_1\Sigma\le G=\pi_1N_g^K,
\]
and a Dehn twist \(\tau:N_g^K\to N_g^K\) along a nonseparating JSJ torus,
such that, for
\[
   Y_\tau=N_g^K\rtimes_\tau S^1,
   \qquad
   X_\tau=Y_\tau\#(S^2\times S^2),
\]
the four-manifold \(X_\tau\) contains a base \(\rho\)-marked surface
\(F_0:\Sigma_g\hookrightarrow X_\tau\) and \(\rho\)-marked surfaces
\(
   F_n:\Sigma_g\hookrightarrow X_\tau\), \(n\ge1,
\)
all homotopic to \(F_0\), such that:
\begin{enumerate}[label=\textup{(\roman*)}]
\item \(F_{n*}\pi_1\Sigma_g=H\) for all \(n\ge0\);
\item the surfaces have a common framed embedded dual sphere;
\item the complement maps
\(
   \pi_1(X_\tau\setminus\nu F_n)\longrightarrow \pi_1X_\tau
\)
are isomorphisms for all \(n\ge0\);
\item there are pairwise distinct nontrivial self-dual double-cosets
\[
   \mathcal D_n=Ht_nH\in H\backslash\Gamma/H
      =H\backslash\pi_1X_\tau/H,
   \qquad
   t_n=a^nt,
   \qquad
   t_n^2=c\in H,
   \qquad n\ge1;
\]
\item for every \(i\ge1\),
\(
   Q^{\mathcal D_i}_\rho(X_\tau,F_0)
   \cong
   \Ftwo\langle u_{\mathcal D_i}\rangle;
\)
\item for all \(i,n\ge1\),
\[
   FQ^{\mathcal D_i}_\rho(F_0,F_n)
   =
   \begin{cases}
   u_{\mathcal D_i},& i=n,\\
   0,& i\ne n.
   \end{cases}
\]
\end{enumerate}
Consequently the embedded images of \(F_0,F_1,F_2,\ldots\) are pairwise not
smoothly image-concordant.
\end{theorem}

\begin{proof}
Let
\(
   \Gamma=\pi_1Y_\tau .
\)
By Lemma~\ref{lem:closed-mapping-torus-group-facts}, \(Y_\tau\) is
aspherical, \(\Gamma\) is torsion-free, \(N_\Gamma(H)=H\), and the labels
\[
   \mathcal D_n=Ht_nH\in H\backslash\Gamma/H,
   \qquad
   t_n=a^nt,
   \qquad n\ge1,
\]
are pairwise distinct, nontrivial, and self-dual. Let
\(
   F_n^Y:\Sigma_g\hookrightarrow Y_\tau,
   \qquad n\ge0,
\)
be the unstabilized surfaces constructed in
Theorem~\ref{thm:unstabilized-closed-mapping-torus-family}.  Thus
\(F_n^Y\) is obtained from \(F_0^Y\) by a \(\rho\)-marked
\(\mathcal D_n\)-crossed track supported in the product neighbourhood of the
fiber \(N_g^K\subset Y_\tau\), and the local calculation gives
\[
   FQ^{0,\mathcal D_i}_\rho(F_0^Y,F_n^Y)
   =
   \begin{cases}
   u_{\mathcal D_i},& i=n,\\
   0,& i\ne n.
   \end{cases}
\]

Now form
\(
   X_\tau=Y_\tau\#(S^2\times S^2)
\)
by taking the connected sum away from the fiber product neighbourhood used by
the crossed tracks, and away from the basepoint and base path.  In the
stabilizing summand write
\[
   S_1=S^2\times\{\operatorname{pt}\},
   \qquad
   D_{\mathrm{dual}}=\{\operatorname{pt}\}\times S^2 .
\]
For each \(n\ge0\), define
\(
   F_n=F_n^Y\# S_1
   \subset X_\tau
\)
using the same fixed tube, chosen away from the marking data and from all
crossed-track supports.  Connected sum with \(S_1\) does not change the genus
or the marked surface subgroup, so
\(
   F_{n*}\pi_1\Sigma_g=H
\)
for every \(n\).  The tracks from \(F_0^Y\) to \(F_n^Y\) extend by the
stationary tube and the stationary sphere \(S_1\) to \(\rho\)-marked tracks
from \(F_0\) to \(F_n\) in \(X_\tau\).  Their moving support is unchanged,
so the raw \(\mathcal D_i\)-counts are the same as in the unstabilized
construction:
\[
   FQ^{0,\mathcal D_i}_\rho(F_0,F_n)
   =
   \begin{cases}
   u_{\mathcal D_i},& i=n,\\
   0,& i\ne n.
   \end{cases}
\]
Finally, \(D_{\mathrm{dual}}\) is a framed embedded dual sphere for every
\(F_n\), because the stabilization tube is fixed and each crossed track is
supported away from the stabilizing summand.

For every \(n\), the punctured dual sphere
\(
   D_{\mathrm{dual}}\setminus\operatorname{Int}\nu F_n
\)
is a disk in \(X_\tau\setminus\nu F_n\) bounding a meridian of \(F_n\).  By
Lemma~\ref{lem:dual-sphere-kills-meridian}, the complement map
\(
   \pi_1(X_\tau\setminus\nu F_n)\longrightarrow\pi_1X_\tau
\)
is an isomorphism.

The survival statement follows from
Lemma~\ref{lem:aspherical-finite-stabilization-survival}, applied to
\[
   Z=Y_\tau,
   \qquad
   \Pi=\Gamma .
\]
By Lemma~\ref{lem:closed-mapping-torus-group-facts}, \(Y_\tau\) is
aspherical and \(\Gamma\) is torsion-free.  Hence, for every \(i\ge1\),
\(
   Q^{\mathcal D_i}_\rho(X_\tau,F_0)
   \cong
   \Ftwo\langle u_{\mathcal D_i}\rangle .
\)

The tracks from the unstabilized theorem extend by the stationary tube and
the stationary sphere \(S_1\).  Their moving supports are unchanged and are
disjoint from the stabilizing summand.  Hence the local labelled fiber
products are the same as before, and the displayed Kronecker formula holds:
\[
   FQ^{\mathcal D_i}_\rho(F_0,F_n)
   =
   \begin{cases}
   u_{\mathcal D_i},& i=n,\\
   0,& i\ne n.
   \end{cases}
\]

Finally, \(\pi_1X_\tau=\Gamma\), and
Lemma~\ref{lem:closed-mapping-torus-group-facts} gives \(N_\Gamma(H)=H\).
Thus \(C_\Gamma(H)\subseteq H\), the image of
\(N_\Gamma(H)\to\Out(H)\) is trivial, and, by the same
Dehn--Nielsen--Baer argument used in the proof of
Theorem~\ref{thm:general-mobius-square-root-consequence}, the full
endpoint-rigidity set is trivial.  Therefore the argument from that proof applies with \(\pi\) replaced by \(\Gamma\).  The
displayed Kronecker formula proves that the embedded images of
\(F_0,F_1,F_2,\ldots\) are pairwise not smoothly image-concordant in
\(X_\tau\).

The dual sphere is not needed for the nonconcordance conclusion, but it is
essential for the additional conclusions of this theorem: it supplies the
common framed dual sphere and, after puncturing, kills a meridian in every
complement by Lemma~\ref{lem:dual-sphere-kills-meridian}.
\end{proof}

\begin{remark}[Persistence under further stabilization]
\label{rem:persistence-under-further-stabilization}
The image-nonconcordance conclusions of both
Theorem~\ref{thm:fixed-klein-bottle-infinite-family} and
Theorem~\ref{thm:closed-graph-manifold-mapping-torus-family} persist under
any finite number of further ambient \(S^2\times S^2\)-stabilizations.

Indeed, let \(X\) denote either \(X_{g,c}^{K}\) in the compact
Klein-bottle \(I\)-bundle family or \(X_\tau\) in the closed graph-manifold
mapping-torus family, and write
\(
   \Pi=\pi_1X.
\)
Thus \(\Pi=\pi_1M_{g,c}^{K}\) in the compact case and
\(\Pi=\Gamma=\pi_1Y_\tau\) in the closed case.  For \(k\ge0\), put
\[
   X^{(k)}=X\# k(S^2\times S^2),
\]
with all connected sums taken away from the relevant surfaces, construction
tracks, marking data, and common dual sphere.  Since \(S^2\times S^2\) is
simply connected,
\(
   \pi_1X^{(k)}\cong \pi_1X=\Pi.
\)
Hence the surface subgroup \(H\), the double-coset labels
\(
   \mathcal D_n=Ht_nH\in H\backslash\Pi/H,
\)
torsion-freeness of \(\Pi\), and the endpoint-rigidity hypothesis
\(N_\Pi(H)=H\) are unchanged.  The local tracks are also
unchanged, since their supports lie in the original summand.

The common framed dual sphere remains a common framed dual sphere in
\(X^{(k)}\).  Moreover, the complement \(\pi_1\)-isomorphisms persist:
adding simply connected summands away from the surfaces does not change the
fundamental group of either the ambient manifold or the surface complement,
and the punctured dual sphere still kills a meridian.

The survival statement also persists.  Indeed, each further stabilization is
still a finite stabilization of the same aspherical ambient \(Z\), and the
group \(\Pi\), subgroup \(H\), and labels \(\mathcal D_n\) are unchanged.
Thus Lemma~\ref{lem:aspherical-finite-stabilization-survival} applies to
\(X^{(k)}\), and gives
\(
   Q^{\mathcal D_i}_\rho(X^{(k)},F_0)
   \cong
   \Ftwo\langle u_{\mathcal D_i}\rangle
\)
for every \(i\ge1\).  The local tracks are unchanged, so the same Kronecker
formula holds.  Hence the embedded images of \(F_0,F_1,F_2,\ldots\) remain
pairwise not smoothly image-concordant after every finite further
\(S^2\times S^2\)-stabilization.
\end{remark}

\subsection{The orientation-double-cover family}
\label{sec:orientation-family}
We finish with the normal index-two specialization coming from the orientation
double cover of a closed nonorientable surface.  In this case the relevant
self-dual label is the nontrivial deck-transformation label.

Let \(N_{g+1}\) be the closed nonorientable surface of genus \(g+1\), and
let
\(
p:\Sigma_g\to N_{g+1}
\)
be its orientation double cover.  Put
\[
\pi=\pi_1(N_{g+1}),
\qquad
H=p_*\pi_1\Sigma_g=\ker w_1(N_{g+1}).
\]
Then \(H\triangleleft\pi\) and \(\pi/H\cong\Z/2\).  Let \(\tau\) denote the nontrivial deck transformation of the regular
\(H\)-cover \(X_H\to X_g\), equivalently the unique nontrivial index-two
label.  Let \(L\to N_{g+1}\) be the real line bundle with
\(w_1(L)=w_1(N_{g+1})\), set
\(
Y_g=D(L\oplus\varepsilon^1),
\)
and define
\[
X_g=Y_g\#(S^2\times S^2).
\]
The total space \(Y_g\) is oriented because
\( 
w_1(TY_g)=p_Y^*(w_1(N_{g+1})+w_1(L))=0,
\)
where \(p_Y:Y_g\to N_{g+1}\) is the bundle projection.  Since \(Y_g\)
deformation retracts to \(N_{g+1}\),
\(
\pi_1X_g\cong\pi.
\)

Let
\(
T=S_r(L)\subset\operatorname{Int}Y_g.
\)
The sphere bundle \(S(L)\to N_{g+1}\) is the orientation double cover, so
\(T\cong\Sigma_g\) and \(\pi_1T\to\pi_1Y_g\) has image \(H\).  Define
\(
F_0=T\#(S^2\times\{\mathrm{pt}\})
\)
inside \(X_g\), with the tube chosen away from the fixed basepoint,
base-path, and marking data.  The marking \(\rho\) is induced by the
orientation-cover identification \(\pi_1\Sigma_g\cong H\).  Let
\(
D_0=\{\mathrm{pt}\}\times S^2
\)
in the \(S^2\times S^2\)-summand.  Then \(D_0\) is a framed embedded dual
sphere and
\(
[F_0]\cdot[D_0]=1.
\)

\begin{lemma}
\label{lem:orientation-complement}
The complement map
\(
\pi_1(X_g\setminus\nu F_0)\to\pi_1X_g
\)
is an isomorphism.
\end{lemma}

\begin{proof}
The punctured dual sphere
\(
D_0\setminus\operatorname{Int}\nu F_0
\)
is a disk in \(X_g\setminus\nu F_0\) bounding a meridian of \(F_0\).
Lemma~\ref{lem:dual-sphere-kills-meridian} gives the complement
\(\pi_1\)-isomorphism.
\end{proof}

\begin{lemma}
\label{lem:orientation-survival}
For the orientation-double-cover family,
\(
\operatorname{im}\mu_\tau=0\) and
\(H^1(\Sigma_g;(\pi_2X_g)_\rho)=0.
\)
Consequently
\(
Q^\tau_\rho(X_g,F_0)\cong\Ftwo\langle u_\tau\rangle.
\)
\end{lemma}

\begin{proof}
The group \(\pi_1(N_{g+1})\) is torsion-free.  For \(g=1\) this follows
from the Klein-bottle normal form; for \(g\ge2\) the universal cover is
\(\mathbb H^2\) and deck transformations act freely.  Therefore no element
of \(\pi\setminus H\) has order two.  By the index-two specialization in
Corollary~\ref{cor:index-two} and Lemma~\ref{lem:muD-zero},
\(
\operatorname{im}\mu_\tau=0.
\)

Now apply Lemma~\ref{lem:aspherical-finite-stabilization-survival} with
\[
   Z=Y_g,\qquad
   \Pi=\pi_1Y_g=\pi,\qquad
   X=X_g,
\]
and with \(D=\pi\setminus H\), the unique nontrivial index-two label.  The
space \(Y_g\) deformation retracts to the aspherical surface \(N_{g+1}\), so
\(Y_g\) is aspherical, and the preceding paragraph shows that \(\pi\) is
torsion-free.  Hence
\(
   H^1(\Sigma_g;(\pi_2X_g)_\rho)=0.
\)
Under the index-two dictionary of Corollary~\ref{cor:index-two}, the label
\(D=\pi\setminus H\) is the deck label \(\tau\).  Therefore
\(
   Q^\tau_\rho(X_g,F_0)\cong\Ftwo\langle u_\tau\rangle .
\)
\end{proof}

\begin{figure}[ht]
\centering
\begin{tikzpicture}[
   scale=1,
   line cap=round,
   line join=round,
   every node/.style={font=\small},
   edge/.style={line width=0.45pt},
   main/.style={line width=0.65pt},
   guide/.style={line width=0.45pt,dashed},
   strip/.style={fill=black!8,draw=none},
   arrow/.style={-{Stealth[length=5pt,width=4pt]},line width=0.45pt}
]

\coordinate (A) at (0,-1.8);
\coordinate (B) at (6,-1.8);
\coordinate (C) at (6,1.8);
\coordinate (D) at (0,1.8);

\def\sp{0.62}
\def\sm{-0.32}

\fill[strip] (0,\sm) rectangle (6,\sp);

\draw[main] (A) rectangle (C);

\draw[guide] (0,0) -- (6,0);
\node[left] at (0,0) {\(s=0\)};

\draw[edge] (0,\sp) -- (6,\sp);
\draw[edge] (0,\sm) -- (6,\sm);
\node[left] at (0,\sp) {\(s_+=2\epsilon\)};
\node[left] at (0,\sm) {\(s_-=-\epsilon\)};

\node at (3,{(\sp+\sm)/2 + 0.16}) {\(\operatorname{pr}_{(x,s)}(W)\)};

\draw[arrow] (0.55,2.15) -- (1.55,2.15);
\node[above] at (1.05,2.15) {\(x\)};

\node[below] at (0,-1.8) {\(x=0\)};
\node[below] at (6,-1.8) {\(x=1\)};

\draw[arrow] (6.75,0.88) -- (6.75,-0.88);
\node[right,align=left] at (6.88,0)
   {\(\tau\) sends\\ \(s\mapsto -s\)};

\end{tikzpicture}
\caption{The cut-open annular chart in
Lemma~\ref{lem:orientation-crossed-block}.  The upstairs annulus
\(\widetilde U\cong S^1_x\times[-3,3]_s\) is cut open to
\([0,1]_x\times[-3,3]_s\); the boundary edges \(x=0\) and \(x=1\) are
re-identified when the annulus is recovered.  The Whitney rectangle \(W\)
projects to the strip between the two \(s\)-levels
\(s_+=2\epsilon\) and \(s_-=-\epsilon\).  The deck transformation acts by
\(\tau(x,s,u,v)=(x+\tfrac12,-s,-u,v)\).  The skew choice
\(s_++s_->0\), together with the reflected disk-coordinate path, is used to
show that \(W\cap\tau W=\varnothing\).}
\label{fig:orientation-crossed-block-chart}
\end{figure}
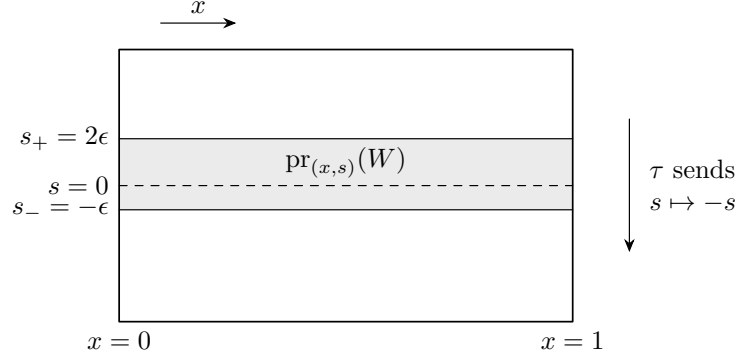

\begin{lemma}
\label{lem:orientation-crossed-block}
The surface \(F_0\subset X_g\) admits a \(D\)-crossed track for the
unique nontrivial index-two label, equivalently for the deck label \(\tau\).
The construction is supported away from \(D_0\) and the marking data, and it
produces a \(\rho\)-marked embedded surface \(F_1\), homotopic to \(F_0\),
with track \(G\) satisfying
\(
FQ^{0,\tau}_\rho(G)=u_\tau.
\)
Consequently
\(
FQ^\tau_\rho(F_0,F_1)=u_\tau\neq0.
\)
\end{lemma}

\begin{proof}
Choose a simple closed curve \(\ell\subset N_{g+1}\) whose regular
neighbourhood \(U\) is a M\"obius band.  Equivalently,
\(w_1([\ell])=1\).  Let
\(
   \eta=[\ell]\in\pi=\pi_1(N_{g+1}).
\)
Then \(\eta\notin H=\ker w_1(N_{g+1})\). Choose a M\"obius-band neighbourhood \(U\) away from the base path
and marking data.  Its orientation double cover is an annulus
\(
\widetilde U\cong S^1_x\times[-3,3]_s
\)
with deck transformation
\[
(x,s)\mapsto(x+1/2,-s).
\]
The pullback of \(L\oplus\varepsilon^1\) to the local orientation double
cover \(\widetilde U\to U\) is trivial.  Thus the corresponding component of
the \(H\)-cover of \(Y_g=D(L\oplus\varepsilon^1)\) has a local product
neighbourhood
\[
   \mathcal B\cong S^1_x\times[-3,3]_s\times D^2_{(u,v)} .
\]
In this trivialization the nontrivial deck transformation acts on the
annulus by
\[
   (x,s)\longmapsto(x+1/2,-s),
\]
acts by \(-1\) on the pulled-back \(L\)-coordinate, and acts trivially on the
\(\varepsilon^1\)-coordinate.  Hence its local form is
\[
   \tau(x,s,u,v)=(x+1/2,-s,-u,v).
\]
The two relevant lifted surface sheets lie at \((r,0)\) and \((-r,0)\) in
the disk factor. The middle circle \(S^1_x\times\{0\}\subset\widetilde U\) maps two-to-one
onto the core \(\ell\subset U\), since the deck involution restricts to
\(x\mapsto x+1/2\) on this circle.  Thus, after orienting the lifted core
compatibly with \(\ell\), the lifted core represents
\(
   \eta^2\in H.
\)
With the opposite orientation it represents \(\eta^{-2}\).

Choose a transverse interval in the annulus
\(\widetilde U=S^1_x\times[-3,3]_s\), disjoint from the small region where
the local movie will be supported, and cut \(\widetilde U\) open along this
interval.  After removing a small collar of the cut, the remaining part is a
rectangular chart
\(
   \mathcal R\cong[0,1]_x\times[-3,3]_s .
\)
Together with the disk factor, this gives
\[
   \mathcal R\times D^2_{(u,v)}
   \cong
   [0,1]_x\times[-3,3]_s\times D^2_{(u,v)}.
\]
Choose two nearby \(s\)-levels in the cut-open annulus,
\[
   s_+=2\epsilon,
   \qquad
   s_-=-\epsilon,
   \qquad 0<\epsilon\ll1,
\]
so that \(s_++s_->0\).  These are the levels
\[
   [0,1]_x\times\{s_+\},
   \qquad
   [0,1]_x\times\{s_-\}
\]
in the rectangular annular factor.  The Whitney rectangle below will run
between these two \(s\)-levels while its other coordinate moves in the disk
factor. When the cut-open annulus is glued back to \(\widetilde U\), the two boundary
edges
\[
   \{0\}\times[-3,3]_s
   \quad\text{and}\quad
   \{1\}\times[-3,3]_s
\]
are identified.  Passing from the \(x=1\) edge to the \(x=0\) edge therefore
traverses the lifted core circle once.  With the orientation convention
chosen above, this loop represents \(\eta^{-2}\in H\).

Let
\[
   R:D^2_{(u,v)}\longrightarrow D^2_{(u,v)},
   \qquad
   R(u,v)=(-u,v),
\]
be the reflection induced by the deck transformation on the disk factor.
Choose an embedded arc
\(
   \beta:[0,1]\longrightarrow D^2_{(u,v)}
\)
from \((r,0)\) to \((-r,0)\), chosen so that
\[
   R(\beta(\vartheta))=\beta(1-\vartheta)
   \qquad\text{for all }\vartheta\in[0,1].
\]
For example, after shrinking \(r\), one may take the straight diameter
\(\beta(\vartheta)=((1-2\vartheta)r,0)\).  Let
\(
   \sigma:[0,1]\longrightarrow[-3,3]
\)
be the linear path from \(s_+\) to \(s_-\).  Define
\[
   W:[0,1]_x\times[0,1]_\vartheta
      \longrightarrow
      [0,1]_x\times[-3,3]_s\times D^2_{(u,v)}
\]
by
\(
   W(x,\vartheta)=(x,\sigma(\vartheta),\beta(\vartheta)).
\)
This is embedded, since the \(x\)-coordinate is retained and the path
\(\vartheta\mapsto(\sigma(\vartheta),\beta(\vartheta))\) is embedded.

We claim that \(W\) is disjoint from its \(\tau\)-translate.  Suppose
\(
   W(x,\vartheta)=\tau W(x',\vartheta').
\)
Equality in the disk factor gives
\(
   \beta(\vartheta)=R(\beta(\vartheta'))=\beta(1-\vartheta').
\)
Since \(\beta\) is embedded, this implies
\(
   \vartheta'=1-\vartheta.
\)
Equality in the \(s\)-coordinate gives
\(
   \sigma(\vartheta)=-\sigma(\vartheta')
   =
   -\sigma(1-\vartheta).
\)
But \(\sigma\) is linear from \(s_+\) to \(s_-\), so
\(
   \sigma(\vartheta)+\sigma(1-\vartheta)=s_++s_-.
\)
This is positive by the choice of \(s_+\) and \(s_-\).  Hence the displayed
equality is impossible.  Therefore
\(
   W\cap\tau W=\varnothing.
\)

The edge \(\vartheta=0\) lies on the lifted sheet with disk coordinate
\((r,0)\), and the edge \(\vartheta=1\) lies on the lifted sheet with disk
coordinate \((-r,0)\).  The side edges \(x=0\) and \(x=1\) are placed along
the two short guide arcs of the local finger move, inside the two sides of
the cut-open annular chart.  Thus the projection of \(W\) to the quotient is
an embedded Whitney disk pairing the two double points created by the local
finger move.  The product coordinates provide the Whitney framing.

We now choose an explicit cutoff representative of the corresponding
finger-plus-Whitney movie and compute its ordered deck fiber product in this
chart.  Put
\[
   A=S^1_x\times[-3,3]_s,
   \qquad
   J(x,s)=\left(x+\frac12,-s\right),
\]
so that, with the reflection \(R\) defined above,
\(\tau(a,z)=(Ja,Rz)\) on \(A\times D^2\).  Choose constants
\[
   0<\alpha<\frac14,
   \qquad
   0<\delta<\frac14,
   \qquad
   0<\varepsilon_0\ll r,
\]
and define
\[
   T(x)=\frac12+\alpha\cos(4\pi x),
   \qquad
   S(x)=\delta\sin(2\pi x).
\]
Let \(\chi:[-3,3]\to[0,1]\) be a smooth even function with
\(\chi=1\) on \([-1,1]\) and \(\chi=0\) for \(|s|\ge2\).  For
\(0\le\lambda\le1\), set
\[
\begin{aligned}
   \Delta_\lambda(x,s)
   &=(1-\chi(s))(2r,0)
     +\chi(s)\varepsilon_0
       \bigl(T(x)-\lambda,\ s-S(x)\bigr),\\
   n_\lambda(x,s)&=\frac12\Delta_\lambda(x,s),\\
   \widehat f_\lambda(x,s)&=(x,s,n_\lambda(x,s))
      \in A\times D^2_{(u,v)}.
\end{aligned}
\]
After decreasing \(\varepsilon_0\) and \(r\), all graphs lie in the disk
factor.  The identities
\[
   T\left(x+\frac12\right)=T(x),
   \qquad
   S\left(x+\frac12\right)=-S(x),
   \qquad
   \chi(-s)=\chi(s)
\]
give
\[
   \Delta_\lambda(Ja)=-R\Delta_\lambda(a),
   \qquad
   n_\lambda(a)-R n_\lambda(Ja)=\Delta_\lambda(a).
\]
The actual local track is obtained by concatenating a deck-free setup
isotopy from the initial preferred sheet
\(A\times\{(r,0)\}\) to the graph of \(n_0\), the graph family
\(\widehat f_\lambda\), and a final product collar at the graph of
\(n_1\).  The setup isotopy is given explicitly by the graphs of
\[
   n^{\operatorname{in}}_\mu
   =\frac12\Delta^{\operatorname{in}}_\mu,
   \qquad
   \Delta^{\operatorname{in}}_\mu
   =(1-\mu)(2r,0)+\mu\Delta_0,
   \qquad 0\le\mu\le1.
\]
A stationary interval at any regular level between the birth and death may
be inserted without changing the component calculation.

For the main movie interval, define the local ordered deck fiber product by
\[
   P^{\operatorname{loc}}_\tau
   =\{(a,b,\lambda)\in A\times A\times I:
      \widehat f_\lambda(a)=\tau\widehat f_\lambda(b)\}.
\]
Equality of the annular coordinates forces \(b=Ja\).  By the displayed
equivariance identities, equality of the normal coordinates is then equivalent to
\(\Delta_\lambda(a)=0\).  If \(\chi(s)=0\), then
\(\Delta_\lambda(x,s)=(2r,0)\ne0\).  If
\(1<|s|<2\) and \(\chi(s)>0\), then
\[
   |s-S(x)|\ge |s|-\delta>0,
\]
so the second coordinate of \(\Delta_\lambda\) is nonzero.  Thus no zero
occurs in the transition region.  In the core \(|s|\le1\),
\[
   \Delta_\lambda(x,s)
   =\varepsilon_0\bigl(T(x)-\lambda,\ s-S(x)\bigr),
\]
and therefore
\[
   \Delta_\lambda(x,s)=0
   \quad\Longleftrightarrow\quad
   s=S(x),\qquad \lambda=T(x).
\]
Consequently the complete ordered fiber product in the main movie is
\[
   Z_\tau=
   \left\{
   \left((x,S(x)),J(x,S(x)),T(x)\right):x\in S^1
   \right\}.
\]
It is one embedded circle.  Replacing \(x\) by \(x+\frac12\) exchanges the
two ordered points and remains on the same circle.

There are no fiber-product points in the setup isotopy.  In the core, the
first coordinate of \(\Delta^{\operatorname{in}}_\mu\) is
\[
   (1-\mu)2r+\mu\varepsilon_0T(x)>0;
\]
in the transition region either this first coordinate is \(2r\), or the
second coordinate is nonzero; and in the outer collar the vector is
\((2r,0)\).  The final graph is also deck-free: in the core
\(T(x)-1<0\), in the transition region the second coordinate is nonzero,
and in the outer collar the first coordinate is \(2r\).  The defining map
near \(Z_\tau\) is
\[
   (x,s,\lambda)\longmapsto
   \varepsilon_0\bigl(T(x)-\lambda,\ s-S(x)\bigr),
\]
whose derivatives in the \(\lambda\)- and \(s\)-directions have rank two.
Thus the ordered fiber product is a smooth one-manifold.  Smoothing the
concatenation junctions inside the deck-free setup and final collars creates
no additional component.

The graph movie also has the required finger and framed Whitney stages. The minima of \(T\) at
\(x=\frac14,\frac34\) are exchanged by \(J\), and its maxima at
\(x=0,\frac12\) are exchanged by \(J\).  Hence downstairs there is one
finger birth and one Whitney death.  Just before the death, put
\(\lambda_+=\frac12+\alpha\) and
\(\lambda=\lambda_+-\eta\), where \(0<\eta\ll1\), and let
\(\pm y_\eta\) be the two solutions near \(x=0\) of
\(T(y)=\lambda\).  The lifted Whitney disk is
\[
\begin{aligned}
   W^{\operatorname{graph}}_\eta:
   [-y_\eta,y_\eta]\times[0,1]&\longrightarrow\mathcal B,\\
   W^{\operatorname{graph}}_\eta(y,\vartheta)
   &=\left(
      y,S(y),
      (1-2\vartheta)\frac{\varepsilon_0}{2}
         (T(y)-\lambda),0
      \right).
\end{aligned}
\]
After the standard corner smoothing, this is an embedded disk.  Its deck
translate lies near \(x=\frac12\), while the disk lies near \(x=0\), so
the two are disjoint, and the product coordinates give the Whitney framing.
It has the same deck-disjoint framed Whitney geometry as the rectangle
constructed above, which records the quotient-chart placement shown in
Figure~\ref{fig:orientation-crossed-block-chart}.

The support is disjoint from \(D_0\), the basepoint, the base path, and all
marking data.  The graph family is the preferred lifted sheet, and its deck
partner is obtained by \(\tau\); hence every point of
\(Z_\tau\) has the unique nontrivial index-two double-coset label.  There
are no same-sheet components because every lifted sheet is a graph, and no
mixed component with the unchanged surface occurs by the choice of support.
Thus \(Z_\tau\) is the entire nontrivial labelled ordered double locus.  The
count is mod two, so no orientation sign is retained, while the product
framing ensures that the top surface is embedded.  Hence the track is a
crossed track for the unique nontrivial index-two label and
\(
   FQ^{0,\tau}_\rho(G)=u_\tau .
\)

Since Lemma~\ref{lem:orientation-survival} identifies the target with
\(
   Q^\tau_\rho(X_g,F_0)\cong \Ftwo\langle u_\tau\rangle,
\)
the absolute obstruction class is
\(
   FQ^\tau_\rho(F_0,F_1)=u_\tau\ne0.
\)
\end{proof}

\begin{lemma}
\label{lem:orientation-F1-complement}
For the surface \(F_1\) produced in
Lemma~\ref{lem:orientation-crossed-block},
\(
F_{1*}\pi_1\Sigma_g=H,
\)
the sphere \(D_0\) is a common framed embedded dual sphere for \(F_0\) and
\(F_1\), and
\(
\pi_1(X_g\setminus\nu F_1)\to\pi_1X_g
\)
is an isomorphism.
\end{lemma}

\begin{proof}
The track in Lemma~\ref{lem:orientation-crossed-block} is
\(\rho\)-marked, so the top surface has the same marked subgroup
\(F_{1*}\pi_1\Sigma_g=H\).  The construction is supported away from
\(D_0\), so \(D_0\) remains a framed embedded dual sphere for \(F_1\).
The punctured dual sphere \(D_0\setminus\operatorname{Int}\nu F_1\) is a disk
in \(X_g\setminus\nu F_1\) whose boundary is a meridian of \(F_1\).
Lemma~\ref{lem:dual-sphere-kills-meridian} gives the complement
\(\pi_1\)-isomorphism.
\end{proof}

\begin{lemma}
\label{lem:orientation-endpoints}
The orientation-double-cover family satisfies the endpoint-rigidity
hypotheses of Theorem~\ref{thm:D-image}.
\end{lemma}

\begin{proof}
We first verify the centralizer condition.  For \(g=1\), the group is the
Klein-bottle group
\[
   \pi=\langle a,t\mid tat^{-1}=a^{-1}\rangle,
   \qquad
   H=\langle a,t^2\rangle .
\]
Every element has normal form \(a^k t^\ell\).  Such an element centralizes
\(a\) only if \(\ell\) is even, and then it lies in
\(\langle a,t^2\rangle=H\).  Hence \(C_\pi(H)\subseteq H\). For \(g\ge2\), the centre of \(H\cong\pi_1\Sigma_g\) is trivial.  If
\(c\in C_\pi(H)\cap H\), then \(c=1\).  If \(c\in C_\pi(H)\setminus H\),
then \(c^2\in H\), since \(\pi/H\cong\mathbb Z/2\), and \(c^2\) still
centralizes \(H\).  Thus \(c^2\in Z(H)=1\), contradicting the
torsion-freeness of \(\pi_1(N_{g+1})\).  Therefore \(C_\pi(H)\subseteq H\)
also for \(g\ge2\).

Next we verify endpoint rigidity.  Since \(H\triangleleft\pi\), we have
\(N_\pi(H)=\pi\).  Elements of \(H\) act on \(H\) by inner automorphisms,
while any element of \(\pi\setminus H\) acts by the deck transformation
\(
   \delta:\Sigma_g\to\Sigma_g
\)
of the orientation double cover.  Hence
\[
   \operatorname{im}\bigl(N_\pi(H)\to\Out(H)\bigr)
   =
   \{1,[\delta_*]\}.
\]
The deck transformation \(\delta\) is orientation-reversing.  For \(g=1\),
this means that \(\delta_*\) has determinant \(-1\) on
\(H_1(T^2;\mathbb Z)\), while orientation-preserving mapping classes act with
determinant \(+1\).  For \(g\ge2\), the Dehn--Nielsen--Baer theorem
identifies orientation-preserving mapping classes with the index-two subgroup
of \(\Out(\pi_1\Sigma_g)\) preserving the orientation class
\cite[Chapter~8]{FarbMargalit2012}.  Thus no orientation-preserving mapping
class induces \([\delta_*]\).  Consequently the only orientation-preserving
mapping class whose outer action lies in
\(
   \operatorname{im}\bigl(N_\pi(H)\to\Out(H)\bigr)
\)
is the identity.

Finally,
\(
   [F_0]\cdot[D_0]=1
\)
implies \([F_0]\ne-[F_0]\) by
Corollary~\ref{cor:primitive-dual-exclusion}.  Hence all endpoint-rigidity
hypotheses of Theorem~\ref{thm:D-image} are satisfied.
\end{proof}

\begin{theorem}[Orientation-double-cover image nonconcordance]
\label{thm:G-orientation-family}
For every \(g\ge1\), the manifold
\[
X_g=D(L\oplus\varepsilon^1)\#(S^2\times S^2)
\]
contains homotopic embedded genus-\(g\) surfaces
\(
F_0,F_1:\Sigma_g\hookrightarrow X_g
\)
such that:
\begin{enumerate}[label=\textup{(\roman*)}]
\item
\(
\pi_1X_g\cong\pi_1(N_{g+1});
\)
\item
\(
F_{i*}\pi_1\Sigma_g=H\neq1
\qquad (i=0,1);
\)
\item
\(
\pi_1(X_g\setminus\nu F_i)\to\pi_1X_g
\)
is an isomorphism for \(i=0,1\);
\item \(F_0\) and \(F_1\) have a common framed embedded dual sphere;
\item the embedded images \(F_0(\Sigma_g)\) and \(F_1(\Sigma_g)\) are not
smoothly image-concordant.
\end{enumerate}
Equivalently, there is no diffeomorphism
\(
\phi:\Sigma_g\to\Sigma_g
\)
such that \(F_0\) is parametrized concordant to \(F_1\circ\phi\).
\end{theorem}

\begin{proof}
The construction of \(X_g\), \(F_0\), \(H\), and \(D_0\) was given at the
start of the section.  Lemma~\ref{lem:orientation-complement} gives the
complement \(\pi_1\)-isomorphism for \(F_0\).  Lemma~\ref{lem:orientation-survival}
gives
\(
Q^\tau_\rho(X_g,F_0)\cong\Ftwo\langle u_\tau\rangle.
\)
Lemma~\ref{lem:orientation-crossed-block} produces \(F_1\) with
\(
FQ^\tau_\rho(F_0,F_1)=u_\tau\neq0.
\)
Lemma~\ref{lem:orientation-F1-complement} gives the marked subgroup,
common dual sphere, and complement \(\pi_1\)-isomorphism for \(F_1\).
Lemma~\ref{lem:orientation-endpoints} verifies the endpoint-rigidity
hypotheses of Theorem~\ref{thm:D-image}.  Therefore
Theorem~\ref{thm:D-image} gives the stated image nonconcordance.
\end{proof}

\section{Questions and further directions}\label{sec:questions}
We close by recording some natural questions suggested by the construction.

\begin{question}\label{q:realization}
Which algebraically self-dual double-cosets admit \(D\)-crossed local
constructions without changing the ambient manifold?
\end{question}

M\"obius bands provide one geometric source of such labels.  This leads to
the fixed-ambient problem of determining when a complement contains an
isolated square-root neighbourhood, or another compactly supported
finger-plus-Whitney movie, whose \(P_D\)-fiber product has exactly one component.

\begin{question}\label{q:non-self-dual}
What is the correct analogue of the obstruction for a non-self-dual pair
\(D\neq D^{-1}\)?
\end{question}

The construction here uses a single-coordinate mod-two target for a
self-dual double-coset. For a non-self-dual label, the target should
remember both \(D\) and \(D^{-1}\), either by a paired-label construction or
by an oriented refinement of the count.  Such a refinement would also
require corresponding versions of the local insertion and mapping-space
indeterminacies.

\begin{question}\label{q:image-orbit}
Is there a computable image-level quotient of the marked obstruction beyond
the endpoint-rigid setting?
\end{question}

An image-level quotient should incorporate both the normalizer action on
\(H\backslash\pi/H\) and endpoint-monodromy contributions from tracks joining
\(F\) to \(F\circ\phi\).  Making this quotient explicit would extend the
endpoint-rigid applications considered here.

\begin{question}\label{q:vanishing}
Under what additional hypotheses can vanishing of the obstruction be used in
a classification theorem?
\end{question}

The obstruction detects nonconcordance when the resulting class is nonzero.
A natural next step is to study whether, in the presence of extra geometric
control such as a common dual sphere, \(\pi_1\)-controlled complements, or
restricted mapping-space indeterminacy, the vanishing case can be combined
with Whitney-move or light-bulb methods to classify homotopic
\(\pi_1\)-injective surfaces.

\end{document}